\documentclass[10pt,a4paper]{amsart}
\usepackage{verbatim}

\usepackage[toc]{appendix}
\usepackage[T1]{fontenc}

\usepackage{graphicx}
\usepackage{enumerate}
\usepackage{amsmath,amsfonts,amssymb}
\usepackage{color}
\usepackage{tabto}
\def\loc{\operatorname{loc}}
\usepackage{cite}
\usepackage{ latexsym }
\definecolor{citation}{rgb}{0.11,0.67,0.84}
\definecolor{formula}{rgb}{0.1,0.2,0.6}
\definecolor{url}{rgb}{0.11,0.67,0.84}
\usepackage{pgf,tikz}
\usepackage{mathrsfs}
\usepackage{fancyhdr}
\usepackage{dutchcal}

\usepackage{dsfont}

\newcommand{\reqnomode}{\tagsleft@false}

\usepackage{hyperref}

\newcommand\ccc{\mathrm{c}}

\def\dx{\,{\rm d}x}

\def\dtt{\,{\rm d}t}

\def\dt{\,{\rm d}t}
\def\dy{\,{\rm d}y}

\def \d{\,{\rm d}}
\def \diver{\,{\rm div}}
\def\dist{\,{\rm dist}}

\def\supp{\,{\rm supp}}
\def\diam{\,{\rm diam}}

\allowdisplaybreaks
\makeatletter
\DeclareRobustCommand*{\bfseries}{%
  \not@math@alphabet\bfseries\mathbf
  \fontseries\bfdefault\selectfont
  \boldmath
}

\DeclareMathOperator*{\osc}{osc}

\makeatother

\newlength{\defbaselineskip}
\newcommand{\setlinespacing}[1]
           {\setlength{\baselineskip}{#1 \defbaselineskip}}
\newcommand{\mint}{\mathop{\int\hskip -1,05em -\, \!\!\!}\nolimits}

\newtheorem{theorem}{Theorem}
\newtheorem{corollary}{Corollary}[section]
\newtheorem{definition}{Definition}
\newtheorem{remark}{Remark}[section]
\newtheorem{lemma}{Lemma}[section]
\newtheorem{proposition}{Proposition}[section]
\numberwithin{equation}{section}
\allowdisplaybreaks[4]

\newcommand{\kk}{\kappa}

\def\en{\mathbb N}
\def\rrr{\textnormal{r}}
\def\er{\mathbb R}

\newcommand{\ti}[1]{\tilde{#1}}

\newcommand\eps\varepsilon

\def\eqn#1$$#2$${\begin{equation}\label#1#2\end{equation}}

\newcommand{\data}{\textnormal{\texttt{data}}}
\newcommand{\be}{\begin{equation}}
\newcommand{\ee}{\end{equation}}

\newcommand{\rr}{\varrho}

\newcommand{\snr}[1]{\lvert #1\rvert}

\newcommand{\nr}[1]{\lVert #1 \rVert}

\newcommand{\tx}[1]{\textnormal{\texttt{#1}}}

\newcommand{\N}{\mathbb{N}}

\def\name[#1, #2]{#1 #2}

\newcommand{\medint}{-\kern -,375cm\int}

\newcommand{\medintinrigo}{-\kern -,315cm\int}
\makeatletter
\newcommand{\linethrough}{\mathpalette\@thickbar}
\newcommand{\@thickbar}[2]{{#1\mkern0mu\vbox{
    \sbox\z@{$#1#2\mkern-0.5mu$}%
    \dimen@=\dimexpr\ht\tw@-\ht\z@+2\p@\relax 
    \hrule\@height0.5\p@ 
    \vskip\dimen@
    \box\z@}}
}
\makeatother
\newcommand{\mathstrike}[1]{\ensuremath{\linethrough{#1}}}

\newcommand{\nra}[1]{\mathstrike{\lVert} #1 \rVert}
\newcommand{\snra}[1]{\mathstrike{[} #1 ]}
\title[Bounded minimizers of double phase problems]{Bounded minimizers of double phase problems\\ at nearly linear growth}
\author[De Filippis]{Cristiana De Filippis}  \address{Cristiana De Filippis\\Dipartimento SMFI, Universit\`a di Parma, Viale delle Scienze 53/a, Campus, 43124 Parma, Italy} \email{\url{cristiana.defilippis@unipr.it}}
\author[De Filippis]{Filomena De Filippis}  \address{Filomena De Filippis\\Fachbereich Mathematik, Universität Salzburg, Hellbrunner Str. 34, 5020 Salzburg, Austria} \email{\url{filomena.defilippis@plus.ac.at}}
\author[Piccinini]{Mirco Piccinini}  \address{Mirco Piccinini\\Dipartimento di Matematica, Politecnico di Milano, via Edoardo Bonardi 9, 20133 Milano, Italy} \email{\url{m.piccinini@polimi.it}}
\begin{document}
\subjclass[2020]{35B65, 31C45 \vspace{1mm}} 

\keywords{Regularity, $(p,q)$-growth, Nonuniform ellipticity, $\mu$-ellipticity\vspace{1mm}}

\thanks{{\it Acknowledgements.} We are grateful to Professor Mikhail Surnachev for pointing out several useful references on fractal constructions. We also thank the anonymous referees for their insightful feedback on our manuscript. This work is supported by the University of Parma through the action "Bando di Ateneo 2024 per la ricerca". C. De Filippis is supported by the European Research Council, through the ERC StG project NEW, nr.~101220121. F. De Filippis is supported by the Austrian Science Fund (FWF) [10.55776/PAT1850524]. M. Piccinini is supported by PRIN 2022 PNRR Project SKYRDM (CUP~D53D23018980001). For open access purposes, the author has applied a CC BY public copyright license to any author accepted manuscript version arising from this submission.
\vspace{1mm}}
\texttt{J. European Math. Soc.}
\maketitle
\vspace{-10mm}
\begin{abstract}
Bounded minimizers of double phase problems at nearly linear growth have locally H\"older continuous gradient within the sharp maximal nonuniformity range $q<1+\alpha$. 
 \end{abstract}
\begin{center}
   \emph{To Jan Kristensen, on his 60th birthday.}
 \end{center}
%
\setcounter{tocdepth}{1}
{\small \tableofcontents}
\vspace{-8mm}
\setlinespacing{1.00}
\section{Introduction}
\noindent In this paper we achieve sharp maximal regularity for minima of nonuniformly elliptic variational integrals at nearly linear growth. More precisely, we obtain Schauder estimates for a class of nondifferentiable problems featuring optimal degree of nonuniformity. The prototypical model we have in mind is given by the "log" double phase energy 
\eqn{ll}
$$
W^{1,1}_{\loc}(\Omega)\ni w\mapsto \mathcal{L}(w;\Omega):=\int_{\Omega}\snr{Dw}\log(1+\snr{Dw})+a(x)\snr{Dw}^{q}\dx,
$$
first treated within the realm of classical Schauder theory in \cite{dm}. The peculiarity of $\mathcal{L}$ lies in the structure of the governing integrand, that is the sum of a slightly superlinear component that grows at infinity slower than any power larger than one, and a $q$-power term weighted via a bounded, nonnegative coefficient $a$. The combined effect of the pointwise nonuniform ellipticity affecting the nearly linear growing term, with the uncontrollable vanishing of the coefficient makes these problems very delicate to handle under the regularity theory viewpoint. Specifically, we focus on achieving gradient continuity of minima under the fastest possible blow up rate of the ellipticity ratio related to these functionals. Counterexamples are built to exhibit severe irregularity phenomena that may arise once our threshold is violated. The first contribution in this respect is the following theorem. 
\begin{theorem}\label{mt1}
    
    Let $u\in W^{1,1}_{\loc}(\Omega)\cap L^{\infty}_{\loc}(\Omega)$ be a local minimizer of functional $\mathcal{L}$ in \eqref{ll}, with
\eqn{1qa}
$$
0\le a(\cdot)\in C^{0,\alpha}(\Omega), \ \ \alpha\in (0,1]\qquad \mbox{and}\qquad q<1+\alpha.
$$
Then $Du$ is locally H\"older continuous in $\Omega$, and whenever $\ti{B}\Subset B\subset  2B\Subset \Omega$ are balls with radius less than one, Lipschitz estimate
\eqn{lip.0}
$$
\nr{Du}_{L^{\infty}(\ti{B})}\le c\left(\int_{B}\snr{Du}\log(1+\snr{Du})+a(x)\snr{Du}^{q}\dx\right)^{\tx{b}}+c
$$
holds with $c\equiv c(\data(2B),\tx{d}(\ti{B},B))$ and $\tx{b}\equiv \tx{b}(n,q,\alpha)$.
\end{theorem}
\noindent Actually, Theorem \ref{mt1} comes as a direct consequence of an abstract result, covering free or constrained minima of more general integrals than \eqref{ll}, i.e.:
\eqn{ll1}
$$
W^{1,1}_{\loc}(\Omega)\ni w\mapsto \mathcal{H}(w;\Omega):=\int_{\Omega}\tx{L}(Dw)+a(x)(s^{2}+\snr{Dw}^{2})^{\frac{q}{2}}\dx,
$$
embracing for instance the iterated logarithmic case 
$$
\begin{cases}
\ \tx{L}(z):=\snr{z}\bar{\tx{L}}_{i+1}(\snr{z})\quad & \mbox{for} \ \ i\ge 0\\
\ \bar{\tx{L}}_{i+1}(\snr{z}):=\log(1+\bar{\tx{L}}_{i}(\snr{z})) \quad & \mbox{for} \ \ i\ge 0\\
\ \bar{\tx{L}}_{0}(z):=\snr{z},
\end{cases}
$$
cf. \cite{FM,dm}, meaning that the growth of integrand $\tx{L}$ may get arbitrarily close to linear, as that of the nonparametric area functional. Notably, nonautonomous area-type integrals mark the limiting case in which full gradient regularity still holds \cite{gms79}.\footnote{Observe that via the maximum principle \cite{ls05} the examples in \cite{gms79} always involve bounded minimizers.} More precisely, our main result reads as follows.
\begin{theorem}\label{mt2}
Under assumptions \eqref{1qa}, \eqref{ass.0}-\eqref{ass.2}, let $\psi\colon \Omega\to \mathbb{R}$ be a function satisfying \eqref{p2}, and $u\in \tx{K}^{\psi}_{\loc}(\Omega)\cap L^{\infty}_{\loc}(\Omega)$ be a constrained local minimizer of functional $\mathcal{H}$ in \eqref{ll1}. There exists $\mu_{\textnormal{max}}\equiv \mu_{\textnormal{max}}(n,q,\alpha)>1$ such that if $1\le \mu<\mu_{\textnormal{max}}$, then $Du$ is locally H\"older continuous in $\Omega$. In particular, whenever $\ti{B}\Subset B\subset 2B\Subset \Omega$ are balls with radius less than one, Lipschitz estimate
\eqn{fin.1}
    $$
    \nr{Du}_{L^{\infty}(\ti{B})}\le c\left(\int_{B}\tx{L}(Du)+a(x)(s^{2}+\snr{Du}^{2})^{\frac{q}{2}}\dx\right)^{\tx{b}}+c,
    $$
holds with $c\equiv c(\data(2B), \tx{d}(\ti{B},B))$ and $\tx{b}\equiv \tx{b}(n,\mu,q,\alpha)$.
\end{theorem}
\noindent We immediately refer to Sections \ref{not} and \ref{sa} for a complete description of the various quantities mentioned in the previous statement. In particular, the notion of constrained local minimizer adopted in this paper is classical and prescribes that a function $u\in \tx{K}^{\psi}_{\loc}(\Omega)\cap L^{\infty}_{\loc}(\Omega)$ is a local minimizer of $\mathcal{H}$ if for every ball $B\Subset \Omega$ it is $\mathcal{H}(u;B)<\infty$ and $\mathcal{H}(u;B)\le \mathcal{H}(w;B)$ whenever $w\in (u+W^{1,1}_{0}(B))\cap \tx{K}^{\psi}(B)$, thus implying that $\mathcal{H}(u;\ti{\Omega})<\infty$ for all open subsets $\ti{\Omega}\Subset \Omega$. Within the nonsingular regime $s>0$ in \eqref{ll1}, we are able to quantify the amount of gradient H\"older continuity gained by minima, and improve \cite[Corollary 1.2]{dm} in the plain logarithmic case.
\begin{corollary}\label{cor1}
In the same setting as Theorem \ref{mt2}, assume further that $\partial^{2}\tx{L}$ is continuous and that $s>0$ in \eqref{ll1}. Then,
\begin{itemize}
    \item if $\mu>1$, then $Du\in C^{0,\alpha/2}_{\loc}(\Omega,\mathbb{R}^{n})$;
    \item if $\mu=1$ and $\alpha\in (0,1)$, then $Du\in C^{0,\alpha}_{\loc}(\Omega,\mathbb{R}^{n})$;
    \item if $\mu=\alpha=1$, then $Du\in C^{0,\beta}_{\loc}(\Omega,\mathbb{R}^{n})$ for all $\beta\in (0,1)$.
\end{itemize}
\end{corollary}
\noindent Theorem \ref{mt2} and Corollary \ref{cor1} offer a complete regularity theory for quite a general family of obstacle problems that are genuinely nonuniformly elliptic and nondifferentiable, in the sense that the second variation may not even exist. This is a critical situation as all the available arguments, relying either on Choe's perturbation methods \cite{ch}, or on Duzaar \& Fuchs's linearisation technique \cite{dufu,du,fuc90}, here dramatically fail. Up to now, the latter was the standard approach to the regularity of differentiable nonuniformly elliptic obstacle problems, see \cite{FM,ss17,ciccio,ko1} and references therein, while in case of functionals defined upon pointwise uniformly elliptic structures, a perturbative argument is still feasible \cite{sch,KL,by0}. We refer to Section \ref{tc} below for a discussion on the most prominent technical novelties introduced here. Let us finally highlight that our results are sharp both in terms of the regularity theory for integrals \eqref{ll}-\eqref{ll1} and of related function space properties. 
\begin{theorem}\label{count}
Let $Q:=(-1,1)^{n}$, $0<\alpha<\infty$, $1<q<\infty$ be numbers such that 
\eqn{66.0}
$$q>1+\alpha,$$ 
$\mathcal{H}$ be the functional in \eqref{ll1}, where $s\in [0,1]$ and integrand $\tx{L}$ verifies \eqref{ass.0}-\eqref{ass.2} with $\mu\ge 1$, and $\mathcal{G}$ be the related modular in \eqref{ggg}. There exist a coefficient $0\le a(\cdot)\in C^{\alpha}(\bar{Q})$, and functions $\bar{u}_{0}\in W^{1,1}_{0}(Q)\cap L^{\infty}(\bar{Q})$ with $\mathcal{G}(\bar{u}_{0};Q)<\infty$, $\ti{u}_{0}\in C^{\infty}(\bar{Q})$ such that the following holds.
\begin{itemize}
\item Lavrentiev phenomenon occurs:
\eqn{count.2}
$$
\inf_{w\in \ti{u}_{0}+C^{\infty}_{c}(Q)}\mathcal{H}(w;Q)>\inf_{w\in \ti{u}_{0}+W^{1,1}_{0}(Q)}\mathcal{H}(w;Q).
$$
\item There is no sequence $\{u_{i}\}_{i\in \mathbb{N}}\subset C^{\infty}_{c}(Q)$ such that $\mathcal{G}(u_{i}-\bar{u}_{0};Q)\to 0$, i.e. smooth maps are not dense in $W^{1,1}_{0}(Q)\cap \left\{w\colon \mathcal{G}(w;Q)<\infty\right\}$ with respect to modular convergence.
\end{itemize}
\end{theorem}
\noindent Theorem \ref{count} shows the occurrence of Lavrentiev phenomenon and the failure of density of smooth maps in \eqref{ll}-\eqref{ll1}, thus also establishing the optimality of \cite[Theorem 2.3]{buli} and of our Lemma \ref{alav} and Corollary \ref{alav.2} below for this class of problems. Most prominently, we exhibit severe anomalies in gradient integrability for minima of more general double phase functionals than \eqref{ll}-\eqref{ll1}, including those whose elliptic term has linear growth. This is the content of the next theorem.
\begin{theorem}\label{count.3}
    With $Q:=(-1,1)^{n}$, $0<\alpha<\infty$ and $1<q<\infty$ being numbers verifying \eqref{66.0}, let $\mathcal{H}$ be the functional in \eqref{ll1}. For every $\varepsilon\in (0,\min\{q-1-\alpha,n-1\})$ there exist a coefficient $0\le a(\cdot)\in C^{\alpha}(\overline{2Q})$ and a function $\ti{u}_{0}\in C^{\infty}(\overline{2Q})$ such that the following holds.
    \begin{itemize}
        \item Under nearly linear growth conditions \eqref{ass.0}-\eqref{ass.2}, the solution $u\in \ti{u}_{0}+W^{1,1}_{0}(Q)$ to Dirichlet problem
\eqn{count.1}
$$
\ti{u}_{0}+W^{1,1}_{0}(Q)\ni w\mapsto \min_{w\in \ti{u}_{0}+W^{1,1}_{0}(Q)}\mathcal{H}(w;Q)
$$
does not belong to $W^{1,p}_{\loc}(Q)$ for all $p>1+\varepsilon$.
\item Assuming \eqref{ass.0} and linear growth conditions \eqref{gcgc}, every relaxed\footnote{Due to lack of compactness of $W^{1,1}(Q)$, if $\tx{g}=1$ in \eqref{ass.1}, problem \eqref{count.1} might not admit solutions in Dirichlet class $\ti{u}_{0}+W^{1,1}_{0}(Q)$. Existence results can rather be obtained in $BV(Q)$ by suitably extending integral \eqref{ll1} from $\ti{u}_{0}+W^{1,1}_{0}(Q)$ to $BV(Q)$, see \cite{gms79a,BS1,gkpq} and Section \ref{lglg} below.} minimizer of functional $\mathcal{H}$ cannot belong to $W^{1,p}_{\loc}(Q)$ for all $p>1+\varepsilon$.
    \end{itemize}
\end{theorem}
\noindent Examples covered by Theorem \ref{count.3} include integrals \eqref{ll}-\eqref{ll1}, as well as anisotropic integrals of the area type such as
\eqn{area}
$$
W^{1,1}_{\loc}(\Omega)\ni w\mapsto \mathcal{A}(w;\Omega):=\int_{\Omega}\sqrt[{p\ }]{1+\snr{Dw}^{p}}+a(x)\snr{Dw}^{q}\dx,\quad q,p>1,
$$
 cf. \cite{BS1,besc} and Section \ref{tc}. Although being stated for unconstrained problems, both Theorems \ref{count}-\ref{count.3} can be extended effortlessly to the variational obstacle setting, given that boundary data and competitors are bounded. The principle behind Theorems \ref{count}-\ref{count.3} is that the vanishing of coefficient $a$ hides the singularities of certain low-integrable competitors whose energy becomes therefore finite, thus making them admissible in \eqref{count.2} and \eqref{count.1}. Such quite irregular maps lessen too much the energy, so that, imposing very high traces on $\partial Q$ for the smooth boundary data in \eqref{count.2}-\eqref{count.1}, leads to the formation of singularities. The strength of these examples lies in three key aspects.
 \begin{itemize}
     \item They hold for scalar minima. This is a genuinely nonuniformly elliptic phenomenon, in contrast with the uniformly elliptic setting where, to produce singularities in the plain energy setting, one needs to look at vectorial problems.
     \item The integrands considered are regular. Specifically, they can be nondegenerate or nonsingular, radial, and the coefficient $a$ can be chosen with an arbitrarily high smoothness degree $\alpha>0$.
     \item Lipschitz domains and smooth boundary data. Singularities are not generated by irregular boundaries or ill-behaved boundary data, they actually originate from the pathological interaction between coefficient and gradient variable.
 \end{itemize}
\noindent Theorems \ref{count}-\ref{count.3} come by revisiting the fractal constructions of Balci \& Diening \& Surnachev in \cite{balci,balci2}, see also Zhikov \cite{Zhom}, Fonseca \& Mal\'y \& Mingione \cite{FMM}, and Esposito \& Leonetti \& Mingione \cite{sharp} for earlier results, and Remark \ref{rem6} below for additional comments. An overview of our results appears in \cite{def25}. The a priori boundedness ansatz in nonuniformly elliptic problems plays a key role in enlarging or maximising the rate of nonuniformity under which gradient continuity holds. This phenomenon already appears in the regularity theory of nonparametric minimal surfaces with the seminal work of Bombieri \& De Giorgi \& Miranda \cite{bdm69}, or more general nonuniformly elliptic equations, cf. Trudinger \cite{tru67}, Ladyzhenskaya \& Ural'tseva \cite{lu70}, Simon \cite{simon} and Giaquinta \& Modica \& Souček \cite{gms79}. These fundamental contributions initiated a systematic investigation into various classes of nonuniformly elliptic PDEs, see Ivanov's monograph \cite{iva84} and references therein for an account of the first advances on the subject. Most of the methods and techniques developed in this period strictly rely on barrier arguments, or on the availability of strong solutions as well as on the possibility of total-differentiating the problem, thus unavoidably escaping the classical Schauder setting. Later on, Marcellini's fundamental variational approach \cite{ma2,ma1,ma4} set, within the natural energy framework, quite a large family of PDEs with polynomial degree of nonuniformity and obtained maximal gradient regularity for autonomous or differentiable problems. In a nutshell, the basic idea consists in relating the growth/ellipticity features of the equation to the rate of blow up of its ellipticity ratio that, if suitably slowed down by choosing growth/ellipticity exponents sufficiently close to each other, assures regularity. This turns out to be necessary and sufficient condition for regular solutions \cite{ma1,ma4}. A very rich literature consequently flourished, cf. \cite{ma5,ba2,BM,BS,hs,Schaeffner2020,BS2,dkk24,fsv24,gkpq,schag} and \cite{bdms,ko1,dp,ik} for interior and boundary regularity respectively in the autonomous case, \cite{choe,elm99,ckp11} for related interpolative results, and \cite{dmon,dima,ciccio,bar23} for nonautonomous, differentiable problems, see also \cite{masu2,def25} for reasonable surveys. Of particular relevance for our ends is the so-called anisotropic $(\mu,q)$-ellipticity \cite{bildhauer}, typical of functionals at linear or nearly linear growth as those in \eqref{area} or in \eqref{ll}-\eqref{ll1}. Being the latter limiting configurations between linear and power growth, these models are rather common in materials science: the theories of Prandtl-Eyring fluids and of plastic materials with logarithmic hardening are prominent instances of applications, cf. Frehse \& Seregin \cite{FrS} - see also Bildhauer \& Fuchs \cite{bildhauer,BF}, Beck \& Schmidt \cite{BS1,besc,bs15b}, Fuchs \& Mingione \cite{FM}, and Gmeineder \& Kristensen \cite{gmek1,gme2,gme1,gkpq} for related deep regularity results, both in the scalar and in the vectorial setting. Regarding nonautonomous, nonuniformly elliptic integrals, Giaquinta \& Modica \& Souček \cite{gms79} and Zhikov \cite{Zhom,Z0,Z1} highlighted the distinctive, unique aspect that is the pathological interaction between space-depending coefficients and gradient variable, possibly leading to the formation of singularities already in the scalar setting. A very efficient testbed of this novel phenomenology is Zhikov's double phase functional,
\eqn{dpdp}
$$
\begin{array}{c}
\displaystyle
W^{1,1}_{\loc}(\Omega)\ni w\mapsto \int_{\Omega}\snr{Dw}^{p}+a(x)\snr{Dw}^{q}\dx,\\[8pt]\displaystyle
0\le a(\cdot)\in C^{0,\alpha}(\Omega), \ \ \alpha\in (0,1],\qquad \quad 1<p\le q<\infty,
\end{array}
$$
a very special example of the class of problems studied by Marcellini, that offers striking counterexamples once certain precise quantitative relations linking growth exponents $(p,q)$, the smoothness degree $\alpha$ of the modulating coefficient $a$ and possibly the ambient dimension $n$, are violated. In fact, Esposito \& Leonetti \& Mingione \cite{sharp} exhibited a minimizer of integral \eqref{dpdp} with one point singularity that forbids $W^{1,q}$-integrability whenever $q>p(n+\alpha)/n$, while Fonseca \& Mal\'y \& Mingione \cite{FMM} and Balci \& Diening \& Surnachev \cite{balci,balci2} constructed bounded minimizers of \eqref{dpdp} with fractal singular set of maximal Hausdorff dimension whenever $q>p+\alpha$. A comprehensive regularity theory for minima of \eqref{dpdp} was eventually obtained by Baroni \& Colombo \& Mingione \cite{CM,CM2,BCM}: if $u\in W^{1,p}_{\loc}(\Omega)$ is a local minimizer of \eqref{dpdp}, then
\eqn{dpr}
$$
\left\{
\begin{array}{c}
\displaystyle
\ u\in W^{1,p}_{\loc}(\Omega) \ \ \mbox{and} \ \ \frac{q}{p}\le 1+\frac{\alpha}{n}  \ \Longrightarrow \ Du \ \ \mbox{locally H\"older continuous,}\\[10pt]\displaystyle
\ u\in W^{1,p}_{\loc}(\Omega)\cap L^{\infty}_{\loc}(\Omega) \ \ \mbox{and} \ \ q\le p+\alpha \ \Longrightarrow \ Du \ \ \mbox{locally H\"older continuous.}
\end{array}
\right.
$$
In \cite{CM,CM2,BCM} a purely perturbative approach is adopted, and this is feasible because the integrand in \eqref{dpdp} is pointwise uniformly elliptic, the vanishing of coefficient $a$ being the only responsible of the mild nonuniformity of \eqref{dpdp}. Such a behavior can be quantified by means of the pointwise ellipticity ratio associated to the double phase integrand $\tx{P}(x,z):=\snr{z}^{p}+a(x)\snr{z}^{q}$,
$$
\mathcal{R}_{\partial\tx{P}}(x,z):=\frac{\mbox{highest eigenvalue of} \ \partial^{2}\tx{P}(x,z)}{\mbox{lowest eigenvalue of} \ \partial^{2}\tx{P}(x,z)}\lesssim_{p,q}1
$$
that stays uniformly bounded, while the nonlocal one\footnote{Here, $B\Subset \Omega$ is a ball.}
$$
\mathcal{R}_{\partial\tx{P}}(z;B):=\frac{\sup_{x\in B}\textnormal{highest eigenvalue of }\partial^{2} \tx{P}(x,z)}{\inf_{x\in B}\textnormal{lowest eigenvalue of }\partial^{2} \tx{P}(x,z)}\lesssim_{p,q}1+\nr{a}_{L^{\infty}(B)}\snr{z}^{q-p},
$$
instead detects the (very soft) rate of nonuniformity of \eqref{dpdp}, that can nonetheless be compensated via zero-order corrections (Gehring-type higher integrability or H\"older continuity à la De Giorgi-Nash-Moser), see \cite[Section 4.6]{ciccio}. The whole strategy was eventually extended to a very general family of pointwise uniformly elliptic problems by H\"asto \& Ok \cite{HO1,HO2,HO3}, see also \cite{bacomi1,boh,bs21,balci3,KL,bal23,bar23,bb90,bar24} for an (incomplete) list of related contributions. This approach breaks down for the integrals treated here already in the basic model \eqref{ll}: letting $\tx{H}(x,z):=\snr{z}\log(1+\snr{z})+a(x)\snr{z}^{q}$ we indeed find
$$
\mathcal{R}_{\partial\tx{H}}(x,z)\lesssim_{q}1+\log(1+\snr{z})\qquad \mbox{and}\qquad \mathcal{R}_{\partial\tx{H}}(z;B)\lesssim_{q}1+\nr{a}_{L^{\infty}(B)}\snr{z}^{q-1},
$$
so \eqref{ll} and a fortiori \eqref{ll1}, fall in the realm of strongly nonuniformly elliptic functionals, characterised by the blow up of the pointwise ellipticity ratio:\footnote{By very definition it is $\mathcal{R}_{\partial\tx{H}}(x,z)\le \mathcal{R}_{\partial\tx{H}}(z;B)$ for all $x\in B$, so the blow up of the pointwise ellipticity ratio implies the blow up of the nonlocal one on the same ball.} nonuniformity now is directly due to outgrows or anisotropy in the gradient variable. A complete Schauder theory was only recently achieved by the first named author and Mingione, see \cite{piovra,jw} for quantitatively superlinear problems, and \cite{dm} for nearly linear growth ones, the multi-phase model being instead treated by the second and third named authors in \cite{dp2}. Integrals \eqref{ll}-\eqref{ll1} can be seen as limiting configurations of \eqref{dpdp} as $p\to 1$. Coherently, in \cite{dm} it is proven that local minimizers of \eqref{ll} or \eqref{ll1} have locally H\"older continuous gradient provided that $q<1+\alpha/n$, that is $\eqref{dpr}_{1}$ with (formally) $p=1$, while here we cover the full interpolative range in \eqref{1qa}, precisely corresponding to $p=1$ in $\eqref{dpr}_{2}$. We conclude by outlining our techniques. 
\subsection{Technical novelties}\label{tc} In this paper there are three salient technical novelties that are worth highlighting. All of them are specifically designed to deliver optimal Schauder theory for free or constrained minimizers of anisotropic variational integrals at nearly linear growth within the maximal rate of nonuniformity. 
\begin{itemize}
    \item The achievement of gradient H\"older continuity for minimizers of functionals $\mathcal{L}$ or $\mathcal{H}$ within the maximal nonuniformity rate $q<1+\alpha$. So far, the only way to cover the full range in \eqref{dpr}$_{2}$ in quantitatively superlinear, pointwise uniformly elliptic problems \cite{CM2,BCM,bb90,HO3} makes crucial use of the pointwise uniform ellipticity of the governing integrand. In fact, this still allows proving almost Lipschitz continuity for minima via intrinsic perturbation arguments such as blow up or harmonic approximation, thus reducing to a minimal level all possible growth/ellipticity interactions that may worsen the bound in $\eqref{dpr}_{2}$.~Once shown H\"older continuity up to any exponent (less than one), the nonuniformity of the functional becomes immaterial, and gradient H\"olderianity follows via Campanato's theory. Due to the strong rate of nonuniformity of functionals $\mathcal{L}$ or $\mathcal{H}$, any approach via perturbation immediately breaks down, and the only effective way to obtain almost Lipschitz continuity is to "fractionally" differentiate the related Euler-Lagrange equations as to traduce the H\"older continuity of coefficients into a gain of fractional differentiability for solutions as done in \cite{jw}. However, another obstruction arises because of the severe loss of ellipticity in integrals at nearly linear growth. Indeed, as already evident in quantitatively superlinear problems with subquadratic growth, a loss of half of the fractional differentiability gain occurs and unavoidably reduces the attainable range of nonuniformity, see \cite[Section 3.6, Step 5]{CM2}, \cite[(1.19), Theorem 3]{dmon}, \cite[(7.9), Section 7.1]{jw}, and Remark \ref{r32} below. To overcome this issue we employ a hybrid, borderline counterpart of the fractional Moser iteration in \cite{jw} that preserves the full fractional differentiability income, and consequently maximises the nonuniformity range. The price to pay is the appearance in the bounding constants of an arbitrarily small power of the $L^{\infty}$-norm of the gradient to compensate the loss of ellipticity and interpolate between an arbitrarily high power of the modulus of the gradient and its $L^{1}$-norm. The $L^{\infty}$-norm is eventually reabsorbed in the final De Giorgi type iteration, see the proof of Propositions \ref{p31} and \ref{lipb} below. Let us point out that this is the very first set of optimal Schauder estimates in nondifferentiable, anisotropic $(\mu,q)$-elliptic problems with the sharp a-dimensional constraint \eqref{1qa} in force. We shall return on this point in the next bullet. Arguments relying on fractional differentiability and various general Besov spaces techniques are employed in uniformly elliptic problems with a certain lack of ellipticity, see Domokos \cite{dom04} on gradient higher differentiability of $p$-harmonic maps in the Heisenberg group, and Brasco \& Lindgren \& Schikorra \cite{bls} and Garain \& Lindgren \cite{gl1} for higher regularity in $(s,p)$-harmonic maps.
    \item $\mu$-ellipticity, nondifferentiable ingredients and convex anisotropy. To prove gradient continuity in $\mu$-elliptic problems without penalising the rate of nonuniformity (in particular without imposing dimensional limitations), heavy smoothness assumptions were imposed on space depending coefficients, cf. \cite[Part I]{lu70}, \cite[page 161]{gms79}, \cite[Chapter 4.2.2.2]{bildhauer} and \cite[Appendix C]{BS1}, due to the unavoidable use of classical Moser iteration technique in combination with Gagliardo-Nirenberg inequalities. Although at least in the case of $(\mu,q)$-elliptic integrals at nearly linear growth these assumptions have been relaxed to Sobolev differentiability \cite[Theorems 1 and 4]{dmon}, and eventually to H\"older continuity \cite[Theorems 1.1 and 1.3]{dm} and \cite[Theorem 1.1 and Corollary 1.2]{dp2}, a dimensional nonuniformity rate of the form $q<1+\texttt{o}(n)$, with $\texttt{o}(n)\to 0$ as $n\to \infty$, plays a crucial role. Here we bypass this obstruction and obtain Schauder estimates for a reasonable class of nonautonomous, anisotropic $(\mu,q)$-elliptic integrals. The sharpness of our achievements is established via fractal counterexamples in the spirit of \cite{balci,balci2}, that connect in an optimal fashion our regularity results to the limiting nonuniformity threshold past which severe irregularity phenomena occur, cf. Theorems \ref{count}-\ref{count.3}. In particular, Theorem \ref{count.3} is stated for more general functionals than \eqref{ll}-\eqref{ll1}, for which we achieve maximal regularity. It indeed holds also for integrals whose elliptic term has linear growth, such as the area-type model in \eqref{area}. Although Theorem \ref{count.3} does not rule out $\textnormal{L-LogL}$ estimates \cite[Theorem 1.10]{BS1}, it shows that pointwise gradient bounds fail whenever $q>1+\alpha$. This obstruction depends only on the growth of the integrand, not on its ellipticity (i.e., the growth of the eigenvalues of its second derivatives). By contrast, a priori estimates hinge on conditions on second derivatives. In superlinear regimes these scale with the integrand’s growth (see \eqref{dpdp} or the $(p,q)$-nonuniformity case \cite{piovra,jw}), but approaching the linear regime a gap may occur: the integrand keeps linear growth while its second derivatives can decay rapidly (cf.~\eqref{area} with large $p>1$). In view of Theorem \ref{mt2}, a formal threshold for a priori gradient regularity is
    \eqn{const}
    $$1<q<2-\mu+\alpha\le 1+\alpha.$$
    For anisotropic area-type functionals as in \eqref{area} one has $\mu=p+1$, \cite[Section 1]{besc}. Coupling this with $p>1$ gives $\alpha>q+p-1$; for $p$ close to $2$ this recovers the classical requirement of $C^{2}$-coefficients, see \cite{lu70} and aligns with the counterexample in \cite{gms79}. If, in addition, $\alpha\in (0,1)$ and we recall that necessarily $q<1+\alpha$, then the condition above forces $\mu\approx 1$ exactly as in Theorems \ref{mt1}–\ref{mt2} and in \cite{dm}. This suggests that \eqref{ll}, \eqref{ll1} are the limiting configurations for Schauder theory with H\"older coefficients, convex anisotropy, and $\mu$-ellipticity.
    \item Nonuniform ellipticity, nondifferentiability, and the obstacle constraint. As already mentioned, we show how to bypass the up-to-now unavoidable linearisation procedure due to Duzaar \& Fuchs \cite{dufu,du} to achieve first order regularity in obstacle problems driven by nondifferentiable, nonuniformly elliptic operators. Duzaar \& Fuchs's technique turns constrained minimizers of homogeneous functionals into unconstrained minima of forced integrals with right-hand side term depending in a nonlinear fashion on the second derivatives of the obstacle function and on the gradient of coefficients. This is no longer possible in our case because of the lack of differentiability of coefficient $a$. Moreover, the strong rate of nonuniform ellipticity of the integrands in \eqref{ll} or \eqref{ll1} forbids to directly perturb around frozen problems, cf. Choe \cite{ch}. In this respect, we rather design a blended scheme relying on both differentiation (at fractional scales) and hybrid comparison aimed at homogenizing the H\"older continuity of coefficients with the rate of fractional differentiability of minima and tailored to account also for the presence of the obstacle. 
\end{itemize}
Finally, a synopsis of the structure of the paper.
\subsubsection*{Outline of the paper} In Section \ref{pre} we describe our notation and collect some auxiliary results that will be helpful at various stages of the paper. In Section \ref{alavcc} we discuss basic approximation features of integrals \eqref{ll}-\eqref{ll1}. In Section \ref{ncz} we design our hybrid version of fractional Moser iteration. Section \ref{sce} is instead devoted to the achievement of Schauder estimates for (possibly constrained) minima of \eqref{ll} or \eqref{ll1}. Section \ref{app} contains the proof of Theorems \ref{mt1}-\ref{mt2} and of Corollary \ref{cor1}. Finally, Section \ref{sec6} offers counterexamples to establish the sharpness of our results.

\section{Preliminaries}\label{pre}
\subsection{Notation}\label{not}
In the following, $\Omega\subset \er^n$, $n\ge 2$, denotes an open, bounded domain with Lipschitz regular boundary. We denote by $c$ a general constant larger than $1$. Different occurrences from line to line will be still denoted by $c$. Special occurrences will be denoted by $c_*,  \tilde c$ or likewise. Relevant dependencies on parameters will be as usual emphasized by putting them in parentheses. Sometimes we shall use symbols "$\gtrsim$", "$\lesssim$" with subscripts, to indicate that a certain inequality holds up to constants whose dependencies are marked in the subfix. We denote by $ B_r(x_0):= \{x \in \er^n  :   |x-x_0|< r\}$ the open ball with center $x_0$ and radius $r>0$; we omit denoting the center when it is not necessary, i.e., $B \equiv B_r \equiv B_r(x_0)$; this especially happens when various balls in the same context share the same center. Finally, with $B$ being a given ball with radius $r$ and $\gamma$ being a positive number, we denote by $\gamma B$ the concentric ball with radius $\gamma r$ and, consequently, it is $B/\gamma \equiv (1/\gamma)B$. To indicate a general function space (i.e. a H\"older space, or a Sobolev space) we use symbol $\mathbb{X}(\Omega)$, while $\mathbb{X}_{\loc}(\Omega)$ will denote its local variant. We shall adopt such a symbol to stress that the degree of smoothness of certain functions depends only on the assumed regularity on a fixed one. With reference to \eqref{ll1}, for the rest of the paper we keep the following notation:\footnote{With some abuse we will keep the same symbol $\ell_{s}(z)$ also when $z$ is a scalar.}
\eqn{hl.00}
$$
\ell_{s}(z):=(s^{2}+\snr{z}^{2})^{\frac{1}{2}}\qquad \mbox{and}\qquad \tx{H}(x,z):=\tx{L}(z)+a(x)\ell_{s}(z)^{q},
$$
for $z\in \mathbb{R}^{n}$, $1<q<\infty$ and $s\in [0,1]$. Whenever $\ti{\Omega} \subset \er^{n}$ is a measurable subset with bounded positive measure $0<|\ti{\Omega}|<\infty$, and with $f \colon \ti{\Omega} \to \er^{k}$, $k\geq 1$, being a measurable map, we use
$$
(f)_{\ti{\Omega}}=\mint_{\ti{\Omega}}f(x)\dx:= \snr{\ti{\Omega}}^{-1}\int_{\ti{\Omega}}  f(x) \dx
$$
to indicate its integral average. If $f\in L^{\gamma}(\ti{\Omega},\mathbb{R}^{k})$ for some $1\le \gamma<\infty$, we shorten its averaged norm as
$$
\nra{f}_{L^{\gamma}(\ti{\Omega})}:=\left(\mint_{\ti{\Omega}}\snr{f}^{\gamma}\dx\right)^{\frac{1}{\gamma}},
$$
while if $f\in W^{s,\gamma}(\ti{\Omega})$ with $1\le \gamma<\infty$ and $s\in (0,1)$, its averaged Sobolev–Slobodecki\v{i} seminorm will be indicated by 
$$
\snra{f}_{s,\gamma;\ti{\Omega}}:=\left(\mint_{\ti{\Omega}}\int_{\ti{\Omega}}\frac{\snr{f(x)-f(y)}^{\gamma}}{\snr{x-y}^{n+s\gamma}}\dx\dy\right)^{\frac{1}{\gamma}}.
$$
Given any open set $\ti{\Omega}\Subset \Omega$, to simplify the notation we collect the main parameters related to the problems under investigation in the shorthands
\begin{flalign*}
\begin{cases}
\ \data_{0}:=\left(n,\Lambda,\tx{g},\mu,q\right),\qquad \data_{*}:=\left(n,\Lambda_{*},\tx{g},\mu,q\right),\vspace{1mm}\\
\ \tx{d}(\ti{\Omega},\Omega):=(\dist(\ti{\Omega},\partial\Omega),\diam(\ti{\Omega}),\diam(\Omega)),\vspace{1mm}\\
\ \data_{*}(\ti{\Omega}):=\left(n,\Lambda_{*},\tx{g},\mu,q,\alpha,\nr{\psi}_{W^{2,\infty}(\ti{\Omega})},\nr{\snr{D\psi}^{q-2}D^{2}\psi}_{L^{\infty}(\ti{\Omega})},\nr{a}_{C^{0,\alpha}(\ti{\Omega})}\right),\vspace{1mm}\\
\ \data(\ti{\Omega}):=\left(n,\Lambda,\tx{g},\mu,q,\alpha,\nr{\psi}_{W^{2,\infty}(\ti{\Omega})},\nr{u}_{L^{\infty}(\ti{\Omega})},\nr{\snr{D\psi}^{q-2}D^{2}\psi}_{L^{\infty}(\ti{\Omega})},\nr{a}_{C^{0,\alpha}(\ti{\Omega})}\right).
\end{cases}
\end{flalign*}
With a slight abuse of notation, the symbol $\texttt{o}(\varepsilon)$ will only indicate a quantity depending
on $\varepsilon$ and vanishing as $\varepsilon\to 0$, with no claim on
its rate relative to $\varepsilon$. Likewise,
$d_{\delta;\varepsilon}=\texttt{o}_{\varepsilon}(\delta)$ means that
$d_{\delta;\varepsilon}\to0$ as $\delta\to 0$, for every fixed
$\varepsilon\in(0,1]$. We finally refer to Section \ref{sa} for a description of the various quantities appearing above.
\subsection{Fractional Sobolev spaces} Here we recall some basic facts about Sobolev functions. For a map $w \colon \Omega \to \mathbb{R}^{k}$, $k\ge 1$, a number $\gamma>0$ and a vector $h \in \mathbb{R}^n$, we set $\Omega_{\gamma h}:=\left\{x\in \Omega\colon \dist(x,\partial \Omega)>\gamma\snr{h}\right\}$, and introduce operators $\tau_{h}\colon L^{1}(\Omega,\mathbb{R}^{k})\to L^{1}(\Omega_{h},\mathbb{R}^{k})$, $\tau^{2}_{h}\colon L^{1}(\Omega,\mathbb{R}^{k})\to L^{1}(\Omega_{2h},\mathbb{R}^{k})$, pointwise defined as
\begin{flalign*}
\tau_{h}w(x):=w(x+h)-w(x)\qquad \mbox{and}\qquad \tau_{h}^{2}w(x):=\tau_{h}(\tau_{h}w)(x)\equiv w(x+2h)-2w(x+h)+w(x).
\end{flalign*}
Given two measurable functions $v,w\colon \mathbb{R}^{n}\to \mathbb{R}$, the discrete Leibniz rule prescribes that:
\eqn{prod}
$$
(\tau_{h}(vw))=w(\cdot+h)\tau_{h}v+v\tau_{h}w.
$$
Finite difference operators and weak differentiability of functions are closely connected. We now record a few basic facts concerning fractional Sobolev spaces, \cite{guide,bls}. 
\begin{definition}\label{fra1def}
Let $p \in [1, \infty)$, $k \in \en$, $n \geq 2$, and $\Omega \subset \er^n$ be an open subset.
\begin{itemize}
\item Let $\alpha_{0}\in (0,1)$. The fractional Sobolev space $W^{\alpha_{0} ,p}(\Omega,\er^k )$ consists of those maps $w \colon \Omega\to \er^k$ such that 
the following Gagliardo type norm is finite:
\begin{eqnarray}
\notag
\| w \|_{W^{\alpha_{0} ,p}(\Omega)} & := &\|w\|_{L^p(\Omega)}+ \left(\int_{\Omega} \int_{\Omega}  
\frac{|w(x)
- w(y) |^{p}}{|x-y|^{n+\alpha_{0} p}} \dx \dy \right)^{1/p}\\
&=:& \|w\|_{L^p(\Omega)} + [w]_{\alpha_{0}, p;\Omega}.\label{gaglia}
\end{eqnarray}
The local variant $W^{\alpha_{0} ,p}_{\loc}(\Omega,\mathbb{R}^{k})$ is defined by requiring that $w \in W^{\alpha_{0} ,p}_{\loc}(\Omega,\mathbb{R}^{k})$ iff $w \in W^{\alpha_{0} ,p}(\tilde{\Omega},\mathbb{R}^{k})$ for every open subset $\tilde{\Omega} \Subset \Omega$. 
\item Let $\alpha_{0}\in (0,1)$. The Nikol'skii space $N^{\alpha_{0},p}(\Omega,\er^k )$ is defined by prescribing that $w \in N^{\alpha_{0},p}(\Omega,\er^k )$ iff
$$\| w \|_{N^{\alpha_{0},p}(\Omega )} :=\|w\|_{L^p(\Omega)} + \left(\sup_{|h|\not=0}\, \int_{\Omega_{h}} \left|\frac{\tau_{h}w}{\snr{h}^{\alpha_{0}}}\right|^{p}
 \dx  \right)^{1/p}<\infty.$$
The local variant $N^{\alpha_{0},p}_{\loc}(\Omega,\er^k )$ is defined analogously to $W^{\alpha_{0} ,p}_{\loc}(\Omega,\er^k )$.
\item With $\alpha_{0}\in (0,2)$, we say that a function $w\colon \Omega\to \mathbb{R}^{k}$ belongs to the Besov-Nikol'skii space $B^{\alpha_{0},p}_{\infty}(\Omega,\mathbb{R}^{k})$ iff
$$
\nr{w}_{B^{\alpha_{0},p}_{\infty}(\Omega)}:=\nr{w}_{L^{p}(\Omega)}+\left(\sup_{\snr{h}\not =0}\int_{\Omega_{2h}}\left|\frac{\tau_{h}^{2}w}{\snr{h}^{\alpha_{0}}}\right|^{p}\dx\right)^{1/p}<\infty.
$$
The local variant $B^{\alpha_{0},p}_{\infty;\loc}(\Omega,\mathbb{R}^{k})$ can be defined as done for $N^{\alpha_{0},p}_{\loc}(\Omega,\er^k )$ and $W^{\alpha_{0},p}_{\loc}(\Omega,\mathbb{R}^{k})$.
\end{itemize}
\end{definition}
\noindent Moreover we have that 
\begin{flalign}\label{33}
W^{\alpha ,p}(\Omega,\er^k)\subsetneqq N^{\alpha,p}(\Omega,\er^k)\subsetneqq
W^{\beta,p}(\Omega,\er^k)\quad \mbox{for every} \ \ \beta<\alpha,
\end{flalign}
holds for sufficiently regular domains $\Omega$. 
A local, quantified version of \eqref{33} is in the next lemma, that can be found in \cite[Section 3]{jw}. 
\begin{lemma}\label{l4}
Let $B_{\varrho} \Subset B_{r}\subset \er^n$ be concentric balls with $r\leq 1$, $w\in L^{p}(B_{r},\mathbb{R}^{k})$, $p>1$ and assume that, for $\alpha_{0} \in (0,1]$, $S\ge 0$, there holds
\eqn{cru1}
$$
\nr{\tau_{h}w}_{L^{p}(B_{\rr})}\le S\snr{h}^{\alpha_{0} } \quad \mbox{
for every $h\in \mathbb{R}^{n}$ with $0<\snr{h}\le \frac{r-\rr}{K}$, where $K \geq 1$}\;.
$$
Then it holds that 
\eqn{cru2}
$$
\nr{w}_{W^{\beta,p}(B_{\rr})}\le\frac{c}{(\alpha_{0} -\beta)^{1/p}}
\left(\frac{r-\rr}{K}\right)^{\alpha_{0} -\beta}S+c\left(\frac{K}{r-\rr}\right)^{n/p+\beta} \nr{w}_{L^{p}(B_{\rr})}\,,
$$
for all $\beta\in(0,\alpha_{0})$, where $c\equiv c(n,k,p)$. 
\end{lemma}
\noindent We further recall the embeddings of fractional Sobolev spaces and Besov-Nikol'skii spaces; in case $w\in W^{\alpha_{0},p}(B_{\rr}(x_{0}),\mathbb{R}^{k})$ this reads as 
\eqn{immersione}
$$
\nra{w}_{L^{\frac{np}{n-p\alpha_{0}}}(B_{\rr}(x_{0}))}\le c\nra{w}_{L^{p}(B_{\rr}(x_{0}))}+c\rr^{\alpha_{0}}\snra{w}_{\alpha_{0},p;B_{\rr}(x_{0})},
$$
provided $p\ge 1, \alpha_{0} \in (0,1)$, and $p\alpha_{0}<n$, with $c\equiv c(n,k,p,\alpha_{0})$, see \cite[Section 6]{guide}; while if
\begin{flalign*}
w\in L^{p}(B_{\rr+10\mathcal{h}_{0}}(x_{0}),\mathbb{R}^{k})\qquad \mbox{and}\qquad \sup_{0<\snr{h}<\mathcal{h}_{0}}\left\|\frac{\tau_{h}^{2}w}{\snr{h}^{\alpha_{0}}}\right\|_{L^{p}(B_{\rr+5\mathcal{h}_{0}}(x_{0}))}<\infty,
\end{flalign*}
where $\mathcal{h}_{0}\in (0,1]$, $\alpha_{0}\in (1,2)$ and $p\ge 1$, it is
\begin{flalign}\label{immersione2}
\nr{Dw}_{L^{p}(B_{\rr}(x_{0}))}\le \frac{c}{(\alpha_{0}-1)(2-\alpha_{0})}\left(\sup_{0<\snr{h}<\mathcal{h}_{0}}\left\|\frac{\tau_{h}^{2}w}{\snr{h}^{\alpha_{0}}}\right\|_{L^{p}(B_{\rr+5\mathcal{h}_{0}}(x_{0}))}+\mathcal{h}_{0}^{-\alpha_{0}}\nr{w}_{L^{p}(B_{\rr+6\mathcal{h}_{0}}(x_{0}))}\right),
\end{flalign}
with $c\equiv c(n,k,p)$, cf. \cite[Lemma 3.2]{jw}. Next, a localized Gagliardo-Nirenberg type inequality interpolating between Nikol'skii spaces and Sobolev spaces, cf. \cite[Lemma 2.4]{jw}.
\begin{lemma}\label{ls}
Let $B_{\rr}(x_{0})\subset \mathbb{R}^{n}$ be a ball, $\theta\in (0,1)$, $m\in [1,\infty)$ be numbers, and $w\in W^{1,m}(B_{\rr}(x_{0}))\cap L^{\infty}(B_{\rr}(x_{0}))$ be a function. Then $w\in N^{s,\frac{m}{s}}(B_{\theta\rr}(x_{0}))$ for all $s\in (0,1)$ with
\begin{eqnarray}\label{fs.0}
\theta^{ns/m}\sup_{0<\snr{h}<\rr(1-\theta)2^{-4}}\left(\mint_{B_{\theta\rr}(x_{0})}\left|\frac{\tau_{h}w}{\snr{h}^{s}}\right|^{\frac{m}{s}}\dx\right)^{\frac{s}{m}}&\le& \frac{c\nr{w}_{L^{\infty}(B_{\rr}(x_{0}))}^{1-s}}{(1-\theta)^{s}\rr^{s}}\left(\mint_{B_{\rr}(x_{0})}\snr{w}^{m}\dx\right)^{\frac{s}{m}}\nonumber \\
&&+c\nr{w}_{L^{\infty}(B_{\rr}(x_{0}))}^{1-s}\left(\mint_{B_{\rr}(x_{0})}\snr{Dw}^{m}\dx\right)^{\frac{s}{m}},
\end{eqnarray}
for $c\equiv c(n,s,m)$.
\end{lemma}
\noindent We close this section by recalling the elementary relation linking finite differences and weakly differentiable functions: if $B'\Subset B\subset \mathbb{R}^{n}$ are bounded open sets and $w\in W^{1,p}(B,\mathbb{R}^{k})$, $p\ge 1$ and $0<\snr{h}<\dist(B',\partial B)$, then
\eqn{gh}
$$
\snr{h}^{-1}\nr{\tau_{h}w}_{L^{p}(B')}\le \nr{Dw}_{L^{p}(B)}.
$$

\subsection{Nonlinear potentials}
A crucial role in this paper will be played by a general class of nonlinear potentials. First introduced by Havin \& Maz'ya \cite{HM}, nonlinear potentials recently found important applications in the regularity theory for nonuniformly elliptic problems \cite{BM,ciccio,piovra,dm,dp,BS2,jw} - in particular we will refer to \cite[Section 4]{piovra} for the potential theoretic technical toolbox needed here. For a ball $B_{r}(x_{0})\subset \mathbb{R}^{n}$, (fixed) parameters $\sigma>0, \vartheta\geq 0$, and a function $f\in L^{1}(B_{r}(x_{0}))$, we introduce the nonlinear Havin-Maz'ya-Wolff type potential ${\bf P}_{\sigma}^{\vartheta}(f;\cdot)$, i.e.:
\eqn{defi-P} 
$$
{\bf P}_{\sigma}^{\vartheta}(f;x_0,r) := \int_0^r \varrho^{\sigma} \left(  \mint_{B_{\varrho}(x_0)} \snr{f} \dx \right)^{\vartheta} \frac{\d\varrho}{\varrho} \,.
$$
The mapping properties among function spaces of ${\bf P}_{\sigma}^{\vartheta}(f;\cdot)$ needed here are contained in the next lemma, cf. \cite[Lemma 4.1]{piovra}.
\begin{lemma}\label{crit} 
Let $B_{\tau}\Subset B_{\tau+r}\subset \mathbb{R}^{n}$ be two concentric balls with $\tau, r\leq 1$, $f\in L^{1}(B_{\tau+r})$ and let $\sigma,\vartheta>0$ be such that $ n\vartheta>\sigma$. Then
\eqn{stimazza}
$$
\nr{{\bf P}_{\sigma}^{\vartheta}(f;\cdot,r)}_{L^{\infty}(B_{\tau})} \lesssim_{n,\vartheta,\sigma,m} \|f\|_{L^{m}(B_{\tau+r})}^{\vartheta} $$
holds whenever $m > n\vartheta/\sigma>1$. 
\end{lemma}
\noindent The next is a nonlinear potential theoretic version of De Giorgi's iteration finding its roots in \cite{kilp} and \cite{min11} - we shall record it in the form of a quantified reverse H\"older inequality, first appeared in \cite{piovra}, see also \cite{dm,jw}.
\begin{lemma}\label{revlem}
Let $B_{r_{0}}(x_{0})\subset \mathbb{R}^{n}$ be a ball, $n\ge 2$, $f \in L^1(B_{2r_0}(x_{0}))$, and constants $\chi >1$, $\sigma, \vartheta,\ti{c},M_{0}>0$ and $\kappa_0, M_{1}\geq 0$. Assume that $v \in L^2(B_{r_0}(x_0))$ is such that for all $\kk\ge \kk_{0}$, and for every concentric ball $B_{\rr}(x_{0})\subseteq B_{r_{0}}(x_{0})$, the inequality
\begin{flalign}
\left(\mint_{B_{\rr/2}(x_{0})}(v-\kk)_{+}^{2\chi}  \dx\right)^{\frac1{2\chi}}  &\le \ti{c}M_{0}\left(\mint_{B_{\rr}(x_{0})}(v-\kk)_{+}^{2}  \dx\right)^{\frac{1}{2}}+\ti{c} M_{1}\rr^{\sigma}\left(\mint_{B_{\rr}(x_{0})}\snr{f}  \dx\right)^{\vartheta}
 \label{revva}
\end{flalign}
holds. If $x_{0}$ is a Lebesgue point of $v$ in the sense that 
$$
v(x_0) = \lim_{r\to 0} (v)_{B_{r}(x_0)}\,,
$$ then
\eqn{siapplica}
$$
 v(x_{0})  \le\kk_{0}+cM_{0}^{\frac{\chi}{\chi-1}}\left(\mint_{B_{r_{0}}(x_{0})}(v-\kk_{0})_{+}^{2}  \dx\right)^{1/2}
+cM_{0}^{\frac{1}{\chi-1}} M_{1}\mathbf{P}^{\vartheta}_{\sigma}(f;x_{0},2r_{0})
$$
holds with $c\equiv c(n,\chi,\sigma,\vartheta,\ti{c})$.  
\end{lemma}

\subsection{Tools for $p$-Laplacian type problems} When dealing with singular or degenerate problems of the $p$-Laplacian type, it is customary to employ a family of vector fields encoding the scaling features of the operator. More precisely, for $s\in [0,1]$, $0<p<\infty$, we let $V_{s,p}\colon \mathbb{R}^{n}\to \mathbb{R}^{n}$ be defined as $V_{s,p}(z):=(s^{2}+\snr{z}^{2})^{(p-2)/4}z$. It is well-known that $V_{s,p}$ satisfies
\eqn{Vm}
$$
\begin{cases}
\ \snr{V_{s,p}(z_{1})-V_{s,p}(z_{2})}\approx_{n,p}(s^{2}+\snr{z_{1}}^{2}+\snr{z_{2}}^{2})^{\frac{p-2}{4}}\snr{z_{1}-z_{2}}\qquad \mbox{for all} \ \ z_{1},z_{2}\in \mathbb{R}^{n},\\
\ \snr{z_{1}-z_{2}}\lesssim_{n,p}\snr{V_{s,p}(z_{1})-V_{s,p}(z_{2})}^{\frac{2}{p}}+\mathds{1}_{\{p<2\}}\snr{V_{s,p}(z_{1})-V_{s,p}(z_{2})}(\snr{z_1}^{2}+s^{2})^{(2-p)/4}
\end{cases}
$$
see \cite[Section 2.2]{dm}, \cite[Section 2.2]{jw} and references therein. 
A helpful monotonicity inequality from \cite[Lemma A.3]{bls} we shall need is
\eqn{mon}
$$
\snr{\snr{t_{1}}^{p-1}t_{1}-\snr{t_{2}}^{p-1}t_{2}}\gtrsim_{p} \snr{t_{1}-t_{2}}^{p},\qquad t_{1},t_{2}\in \mathbb{R}, \ \  p\ge 1.
$$
We further recall from \cite[Section 2.2]{jw} the nonautonomous counterpart of the integration by parts trick appearing in \cite[Section 8.2]{giu}.

\begin{lemma}
Let $\rr,\mathcal{h}_{0}>0$ be numbers, $h\in \mathbb{R}^{n}$ be a vector such that $\snr{h}\in (0,\mathcal{h}_{0}/4)$, $B_{\rr}(x_{0})\subset \mathbb{R}^{n}$ be a ball, $V\in L^{\infty}(B_{\rr+\mathcal{h}_{0}}(x_{0}),\mathbb{R}^{n})$ and $W\in W^{1,\infty}_{0}(B_{\rr}(x_{0}),\mathbb{R}^{n})$ be functions, and $H\in C(B_{\rr+\mathcal{h}_{0}}(x_{0})\times \mathbb{R}^{n},\mathbb{R}^{n})$ be a continuous vector field which is bounded on $B_{\rr+\mathcal{h}_{0}}(x_{0})\times \ti{\Omega}$ for every bounded subset $\ti{\Omega}\subset \mathbb{R}^{n}$. Then
\eqn{af}
$$
\int_{B_{\rr}(x_{0})}\langle \tau_{h}H(\cdot,V(\cdot)),W\rangle\dx=-\snr{h}\int_{B_{\rr}(x_{0})}\int_{0}^{1}\langle H(x+\beta h,V(x+\beta h)),\partial_{h/\snr{h}}W\rangle\d\beta\dx.
$$
\end{lemma}
\noindent We conclude this section with a basic iteration lemma.
\begin{lemma}\label{iterlem}
Let $h\colon [t,s]\to \mathbb{R}$ be a non-negative and bounded function, and let $a,b, \gamma$ be non-negative numbers. Assume that the inequality 
$ 
h(\tau_2)\le  (1/2) h(\tau_1)+(\tau_1-\tau_2)^{-\gamma}a+b,
$
holds whenever $t\le \tau_2<\tau_1\le s$. Then $
h(t)\le c( \gamma)[a(s-t)^{-\gamma}+b]
$, holds too. 
\end{lemma}

\subsection{Structural assumptions}\label{sa} In this section, we list the minimal set of assumptions needed in the proof of our results. For the sake of clarity, we distinguish between the nearly linear growth case, treated in Theorems \ref{mt1}-\ref{count.3}, and the linear growth one, object of Theorem \ref{count.3}.
\subsubsection*{Nearly linear growth} We assume that the integrand $\tx{L}\colon \mathbb{R}^{n}\to \mathbb{R}$ satisfies the following regularity
\eqn{ass.0}
$$
\tx{L}\in C^{2}_{\loc}(\mathbb{R}^{n}\setminus \{0\})\cap C^{1}_{\loc}(\mathbb{R}^{n}),
$$
and growth/ellipticity conditions:
\eqn{ass.1}
$$
\left\{
\begin{array}{c}
\displaystyle
\ \frac{1}{\Lambda}\snr{z}\tx{g}(\snr{z})-\Lambda\le \tx{L}(z)\le \Lambda\snr{z}\tx{g}(\snr{z})+\Lambda\\[8pt]\displaystyle
\ \frac{\snr{\xi}^{2}}{\Lambda(1+\snr{z}^{2})^{\mu/2}}\le \langle \partial^{2}\tx{L}(z)\xi,\xi\rangle,\quad \quad \snr{\partial^{2}\tx{L}(z)}\le \frac{\Lambda(\tx{g}(\snr{z})+1)}{(1+\snr{z}^{2})^{1/2}},
\end{array}
\right.
$$
for all $z\in \mathbb{R}^{n}\setminus \{0\}$, $\xi\in \mathbb{R}^{n}$, some constant $\Lambda\ge 1$, and $\mu\ge 1$. Here, function $\tx{g}$ verifies:
\eqn{ass.2}
$$
\begin{cases}
\ 0\le \tx{g}(\cdot)\in C[0,\infty),\qquad t\tx{g}(t) \ \ \mbox{is convex}, \vspace{1mm}\\
\ \tx{g} \ \ \mbox{is nondecreasing, unbounded, and concave,}\vspace{1mm}\\
\ \tx{g}(t)\le c(\tx{g},\varepsilon)(1+t^{\varepsilon})\qquad \mbox{for all} \ \ \varepsilon>0, \ \ t\ge 0.
\end{cases}
$$
\subsubsection*{Linear growth} To describe the elliptic component of area-type double phase functionals (keep \eqref{area} in mind), we assume that integrand $\tx{L}\colon \mathbb{R}^{n}\to \mathbb{R}$ satisfies \eqref{ass.0} and
\eqn{gcgc}
$$
\begin{cases}
\displaystyle
\ z\mapsto \tx{L}(z)\mbox{ is strictly convex}\vspace{1.5mm}\\ \displaystyle
\displaystyle
    \ \mbox{\eqref{ass.1} holds with $\tx{g}=1$ and $\mu>1$ for all $z\in \mathbb{R}^{n}\setminus B_{1}(0)$}.
\end{cases}
$$

\subsubsection*{$q$-anisotropy and the obstacle function}
Unless differently specified, in the $q$-component of energy \eqref{ll1} it will always be $s\in [0,1]$. The modulating coefficient $a\colon \Omega\to [0,\infty)$ verifies $\eqref{1qa}_{1}$,
and the obstacle function $\psi\in W^{1,1}_{\loc}(\Omega)$ will always have locally finite energy, i.e.,
\eqn{p1}
$$
\mathcal{H}(\psi;\ti{\Omega})<\infty\qquad \mbox{for all open sets} \ \ \ti{\Omega}\Subset \Omega,
$$
while in the proof of our main theorems, we shall further require that 
\eqn{p2}
$$
\psi\in W^{1,\infty}_{\loc}(\Omega)\cap W^{2,1}_{\loc}(\Omega)\quad \mbox{with} \quad \snr{D\psi}^{q-2}D^{2}\psi\in L^{\infty}_{\loc}(\Omega,\otimes_{2}\mathbb{R}^{n}).
$$
Moreover, let us introduce the modular \cite[Section 3.1]{hh19} naturally related to the class of functionals that we are considering. For $q>1$, coefficient $a$ as in \eqref{1qa}$_{1}$, and $z\in \mathbb{R}^{n}$, set 
$$
\ti{\tx{g}}(\snr{z}):=\begin{cases}
    \displaystyle
    \ \snr{z}\tx{g}(\snr{z})\quad &\mbox{if} \ \ \eqref{ass.1}\mbox{-}\eqref{ass.2} \ \ \mbox{hold}\vspace{1.5mm}\\
    \displaystyle
    \ \sqrt{1+\snr{z}^{2}}-1\quad &\mbox{if} \ \ \eqref{gcgc} \ \ \mbox{is in force},
\end{cases}
$$
and define
\eqn{ggg}
$$
\begin{array}{c}
\displaystyle
W^{1,1}_{\loc}(\Omega)\ni w\mapsto \mathcal{G}(w;\Omega):=\int_{\Omega}\tx{G}(x,\snr{Dw})\dx\\[8pt]\displaystyle
\tx{G}(x,\snr{z}):=\ti{\tx{g}}(\snr{z})+a(x)\snr{z}^{q}.
\end{array}
$$
Finally, let us quickly describe the subspaces of $W^{1,1}_{\loc}(\Omega)$ that are the natural functional setting for our problems. With $B\Subset \Omega$ being a ball, we introduce constrained classes
\begin{flalign*}
&\tx{K}^{\psi}_{\loc}(\Omega):=\left\{w\in W^{1,1}_{\loc}(\Omega)\colon w(x)\ge \psi(x) \ \ \mbox{for a.e.} \ \ x\in \Omega\right\},\\
&\tx{K}^{\psi}(B):=\left\{w\in W^{1,1}(B)\colon w(x)\ge \psi(x) \ \ \mbox{for a.e.} \ \ x\in B\right\},
\end{flalign*}
that are convex subsets of $W^{1,1}_{\loc}(\Omega)$, $W^{1,1}(B)$ respectively. It goes without saying that this set of assumptions allows handling simultaneously free and constrained local minimizer of functional $\mathcal{H}$ - in fact formally taking $\psi\equiv -\infty$ yields $\tx{K}^{-\infty}_{\loc}(\Omega)\equiv W^{1,1}_{\loc}(\Omega)$.
\begin{remark}\label{r21} \emph{A few of comments are in order.}
\begin{itemize}
\item \emph{Up to replace $\mathbb{R}^{n}\ni z\mapsto \tx{L}(z)$ with $\mathbb{R}^{n}\ni z\mapsto\tx{L}(z)-\langle \partial \tx{L}(0),z\rangle$ that does not change the set of local minimizers, we can always assume that $\partial \tx{L}(0)=0$. Needless to say, the "perturbed" integrand still satisfies \eqref{ass.1} with possibly enlarged constants.}
\item \emph{The unboundedness of $\tx{g}$ implies that there exists $\tx{T}_{\tx{g}}\equiv \tx{T}_{\tx{g}}(\tx{g})\ge 1$ such that $\tx{g}(t)\ge 1$ whenever $t\ge \tx{T}_{\tx{g}}$, therefore}
\eqn{gt}
$$
\snr{z}\le \tx{T}_{\tx{g}}+\snr{z}\tx{g}(\snr{z})\stackrel{\eqref{ass.1}_{1}}{\le}\tx{T}_{\tx{g}}+\Lambda^{2}+\Lambda\tx{L}(z)\qquad \mbox{for all} \ \ z\in \mathbb{R}^{n}.
$$
\item \emph{Under the main regularity assumption \eqref{1qa}, we have $1<q<2$ and condition \eqref{p2} implies that $\ell_{s}(D\psi)^{q-2}D^{2}\psi\in L^{\infty}_{\loc}(\Omega,\otimes_{2}\mathbb{R}^{n})$ for all $s\in [0,1]$ and $\psi\in W^{2,\infty}_{\loc}(\Omega)$: in fact, $\nr{D^{2}\psi}_{L^{\infty}(B)}\lesssim_{q}(1+\nr{D\psi}_{L^{\infty}(B)}^{2-q})\nr{\snr{D\psi}^{q-2}D^{2}\psi}_{L^{\infty}(B)}<\infty$ for all balls $B\Subset \Omega$. Let us stress that assumptions \eqref{p1}-\eqref{p2} imposed on the obstacle function $\psi$ are the classical ones considered in the literature when possibly degenerate or singular energy functionals are considered, see e.g. \cite{fuc90,FM}. Since one of the key points of this paper is to introduce a novel technical toolbox that allows for the first time to deliver maximal gradient regularity in nondifferentiable, nonuniformly elliptic obstacle problems, we chose to stick to the classical setting in order to keep at a reasonable level the amount of technicalities involved. Relaxing \eqref{p2} means either imposing a Dini-type assumption on $D\psi$ as in \cite{km12} or assuming that $D^{2}\psi$ belongs to suitable Lorentz or Orlicz spaces, cf. \cite{ciccio}. As $\mu$-ellipticity in \eqref{ass.1} forces an important loss of ellipticity information, both improvements - if feasible - would call for new borderline approaches, which we plan to investigate in future work.}
\end{itemize}
\end{remark}

\section{Absence of Lavrentiev phenomenon in constrained classes}\label{alavcc} \noindent Let us quickly discuss the local approximation in energy of functions belonging to $\tx{K}_{\loc}^{\psi}(\Omega)$ in terms of maps sharing the same regularity as the obstacle $\psi$. We rely on elementary truncation and mollification techniques: similar arguments were indeed applied in \cite{balci,bs21,balci3,buli,bb90}, see also \cite{sharp,dm,krs}, in relation to the (non)occurrence of Lavrentiev phenomenon in double phase problems. We stress that our findings are sharp thanks to Theorem \ref{count}. Specifically, in line with \cite[Lemma 6.2]{jw},\footnote{In \cite[Lemma 6.2]{jw} it is shown that on a priori bounded functions, the usual Lebesgue-Serrin-Marcellini extension coincides with the relaxation defined along (suitably regular) sequences of uniformly bounded maps.} we revisit the density results of \cite{buli} in the quantitative form we need. 
\begin{lemma}\label{alav.1}
Under assumption $\eqref{ass.0}$, $\eqref{ass.1}_{1}$, $\eqref{ass.2}$, let $w\in W^{1,1}_{\loc}(\Omega)$ be a function such that $\snr{Dw}\tx{g}(\snr{Dw})\in L^{1}_{\loc}(\Omega)$, $\vartheta>0$ be a number, $\ti{\Omega}\Subset \Omega$ be an open, bounded set with Lipschitz-regular boundary. With $0<\varepsilon<\varepsilon_{0}:=\min\{1,\dist(\ti{\Omega},\partial \Omega)/4\}$, set $\Omega_{0}:=\{x+B_{\varepsilon_{0}}(0)\colon x\in \ti{\Omega}\}\Subset \Omega$. There exists a sequence $\{\bar{w}_{\varepsilon}\}_{\varepsilon>0}\subset C^{\infty}(\ti{\Omega})$ such that $\bar{w}_{\varepsilon}\to w $ strongly in $W^{1,1}(\ti{\Omega})$ and $\nr{D\bar{w}_{\varepsilon}}_{L^{\infty}(\ti{\Omega})}\lesssim_{n}\varepsilon^{-(1+\vartheta)}$. If $w\in L^{\infty}_{\loc}(\Omega)\cap W^{1,1}_{\loc}(\Omega)$ with $\snr{Dw}\tx{g}(\snr{Dw})\in L^{1}_{\loc}(\Omega)$, together with the strong convergence in $W^{1,1}(\ti{\Omega})$, we also have the bounds $\nr{\bar{w}_{\varepsilon}}_{L^{\infty}(\ti{\Omega})}\le \nr{w}_{L^{\infty}(\Omega_{0})}$ and $\nr{D\bar{w}_{\varepsilon}}_{L^{\infty}(\ti{\Omega})}\lesssim_{n}\varepsilon^{-1}\nr{w}_{L^{\infty}(\Omega_{0})}$.
\end{lemma}
\begin{proof}
    Let $\varepsilon,\varepsilon_{0},\ti{\Omega},\Omega_{0}$ be as in the statement, $\phi\in C^{\infty}_{c}(B_{1}(0))$ be a smooth, radially symmetric, nonnegative function with $\nr{\phi}_{L^{1}(\mathbb{R}^{n})}=1$ and $\nr{\phi}_{L^{\infty}(\mathbb{R}^{n})}\lesssim_{n}1$, and $\{\phi_{\varepsilon}\}_{\varepsilon>0}\subset C^{\infty}(\mathbb{R}^{n})$ be the sequence of mollifiers defined through $\phi$. We truncate $w$ by setting $w_{\varepsilon}:=\min\{\varepsilon^{-\vartheta},\max\{w,-\varepsilon^{-\vartheta}\}\}$ and regularize it by convolution defining $\bar{w}_{\varepsilon}:=w_{\varepsilon}*\phi_{\varepsilon}$. We start without the local boundedness assumption in force. By definition it is $\{\bar{w}_{\varepsilon}\}_{\varepsilon>0}\subset C^{\infty}(\ti{\Omega})$, $\nr{w_{\varepsilon}}_{L^{\infty}(\Omega_{0})}\le \varepsilon^{-\vartheta}$, thus $\nr{\bar{w}_{\varepsilon}}_{L^{\infty}(\ti{\Omega})}\le \varepsilon^{-\vartheta}$. By the former bound and Young's convolution inequality we control $\nr{D\bar{w}_{\varepsilon}}_{L^{\infty}(\ti{\Omega})}\le \nr{w_{\varepsilon}}_{L^{\infty}(\Omega_{0})}\nr{D\phi_{\varepsilon}}_{L^{1}(\Omega_{0})}\lesssim_{n}\varepsilon^{-(1+\vartheta)}$. We only need to prove strong $W^{1,1}$-convergence in $\ti{\Omega}$. Recalling that by Sobolev embedding theorem it is $w\in L^{\frac{n}{n-1}}(\Omega_{0})$, for $x\in \ti{\Omega}$ we bound
    \begin{eqnarray*}
    \snr{\bar{w}_{\varepsilon}(x)-w(x)}&\le&\snr{w*\phi_{\varepsilon}(x)-w(x)}\nonumber \\
    &&+\left|\int_{B_{1}\cap\{y\colon\snr{w(x+\varepsilon y)}>\varepsilon^{-\vartheta}\}}[w_{\varepsilon}(x+\varepsilon y)-w(x+\varepsilon y)]\phi(y)\dy\right|\nonumber \\
    &\lesssim_{n}&\snr{w*\phi_{\varepsilon}(x)-w(x)}+\int_{B_{1}\cap\{y\colon\snr{w(x+\varepsilon y)}>\varepsilon^{-\vartheta}\}}\varepsilon^{-\vartheta}+\snr{w(x+\varepsilon y)}\dy\nonumber \\
    &\lesssim_{n}&\snr{w*\phi_{\varepsilon}(x)-w(x)}+\int_{B_{1}\cap\{y\colon\snr{w(x+\varepsilon y)}>\varepsilon^{-\vartheta}\}}\snr{w(x+\varepsilon y)}\dy\nonumber \\
    &\lesssim_{n}&\snr{w*\phi_{\varepsilon}(x)-w(x)}+\varepsilon^{\frac{\vartheta}{n-1}}\int_{B_{1}}\snr{w(x+\varepsilon y)}^{\frac{n}{n-1}}\dy,
    \end{eqnarray*}
    therefore $\nr{\bar{w}_{\varepsilon}-w}_{L^{1}(\ti{\Omega})}\lesssim_{n}\nr{w*\phi_{\varepsilon}-w}_{L^{1}(\ti{\Omega})}+\varepsilon^{\frac{\vartheta}{n-1}}\nr{w}_{L^{\frac{n}{n-1}}(\Omega_{0})}^{\frac{n}{n-1}}\equiv \texttt{o}(\varepsilon)$, by basic mollification properties. Similarly, with $x\in \ti{\Omega}$, we bound
    \begin{eqnarray*}
    \snr{D\bar{w}_{\varepsilon}(x)-Dw(x)}&\le&\snr{Dw*\phi_{\varepsilon}(x)-Dw(x)}+\int_{B_{1}\cap\{\snr{w(x+\varepsilon y)}>\varepsilon^{-\vartheta}\}}\snr{Dw(x+\varepsilon y)}\phi(y)\dy\nonumber \\
    &\lesssim_{n}&\snr{Dw*\phi_{\varepsilon}(x)-Dw(x)}\nonumber \\
    &&+\frac{1}{\tx{g}(\varepsilon^{-\vartheta})}\int_{B_{1}\cap\{\snr{w(x+\varepsilon y)}>\varepsilon^{-\vartheta}\}}\tx{g}(\snr{Dw(x+\varepsilon y)})\snr{Dw(x+\varepsilon y)}\dy\nonumber \\
    &&+\varepsilon^{\frac{\vartheta}{n-1}}\int_{B_{1}\cap\{\snr{w(x+\varepsilon y)}>\varepsilon^{-\vartheta}\}}\snr{w(x+\varepsilon y)}^{\frac{n}{n-1}}\dy,
    \end{eqnarray*}
    so $\nr{D\bar{w}_{\varepsilon}-Dw}_{L^{1}(\ti{\Omega})}\lesssim_{n} \nr{Dw*\phi_{\varepsilon}-Dw}_{L^{1}(\ti{\Omega})}+\tx{g}(\varepsilon^{-\vartheta})^{-1}\nr{\snr{Dw}\tx{g}(\snr{Dw})}_{L^{1}(\Omega_{0})}+\varepsilon^{\frac{\vartheta}{n-1}}\nr{w}_{L^{\frac{n}{n-1}}(\Omega_{0})}^{\frac{n}{n-1}}\equiv \texttt{o}(\varepsilon)$, where we also used that $\tx{g}$ is nondecreasing, continuous and unbounded. This completes the proof in case $w$ is unbounded. If additionally $w\in L^{\infty}_{\loc}(\Omega)$, the above procedure boils down to standard mollification as truncation is obviously no longer needed, and the estimates on $\nr{\bar{w}_{\varepsilon}}_{L^{\infty}(\ti{\Omega})}$, $\nr{D\bar{w}_{\varepsilon}}_{L^{\infty}(\ti{\Omega})}$ directly follow by definition and Young's convolution inequality. 
\end{proof}
\noindent Now we are ready to prove a density result: our construction yields a sequence of approximating maps sharing the same regularity as the obstacle $\psi$ - to highlight this, we assume that $\psi$ belongs to a general function space\footnote{Recall the convention established in Section \ref{not}.} $\mathbb{X}_{\loc}(\Omega)$ and has locally finite energy \eqref{p1}.
\begin{lemma}\label{alav}
Under assumptions \eqref{1qa} and \eqref{ass.0}, $\eqref{ass.1}_{1}$-\eqref{p1}, let $w\in \tx{K}^{\psi}_{\loc}(\Omega)$ be any function with $\tx{H}(\cdot,Dw)\in L^{1}_{\loc}(\Omega)$, and $\ti{\Omega}$, $\Omega_{0}$, $\varepsilon$, $\varepsilon_{0}$, $\{\phi_{\varepsilon}\}_{\varepsilon>0}$ be as in Lemma \ref{alav.1}. If $\psi\in \mathbb{X}_{\loc}(\Omega)$, there exists a sequence $\{\ti{w}_{\varepsilon}\}_{\varepsilon>0}\subset \mathbb{X}(\ti{\Omega})\cap \tx{K}^{\psi}(\ti{\Omega})$ such that 
\eqn{conv}
$$
\ti{w}_{\varepsilon}\to w \ \ \mbox{strongly in} \ \ W^{1,1}(\ti{\Omega}), \qquad \mathcal{H}(\ti{w}_{\varepsilon};\ti{\Omega})\to \mathcal{H}(w;\ti{\Omega}),\qquad  \mathcal{G}(\ti{w}_{\varepsilon}-w;\ti{\Omega})\to 0,
$$
where $\mathcal{H}$ is the functional in \eqref{ll1} and $\mathcal{G}$ is the related modular in \eqref{ggg}.
\end{lemma}
\begin{proof}
We first handle the case of unbounded maps. We apply Lemma \ref{alav.1} to $w$ and $\psi$ with $\vartheta:=-1+\alpha/(q-1)>0$, thus obtaining sequences $\{\bar{w}_{\varepsilon}\}_{\varepsilon>0},\{\bar{\psi}_{\varepsilon}\}_{\varepsilon>0}\in C^{\infty}(\ti{\Omega})$, satisfying
\eqn{conv.1}
$$
\max\left\{\nr{D\bar{w}_{\varepsilon}}_{L^{\infty}(\ti{\Omega})},\nr{D\bar{\psi}_{\varepsilon}}_{L^{\infty}(\ti{\Omega})}\right\}\lesssim_{n}\varepsilon^{-(1+\vartheta)}.
$$
and such that
\eqn{conv.3}
$$
 \bar{w}_{\varepsilon}\to w \ \ \mbox{strongly in} \ \ W^{1,1}(\ti{\Omega})\qquad\mbox{and}\qquad   \bar{\psi}_{\varepsilon}\to \psi \ \ \mbox{strongly in} \ \ W^{1,1}(\ti{\Omega}).
$$
We then let $\ti{w}_{\varepsilon}:=\bar{w}_{\varepsilon}-\bar{\psi}_{\varepsilon}+\psi$. Via \eqref{conv.3} we immediately see that $\eqref{conv}_{1}$ holds true. Moreover, being $\bar{w}_{\varepsilon}$, $\bar{\psi}_{\varepsilon}$ smooth, $\ti{w}_{\varepsilon}$ is as regular as $\psi$ allows, and,
\begin{eqnarray*}
\ti{w}_{\varepsilon}(x)&=&\bar{w}_{\varepsilon}(x)-\bar{\psi}_{\varepsilon}(x)+\psi(x)\nonumber \\
&=&\int_{B_{1}(0)}\min\{\varepsilon^{-\vartheta},\max\{w(x+\varepsilon y),-\varepsilon^{-\vartheta}\}\}\phi(y)\dy-\bar{\psi}_{\varepsilon}(x)+\psi(x)\nonumber \\
&\stackrel{w\ge \psi}{\ge}&\int_{B_{1}(0)}\min\{\varepsilon^{-\vartheta},\max\{\psi(x+\varepsilon y),-\varepsilon^{-\vartheta}\}\}\phi(y)\dy-\bar{\psi}_{\varepsilon}(x)+\psi(x)\ge \psi(x) \ \Longrightarrow \ti{w}_{\varepsilon}\in \tx{K}^{\psi}(\ti{\Omega}).
\end{eqnarray*}
Observe also that by construction, on $\ti{\Omega}$ it is
\eqn{cons}
$$
\snr{D\bar{w}_{\varepsilon}(x)}\le \snr{Dw}*\phi_{\varepsilon}(x)\qquad \mbox{and}\qquad \snr{D\bar{\psi}_{\varepsilon}(x)}\le \snr{D\psi}*\phi_{\varepsilon}(x).
$$
Finally, for $x\in \ti{\Omega}$ set $a_{\varepsilon}(x):=\inf_{y\in B_{\varepsilon}(x)}a(y)$, $\tx{G}_{\varepsilon}(x,\snr{z}):=\snr{z}\tx{g}(\snr{z})+a_{\varepsilon}(x)\snr{z}^{q}$ and observe that
\eqn{conv.6}
$$
\tx{G}_{\varepsilon}(x,\snr{z})\le \tx{G}(x+\varepsilon y,\snr{z})\qquad \mbox{for all} \ \ y\in B_{1}(0).
$$
Now, recall our choice of $\vartheta$ to control 
\begin{eqnarray}\label{conv.5}
\snr{a(x)-a_{\varepsilon}(x)}\snr{D\bar{w}_{\varepsilon}-D\bar{\psi}_{\varepsilon}}^{q}&\le& [a]_{0,\alpha;\Omega_{0}}\varepsilon^{\alpha}\nr{D\bar{w}_{\varepsilon}-D\bar{\psi}_{\varepsilon}}_{L^{\infty}(\ti{\Omega})}^{q-1}\snr{D\bar{w}_{\varepsilon}-D\bar{\psi}_{\varepsilon}}\nonumber \\
&\stackrel{\eqref{conv.1}}{\lesssim_{n,q}}&[a]_{0,\alpha;\Omega_{0}}\varepsilon^{\alpha-(1+\vartheta)(q-1)}\snr{D\bar{w}_{\varepsilon}-D\bar{\psi}_{\varepsilon}}\nonumber \\
&\lesssim_{n,q}&[a]_{0,\alpha;\Omega}\snr{D\bar{w}_{\varepsilon}-D\bar{\psi}_{\varepsilon}},
\end{eqnarray}
and apply the definition of convolution, Jensen's inequality, \eqref{ass.2} and \eqref{cons}-\eqref{conv.5} to gain
\begin{eqnarray}\label{conv.2}
\tx{G}(x,\snr{D\ti{w}_{\varepsilon}})&\lesssim&\tx{G}(x,\snr{D\psi})+\tx{G}_{\varepsilon}(x,\snr{D\bar{w}_{\varepsilon}-D\bar{\psi}_{\varepsilon}})+\snr{a(x)-a_{\varepsilon}(x)}\snr{D\bar{w}_{\varepsilon}-D\bar{\psi}_{\varepsilon}}^{q}\nonumber \\
&\lesssim&\tx{G}(x,\snr{D\psi})+\tx{G}_{\varepsilon}(x,\snr{D\bar{w}_{\varepsilon}-D\bar{\psi}_{\varepsilon}})+[a]_{0,\alpha;\Omega}\snr{D\bar{w}_{\varepsilon}-D\bar{\psi}_{\varepsilon}}\nonumber \\
&\lesssim&1+\tx{G}(x,\snr{D\psi})+\tx{G}(\cdot,\snr{Dw})*\phi_{\varepsilon}(x)+\tx{G}(\cdot,\snr{D\psi})*\phi_{\varepsilon}(x),
\end{eqnarray}
up to constants depending on $(n,\tx{g},q,[a]_{0,\alpha;\Omega})$. Keeping in mind \eqref{conv.3}, \eqref{conv.2}, \eqref{ass.1}$_{1}$ and a well-known variant of the dominated convergence theorem imply that $\mathcal{G}(\ti{w}_{\varepsilon}-w;\ti{\Omega})\to 0$ and $\tx{H}(\cdot,D\ti{w}_{\varepsilon})\to \tx{H}(\cdot,Dw)$ in $L^{1}(\ti{\Omega})$ (up to nonrelabelled subsequences), thus \eqref{conv} is completely proven. 
\end{proof}
\noindent The next corollary is a straightforward consequence of Lemmas \ref{alav.1}-\ref{alav}.
\begin{corollary}\label{alav.2}
    In the same setting as Lemma \ref{alav}, convergence \eqref{conv} holds for a priori bounded maps $w\in \tx{K}^{\psi}_{\loc}(\Omega)\cap L^{\infty}_{\loc}(\Omega)$, $\psi\in\mathbb{X}_{\loc}(\Omega)\cap L^{\infty}_{\loc}(\Omega) $ provided the large bound $q\le 1+\alpha$ is in force. The $L^{\infty}$-bounds
\eqn{linfb}
$$
\begin{cases}
\displaystyle
\ \nr{\ti{w}_{\varepsilon}}_{L^{\infty}(\ti{\Omega})}\le 4\max\left\{\nr{w}_{L^{\infty}(\Omega_{0})},\nr{\psi}_{L^{\infty}(\Omega_{0})}\right\},\vspace{1.5mm}\\ \displaystyle
\ \nr{D\ti{w}_{\varepsilon}}_{L^{\infty}(\ti{\Omega})}\lesssim_{n}\varepsilon^{-1}\max\left\{\nr{w}_{L^{\infty}(\Omega_{0})},\nr{\psi}_{L^{\infty}(\Omega_{0})}\right\}+\nr{D\psi}_{L^{\infty}(\tilde{\Omega})}
\end{cases}
$$
hold true.
\end{corollary}
\begin{proof}
    The claim follows via the same procedure in the proof of Lemma \ref{alav} with $\vartheta=0$, after observing that if $w,\psi\in L^{\infty}_{\loc}(\Omega)$, the sequence $\{\ti{w}_{\varepsilon}\}_{\varepsilon>0}\subset \mathbb{X}(\ti{\Omega})\cap L^{\infty}(\ti{\Omega})$ given by Lemma \ref{alav} involves the usual convolutions of $w$ and $\psi$. Bounds \eqref{linfb} are granted by Lemma \ref{alav.1}.
\end{proof}
\section{Hybrid fractional Moser iteration} \label{ncz} 
\noindent This section is devoted to the proof of a hybrid version of fractional Moser iteration in the sense that the resulting a priori estimate still contains a very small power of the $L^{\infty}$-norm of the gradient of (free or constrained) minima of functionals governed by regular log-double phase integrands. To this end, we introduce a smoothed form of the governing general integrand in \eqref{ll1}, endowed with additional (artificial) structural properties. These reduce the problem to the nonsingular $q$-Laplacian setting, where (a priori) $C^{1}$-regularity of solutions follows from classical theory \cite{fuc90,ch}. Specifically, we mimic the family of regularized integrands constructed in Section \ref{app} - essentially the same objects, but, for the sake of clarity, written without indices - which are the key to our approximation scheme. We consider a map $\tx{L}_{*}\colon \mathbb{R}^{n}\to \mathbb{R}$ satisfying
\begin{flalign}\label{assfr}
\left\{
\begin{array}{c}
\displaystyle
\ \tx{L}_{*}\in C^{2}_{\loc}(\mathbb{R}^{n}) \\[8pt]\displaystyle
\ \frac{1}{\Lambda_{*}}\snr{z}\tx{g}(\snr{z})-\Lambda_{*}\le \tx{L}_{*}(z)\le \Lambda_{*}\snr{z}\tx{g}(\snr{z})+\Lambda_{*} \\[8pt]\displaystyle
\ \frac{\snr{\xi}^{2}}{\Lambda_{*}(1+\snr{z}^{2})^{\mu/2}}\le \langle\partial^{2}\tx{L}_{*}(z)\xi,\xi\rangle,\qquad \quad \snr{\partial^{2}\tx{L}_{*}(z)}\le\frac{\Lambda_{*}(\tx{g}(\snr{z})+1)}{(1+\snr{z}^{2})^{1/2}},
\end{array}
\right.
\end{flalign}
for all $z,\xi\in \mathbb{R}^{n}$, and some absolute constant $\Lambda_{*}\ge 1$. Here $\tx{g}$ is the same function in \eqref{ass.2} - notice that now \eqref{gt} holds for $\tx{L}_{*}$ with $\Lambda_{*}$ replacing $\Lambda$. Arguing as in \cite[Lemma 2.1]{ma1}, we deduce that $\eqref{assfr}$ yields
\eqn{d.1}
$$
\snr{\partial \tx{L}_{*}(z)}\lesssim_{\Lambda_{*},\tx{g}}\tx{g}(\snr{z})+1\qquad \mbox{for all} \ \ z\in \mathbb{R}^{n}.
$$
Moreover, with $a$ being the nonnegative modulating coefficient in \eqref{1qa}$_{1}$,
\eqn{sasa}
$$
0<a_{*}\le 1,\qquad \quad 0<s\le 2
$$
being strictly positive constants, we set $\tx{a}(x):=a(x)+a_{*}$ - obviously by \eqref{1qa}$_{1}$ it is $\tx{a}\in C^{0,\alpha}(\Omega)$ and $[\tx{a}]_{0,\alpha;\ti{\Omega}}=[a]_{0,\alpha;\ti{\Omega}}$ for all subsets $\ti{\Omega}\subseteq \Omega$ - and define
\eqn{hlhl}
$$
\begin{array}{c}
\tx{H}_{*}(x,z):=\tx{L}_{*}(z)+\tx{a}(x)\ell_{s}(z)^{q}\\[8pt]\displaystyle
\lambda_{1;*}(x,\snr{z}):=\ell_{1}(z)^{-\mu}+\tx{a}(x)\ell_{s}(z)^{q-2},\qquad \quad \tx{E}_{*}(x,\snr{z}):=\int_{0}^{\snr{z}}\lambda_{1;*}(x,t)t\dt.
\end{array}
$$
Finally, exponents $(\mu,q)$ verify
\eqn{mumu}
  $$
  1\le \mu<2\qquad \mbox{and}\qquad 1<q<2-\mu+\alpha\le 1+\alpha.
  $$
Straightforward computations, \eqref{assfr} and \eqref{ass.2} give
\begin{flalign}\label{assfr.2}
\left\{
\begin{array}{c}
\displaystyle
\ z\mapsto \tx{H}_{*}(x,z)\in C^{2}_{\loc}(\mathbb{R}^{n}),\qquad \quad x\mapsto \tx{H}_{*}(x,z)\in C^{0,\alpha}(\Omega)\\[8pt]\displaystyle
\ \ti{c}^{-1}a_{*}\ell_{1}(z)^{q}-\ti{c}\le \tx{H}_{*}(x,z)\le \ti{c}\ell_{1}(z)^{q}\\[8pt]\displaystyle
\ \ti{c}^{-1}a_{*}\ell_{1}(z)^{q-2}\snr{\xi}^{2}\le \langle\partial^{2}\tx{H}_{*}(x,z)\xi,\xi\rangle,\qquad \quad \snr{\partial^{2}\tx{H}_{*}(x,z)}\le\ti{c}\ell_{1}(z)^{q-2}\\[8pt]\displaystyle
\ \snr{\partial \tx{H}_{*}(x_{1},z)-\partial\tx{H}_{*}(x_{2},z)}\le q[a]_{0,\alpha;\Omega}\snr{x_{1}-x_{2}}^{\alpha}\ell_{1}(z)^{q-1},
\end{array}
\right.
\end{flalign}
for $\ti{c}\equiv \ti{c}(\Lambda_{*},\tx{g},q,\nr{a}_{L^{\infty}(\Omega)},s)\ge 1$. Within this framework, let $B\subset 2B\Subset \Omega$ be concentric balls, introduce functional
 $$
 W^{1,q}(B)\ni w\mapsto \mathcal{H}_{*}(w;B):=\int_{B}\tx{H}_{*}(x,Dw)\dx,
 $$
 obstacle $\psi$ as in \eqref{p2}, boundary datum
 \eqn{u0}
 $$u_{0}\in \tx{K}^{\psi}(B)\cap W^{1,\infty}(B),$$ and consider the solution $u\in (u_{0}+W^{1,q}_{0}(B))\cap \tx{K}^{\psi}(B)$ of problem
 \begin{eqnarray}\label{pdreg}
 (u_{0}+W^{1,q}_{0}(B))\cap \tx{K}^{\psi}(B)\ni w\mapsto \min \mathcal{H}_{*}(w;B).
 \end{eqnarray}
 Existence and uniqueness of $u\in (u_{0}+W^{1,q}_{0}(B))\cap \tx{K}^{\psi}(B)$ are granted by standard direct methods; by minimality the variational inequality
 \eqn{vel}
 $$
 \int_{B}\langle \partial \tx{H}_{*}(x,Du),Dw-Du\rangle\dx\ge 0\qquad \mbox{for all} \ \ w\in (u_{0}+W^{1,q}_{0}(B))\cap \tx{K}^{\psi}(B)
 $$
 holds true. Moreover by \eqref{assfr.2}, classical regularity theory for variational inequalities and obstacle problems \cite{fuc90,ch} assures that
 \eqn{areg}
 $$
 u\in C^{1}_{\loc}(B).
 $$
 \subsection{Global boundedness}\label{gb} Our first move is controlling the maximum modulus of $u$ by means of the largest between the $L^{\infty}$-norms of the obstacle and of the boundary datum. Letting
 $$\tx{S}:=\max\left\{1,\nr{u_{0}}_{L^{\infty}(B)},\nr{\psi}_{L^{\infty}(B)}\right\},$$ via elementary considerations we have that $w:=u-(u-\tx{S})_{+}$ is admissible in \eqref{vel}, in particular $(u-\tx{S})_{+}\in W^{1,q}_{0}(B)$, therefore recalling Remark \ref{r21} we get
 \begin{eqnarray*}
 0&\stackrel{\eqref{vel}}{\le}&-\int_{B}\langle\partial \tx{H}_{*}(x,Du),D(u-\tx{S})_{+}\rangle\dx\le -\int_{B}\langle\partial\tx{L}_{*}(Du),D(u-\tx{S})_{+}\rangle\dx\nonumber \\
 &=&-\int_{B}\int_{0}^{1}\langle\partial^{2}\tx{L}_{*}(tDu) Du,D(u-\tx{S})_{+}\rangle\dtt\dx\stackrel{\eqref{assfr}}{\le}-\frac{1}{\Lambda_{*}}\int_{B}\frac{\snr{D(u-\tx{S})_{+}}^{2}}{(1+\snr{Du}^{2})^{\mu/2}}\dx,
 \end{eqnarray*}
 which implies that $u(x)\le \tx{S}$ for any $x\in B$. On the other hand, $u(x)\ge \psi(x)\ge -\tx{S}$ for all $x\in B$, so we can conclude that $\nr{u}_{L^{\infty}(B)}\le \tx{S}$. 
 \subsection{Zero order scaling}\label{sca} We fix a ball $B_{\rr}(x_{0})\subset B_{2\rr}(x_{0})\Subset B$ with $\rr\in (0,1]$ and blow up $u$, $\psi$ and $\tx{a}$ on $B_{\rr}(x_{0})$ by defining the maps $\ti{u}(x):=(8\tx{S})^{-1}u(x_{0}+\rr x)$, $\ti{\psi}(x):=(8\tx{S})^{-1}\psi(x_{0}+\rr x)$, $\ti{\tx{a}}(x):=\tx{a}(x_{0}+\rr x)$. We further introduce constant $\tx{C}_{\rr}:=8\tx{S}/\rr\ge 1$, integrands $\ti{\tx{L}}(z):=\tx{L}_{*}(\tx{C}_{\rr}z)$, $\ti{\ell}_{s}(z):=\ell_{s}(\tx{C}_{\rr}z)$, $\ti{\tx{H}}(x,z):=\ti{\tx{L}}(z)+\ti{\tx{a}}(x)\ti{\ell}_{s}(z)^{q}$, and observe that via $\eqref{assfr}_{3}$-\eqref{d.1} we have
 \eqn{scasca}
 $$
 \left\{
 \begin{array}{c}
 \displaystyle
 \ \snr{\partial \ti{\tx{H}}(x,z)}\lesssim_{\Lambda_{*},\tx{g},q}\tx{C}_{\rr}^{2}\left(1+\tx{g}(\snr{z})\right)+\tx{C}_{\rr}^{q}\ti{\tx{a}}(x)\ell_{s}(z)^{q-1}\\[8pt]\displaystyle
 \ \snr{\partial^{2}\ti{\tx{H}}(x,z)}\lesssim_{n,\Lambda_{*},q}\frac{\tx{C}_{\rr}^{3}\left(1+\tx{g}(\snr{z})\right)}{(1+\snr{z}^{2})^{1/2}}+\tx{C}_{\rr}^{2}\ti{\tx{a}}(x)\ell_{s}(z)^{q-2}\\[8pt]\displaystyle
 \ \langle\partial^{2}\ti{\tx{H}}(x,z)\xi,\xi\rangle\gtrsim_{\mu,q}\frac{\tx{C}_{\rr}^{2-\mu}\snr{\xi}^{2}}{(1+\snr{z}^{2})^{\mu/2}}+\tx{C}_{\rr}^{q}\ti{\tx{a}}(x)\ell_{s}(z)^{q-2}\snr{\xi}^{2},
 \end{array}
 \right.
 $$
 for all $z,\xi\in \mathbb{R}^{n}$. By scaling we see that $\ti{u}\in \tx{K}^{\ti{\psi}}(B_{1}(0))$ is a constrained local minimizer of functional
 \eqn{fbw}
 $$
 W^{1,q}(B_{1}(0))\ni w\mapsto \int_{B_{1}(0)}\ti{\tx{H}}(x,Dw)\dx,
 $$
 solving the variational inequality
 \eqn{vler}
 $$
 \int_{B_{1}(0)}\langle\partial \ti{\tx{H}}(x,D\ti{u}),Dw-D\ti{u}\rangle\dx\ge 0\qquad \mbox{for all} \ \ w\in (\ti{u}+W^{1,q}_{0}(B_{1}(0)))\cap \tx{K}^{\ti{\psi}}(B_{1}(0)).
 $$
 Now we are ready to build up our iteration scheme.

  \subsection{Linear iteration}\label{fm} With \eqref{assfr}, \eqref{mumu}, \eqref{ass.2} and \eqref{p2} in force, we introduce radii $0<r_{2}<r_{1}\le 1$, parameters $r_{2}\le \tau_{2}<\tau_{1}\le r_{1}$, set $\rr_{0}:=\tau_{2}+(\tau_{1}-\tau_{2})2^{-2}$, $\rr_{1}:=\tau_{2}+(\tau_{1}-\tau_{2})2^{-3}$, and pick other two numbers $\rr_{1}\le \ti{\tau}_{2}<\ti{\tau}_{1}\le \rr_{0}$. Then let $\varsigma_{0}:=\ti{\tau}_{2}+(\ti{\tau}_{1}-\ti{\tau}_{2})2^{-2}$, $\varsigma_{1}:=\ti{\tau}_{2}+(\ti{\tau}_{1}-\ti{\tau}_{2})2^{-3}$, $\ti{r}_{0}:=(\varsigma_{1}+\varsigma_{0})/2$, $\bar{r}_{0}:=(3\varsigma_{1}+\varsigma_{0})/4$, $\bar{r}_{1}:=(5\varsigma_{1}+\varsigma_{0})/6$ and $\ti{r}_{1}:=(7\varsigma_{1}+\varsigma_{0})/8$. It is easy to see that $\rr_{0}\ge\ti{\tau}_{1}>\ti{r}_{0}>\bar{r}_{0}>\bar{r}_{1}>\ti{r}_{1}>\ti{\tau}_{2}\ge \rr_{1}$. Further, pick a vector $h\in \mathbb{R}^{n}\setminus \{0\}$ such that $\snr{h}\in (0,\mathcal{h}_{0})$ with $\mathcal{h}_{0}:=2^{-17}(\ti{\tau}_{1}-\ti{\tau}_{2})$ and a cut-off function $\eta\in C^{2}_{0}(B_{1}(0))$ satisfying 
  \eqn{etaeta}
  $$
  \mathds{1}_{B_{\bar{r}_{1}}(0)}\le \eta\le \mathds{1}_{B_{\bar{r}_{0}}(0)},\qquad \quad \snr{D\eta}\le 2^{6}/(\bar{r}_{0}-\bar{r}_{1}),\qquad \quad \snr{D^{2}\eta}\le 2^{12}/(\bar{r}_{0}-\bar{r}_{1})^{2}.
  $$
  A quick notation remark: since all balls we shall consider now are concentric in the origin, we will omit explicitly denoting their center. Next, for reasons that will be clear in a few lines, we set
  \eqn{thre}
  $$
  t_{*}:=2q+8,\qquad \qquad \quad \omega_{*}:=\min\left\{\alpha+2-\mu-q,\frac{1}{2}\right\},
  $$
 take 
 \eqn{kkk}
  $$
 \begin{array}{c}
 \displaystyle
\omega\in \left(0,\omega_{*}\right),\qquad \quad \ti{\delta}:=1-\mu\omega,\qquad \quad \theta:=2\omega+\ti{\delta}-1,\\[8pt]\displaystyle
\ti{\gamma}:=\max\{\theta,\mu-1\},\qquad \quad \gamma_{\mu;\omega}:=\frac{2\max\{\mu-1,\omega\}}{1+\alpha-q},
 \end{array}
 $$ and, with $t\ge t_{*}$ introduce quantities
 \eqn{ts}
 $$
 \begin{array}{c}
 \displaystyle
  \sigma:=t-2q+3-\mu+2\alpha-2\omega,\qquad \quad  \vartheta=\frac{\sigma-2-\alpha+2\omega}{\sigma},\\[10pt]\displaystyle
 \ \tx{s}_{*}:=2-\mu+\alpha-q-\frac{\theta}{2}.
 \end{array}
 $$
 We point out that being $t\ge t_{*}$ and $0<\omega<\omega_{*}\le \alpha$, it is $\sigma>9$ and $\vartheta\in (0,1)$ in \eqref{ts}. Moreover, a direct computation and $\eqref{mumu}_{1}$ show that $\ti{\delta}\in (0,1)$ and $\theta\in (0,2-\mu)$ in \eqref{kkk}. At this stage, we need to look for a reasonable test function to plug in \eqref{vler}. To this aim, we preliminary observe that given any function $\tx{b}\in L^{\infty}(B_{1})$ such that $0\le \tx{b}\le 1$, map $w:=\ti{u}-2^{-1}\tau_{-h}(\tx{b}\tau_{h}(\ti{u}-\ti{\psi}))$ satisfies the obstacle constraint. In fact it is
 \begin{eqnarray}\label{obc}
 w(x)&=&\ti{u}(x)+\frac{1}{2}\tx{b}(x)(\ti{u}(x+h)-\ti{\psi}(x+h))+\frac{1}{2}\tx{b}(x-h)(\ti{u}(x-h)-\ti{\psi}(x-h))\nonumber \\
 &&-\frac{1}{2}(\tx{b}(x)+\tx{b}(x-h))(\ti{u}(x)-\ti{\psi}(x))\nonumber \\
 &\ge&\frac{1}{2}\left((1-\tx{b}(x))+(1-\tx{b}(x-h))\right)\ti{u}(x)+\frac{1}{2}\left(\tx{b}(x)+\tx{b}(x-h)\right)\ti{\psi}(x)\stackrel{\ti{u}\ge \ti{\psi}}{\ge}\ti{\psi}(x).
 \end{eqnarray}
 In \eqref{obc} we choose $$\tx{b}\equiv \tx{b}_{\sigma}:=\eta^{2}\int_{0}^{1}\snr{\tau_{h}\ti{\psi}+\beta\tau_{h}(\ti{u}-\ti{\psi})}^{\sigma-1}\d\beta.$$ Notice that $\eqref{etaeta}_{1}$, the restrictions imposed on the size of $\mathcal{h}_{0}$, and the content of sections \ref{gb}-\ref{sca} give 
\begin{eqnarray}\label{bdd}
 \max\left\{\nr{\tau_{h}\ti{u}}_{L^{\infty}(B_{1})},\nr{\tau_{h}\ti{\psi}}_{L^{\infty}(B_{1})}\right\}&\le& 2\max\left\{\nr{\ti{u}}_{L^{\infty}(B_{2})},\nr{\ti{\psi}}_{L^{\infty}(B_{2})}\right\}\nonumber \\
 &\le&(4\tx{S})^{-1}\max\left\{\nr{u}_{L^{\infty}(B_{2\rr})},\nr{\psi}_{L^{\infty}(B_{2\rr})}\right\}\le \frac{1}{4}\nonumber \\
 & \Longrightarrow&   0\le \tx{b}_{\sigma}(x)\stackrel{\sigma> 1}{\le} \left(\nr{\tau_{h}\ti{\psi}}_{L^{\infty}(B_{1})}+\nr{\tau_{h}\ti{u}}_{L^{\infty}(B_{1})}\right)^{\sigma-1}\le 1,
\end{eqnarray}
for all $x\in B_{1}$, therefore recalling the mean value theorem and that $\supp(\eta)\Subset B_{\ti{r}_{0}}\subset B_{1}$, we can conclude with 
 \begin{eqnarray*}
 w&:=&\ti{u}-\frac{1}{2}\tau_{-h}\left(\sigma^{-1}\eta^{2}\left(\snr{\tau_{h}\ti{u}}^{\sigma-1}\tau_{h}\ti{u}-\snr{\tau_{h}\ti{\psi}}^{\sigma-1}\tau_{h}\ti{\psi}\right)\right)\nonumber \\
 &=&\ti{u}-\frac{1}{2}\tau_{-h}(\tx{b}_{\sigma}\tau_{h}(\ti{u}-\ti{\psi}))\stackrel{\eqref{etaeta}}{\in} \left(\ti{u}+W^{1,q}_{0}(B_{1})\right)\cap \tx{K}^{\ti{\psi}}(B_{1}),
 \end{eqnarray*}
 and, thanks also to \eqref{p2} and \eqref{areg}, testing in \eqref{vler} is legit. We indeed get, after applying the integration-by-parts formula for finite difference operators, getting rid of factor $2\sigma>0$, and multiplying the outcome by $\snr{h}^{-1-\vartheta\sigma-\alpha}$:
  \begin{eqnarray*}
\mbox{(I)}+\mbox{(II)}&=:& \frac{1}{\snr{h}^{1+\vartheta\sigma+\alpha}}\int_{B_{1}}\langle\tau_{h}\partial\ti{\tx{L}}(D\ti{u}),D(\eta^{2}\snr{\tau_{h}\ti{u}}^{\sigma-1}\tau_{h}\ti{u})\rangle\dx\nonumber \\
&&+\frac{q\tx{C}_{\rr}^{2}}{\snr{h}^{1+\vartheta\sigma+\alpha}}\int_{B_{1}}\langle\tau_{h}(\ti{\tx{a}}(\cdot)\ti{\ell}_{s}(D\ti{u})^{q-2}D\ti{u}),D(\eta^{2}\snr{\tau_{h}\ti{u}}^{\sigma-1}\tau_{h}\ti{u})\rangle\dx\nonumber \\
&\le&\frac{1}{\snr{h}^{1+\vartheta\sigma+\alpha}}\int_{B_{1}}\langle\tau_{h}\partial\ti{\tx{L}}(D\ti{u}),D(\eta^{2}\snr{\tau_{h}\ti{\psi}}^{\sigma-1}\tau_{h}\ti{\psi})\rangle\dx\nonumber \\
&&+\frac{q\tx{C}_{\rr}^{2}}{\snr{h}^{1+\vartheta\sigma+\alpha}}\int_{B_{1}}\langle\tau_{h}(\ti{\tx{a}}(\cdot)\ti{\ell}_{s}(D\ti{u})^{q-2}D\ti{u}),D(\eta^{2}\snr{\tau_{h}\ti{\psi}}^{\sigma-1}\tau_{h}\ti{\psi})\rangle\dx=:\mbox{(III)}+\mbox{(IV)}.
 \end{eqnarray*}
 Before proceeding further, let us introduce some abbreviations. With $h$ and $\ti{u}$ as above, we set
 \begin{flalign*}
 \begin{cases}
     \ \displaystyle\mathcal{D}(x,h):=1+\snr{D\ti{u}(x)}^{2}+\snr{D\ti{u}(x+h)}^{2},\qquad \quad &\displaystyle\tx{A}_{\rr}:=1+\nr{\ti{\tx{a}}}_{C^{0,\alpha}(B_{1})}\vspace{2.5mm}\\
     \ \displaystyle \mathcal{N}_{\infty}:=1+\nr{D\eta}_{L^{\infty}(B_{\bar{r}_{0}})}^{2}+\nr{D^{2}\eta}_{L^{\infty}(B_{\bar{r}_{0}})},\qquad \quad &\displaystyle\Psi:=1+\nr{D\ti{\psi}}_{L^{\infty}(B_{1})}+\nr{D^{2}\ti{\psi}}_{L^{\infty}(B_{1})}\vspace{2.5mm}\\
     \ \displaystyle \mathcal{G}_{1}(x,h):=\int_{0}^{1}\snr{\partial\ti{\tx{L}}(D\ti{u}(x+\beta h))}\d\beta,\qquad \quad & \displaystyle\mathcal{G}_{q}(x,h):=\int_{0}^{1}\ti{\ell}_{s}(D\ti{u}(x+\beta h))^{q-1}\d\beta,
 \end{cases}
 \end{flalign*}
and estimate by the mean value theorem, \eqref{af}, \eqref{scasca}, \eqref{ass.2}, and Cauchy-Schwarz, Young and Jensen inequalities,
\begin{eqnarray*}
 \mbox{(I)}&=&\frac{\sigma}{\snr{h}^{1+\vartheta\sigma+\alpha}}\int_{B_{1}}\int_{0}^{1}\eta^{2}\snr{\tau_{h}\ti{u}}^{\sigma-1}\langle\partial^{2}\ti{\tx{L}}(D\ti{u}+\beta\tau_{h}D\ti{u})\tau_{h}D\ti{u},\tau_{h}D\ti{u}\rangle\d\beta\dx\nonumber \\
 &&-\frac{2}{\snr{h}^{\vartheta\sigma+\alpha}}\int_{B_{1}}\int_{0}^{1}\langle\partial\ti{\tx{L}}(D\ti{u}(x+\beta h)),\partial_{h/\snr{h}}(\eta D\eta\snr{\tau_{h}\ti{u}}^{\sigma-1}\tau_{h}\ti{u})\rangle\d\beta\dx\nonumber \\
 &\ge&\frac{\sigma}{\snr{h}^{1+\vartheta\sigma+\alpha}}\int_{B_{1}}\int_{0}^{1}\eta^{2}\snr{\tau_{h}\ti{u}}^{\sigma-1}\langle\partial^{2}\ti{\tx{L}}(D\ti{u}+\beta\tau_{h}D\ti{u})\tau_{h}D\ti{u},\tau_{h}D\ti{u}\rangle\d\beta\dx\nonumber \\
 &&-\frac{c\mathcal{N}_{\infty}}{\snr{h}^{\vartheta\sigma+\alpha}}\int_{B_{\bar{r}_{0}}}\mathcal{G}_{1}(x,h)\snr{\tau_{h}\ti{u}}^{\sigma}\dx-\frac{c\sigma\mathcal{N}_{\infty}}{\varepsilon\tx{C}_{\rr}^{2-\mu}\snr{h}^{\vartheta\sigma+\alpha-1}}\int_{B_{\bar{r}_{0}}}\snr{\tau_{h}\ti{u}}^{\sigma-1}\mathcal{G}_{1}(x,h)^{2}\mathcal{D}(x,h)^{\frac{\mu}{2}}\dx\nonumber \\
 &&-\frac{c\varepsilon\sigma\tx{C}_{\rr}^{2-\mu}}{\snr{h}^{1+\vartheta\sigma+\alpha}}\int_{B_{1}}\eta^{2}\snr{\tau_{h}\ti{u}}^{\sigma-1}\mathcal{D}(x,h)^{-\frac{\mu}{2}}\snr{\tau_{h}D\ti{u}}^{2}\dx\nonumber \\
 &\ge&\frac{\sigma\tx{C}_{\rr}^{2-\mu}}{c\snr{h}^{1+\vartheta\sigma+\alpha}}\int_{B_{1}}\int_{0}^{1}\frac{\eta^{2}\snr{\tau_{h}\ti{u}}^{\sigma-1}\snr{\tau_{h}D\ti{u}}^{2}}{(1+\snr{D\ti{u}+\beta\tau_{h}D\ti{u}}^{2})^{\mu/2}}\d\beta\dx\nonumber \\
  &&-\frac{c\varepsilon\sigma\tx{C}_{\rr}^{2-\mu}}{\snr{h}^{1+\vartheta\sigma+\alpha}}\int_{B_{1}}\eta^{2}\snr{\tau_{h}\ti{u}}^{\sigma-1}\mathcal{D}(x,h)^{-\frac{\mu}{2}}\snr{\tau_{h}D\ti{u}}^{2}\dx\nonumber \\
 &&-\frac{c\mathcal{N}_{\infty}}{\snr{h}^{\vartheta\sigma+\alpha}}\left(\int_{B_{\bar{r}_{0}}}\mathcal{G}_{1}(x,h)^{\frac{t}{q-1}}\dx\right)^{\frac{q-1}{t}}\left(\int_{B_{\bar{r}_{0}}}\snr{\tau_{h}\ti{u}}^{\frac{t\sigma}{t+1-q}}\dx\right)^{\frac{t+1-q}{t}}\nonumber \\
 &&-\frac{c\sigma\mathcal{N}_{\infty}}{\varepsilon\tx{C}_{\rr}^{2-\mu}\snr{h}^{\vartheta\sigma+\alpha-1}}\int_{B_{\bar{r}_{0}}}\mathcal{G}_{1}(x,h)^{\frac{2q-2+\mu}{q-1}}\snr{\tau_{h}\ti{u}}^{\sigma-1}\dx\nonumber \\
 &&-\frac{c\sigma\mathcal{N}_{\infty}}{\varepsilon\tx{C}_{\rr}^{2-\mu}\snr{h}^{\vartheta\sigma+\alpha-1}}\int_{B_{\bar{r}_{0}}}\mathcal{D}(x,h)^{\frac{2q-2+\mu}{2}}\snr{\tau_{h}\ti{u}}^{\sigma-1}\dx\nonumber \\
 &\ge&\frac{\sigma\tx{C}_{\rr}^{2-\mu}}{\snr{h}^{1+\vartheta\sigma+\alpha}}\left(\frac{1}{c}-c\varepsilon\right)\int_{B_{1}}\eta^{2}\snr{\tau_{h}\ti{u}}^{\sigma-1}\mathcal{D}(x,h)^{-\frac{\mu}{2}}\snr{\tau_{h}D\ti{u}}^{2}\dx\nonumber \\
 &&-\frac{c\tx{C}_{\rr}^{2}\mathcal{N}_{\infty}}{\snr{h}^{\vartheta\sigma+\alpha}}\left(\int_{B_{\ti{r}_{0}}}\ell_{1}(D\ti{u})^{t}\dx\right)^{\frac{q-1}{t}}\left(\int_{B_{\bar{r}_{0}}}\snr{\tau_{h}\ti{u}}^{\frac{t\sigma}{t+1-q}}\dx\right)^{\frac{t+1-q}{t}}\nonumber \\
 &&-\frac{c\sigma\tx{C}_{\rr}^{\frac{4q}{q-1}}\mathcal{N}_{\infty}}{\varepsilon\snr{h}^{\vartheta\sigma+\alpha-1}}\left(\int_{B_{\ti{r}_{0}}}\ell_{1}(D\ti{u})^{t}\dx\right)^{\frac{2q-2+\mu}{t}}\left(\int_{B_{\bar{r}_{0}}}\snr{\tau_{h}\ti{u}}^{\frac{t(\sigma-1)}{t+2-\mu-2q}}\dx\right)^{\frac{t+2-\mu-2q}{t}},
 \end{eqnarray*}
for $c\equiv c(\Lambda_{*},\tx{g},\mu,q)$. A suitably small choice of $\varepsilon>0$ eventually yields
\begin{eqnarray*}
\mbox{(I)}&\ge&\frac{\sigma\tx{C}_{\rr}^{2-\mu}}{c\snr{h}^{1+\vartheta\sigma+\alpha}}\int_{B_{1}}\eta^{2}\snr{\tau_{h}\ti{u}}^{\sigma-1}\mathcal{D}(x,h)^{-\frac{\mu}{2}}\snr{\tau_{h}D\ti{u}}^{2}\dx\nonumber \\
&&-\frac{c\tx{C}_{\rr}^{2}\mathcal{N}_{\infty}}{\snr{h}^{\vartheta\sigma+\alpha}}\left(\int_{B_{\ti{r}_{0}}}\ell_{1}(D\ti{u})^{t}\dx\right)^{\frac{q-1}{t}}\left(\int_{B_{\bar{r}_{0}}}\snr{\tau_{h}\ti{u}}^{\frac{t\sigma}{t+1-q}}\dx\right)^{\frac{t+1-q}{t}}\nonumber \\
&&-\frac{c\sigma\tx{C}_{\rr}^{\frac{4q}{q-1}}\mathcal{N}_{\infty}}{\snr{h}^{\vartheta\sigma+\alpha-1}}\left(\int_{B_{\ti{r}_{0}}}\ell_{1}(D\ti{u})^{t}\dx\right)^{\frac{2q-2+\mu}{t}}\left(\int_{B_{\bar{r}_{0}}}\snr{\tau_{h}\ti{u}}^{\frac{t(\sigma-1)}{t+2-\mu-2q}}\dx\right)^{\frac{t+2-\mu-2q}{t}}\nonumber \\
&=&\frac{\sigma\tx{C}_{\rr}^{2-\mu}\mbox{(I)}_{1}}{c\snr{h}^{1+\vartheta\sigma+\alpha}}-\frac{c\tx{C}_{\rr}^{2}\mathcal{N}_{\infty}}{\snr{h}^{\vartheta\sigma+\alpha}}\left(\int_{B_{\ti{r}_{0}}}\ell_{1}(D\ti{u})^{t}\dx\right)^{\frac{q-1}{t}}\left(\int_{B_{\bar{r}_{0}}}\snr{\tau_{h}\ti{u}}^{\frac{t\sigma}{t+1-q}}\dx\right)^{\frac{t+1-q}{t}}\nonumber \\
&&-\frac{c\sigma\tx{C}_{\rr}^{\frac{4q}{q-1}}\mathcal{N}_{\infty}}{\snr{h}^{\vartheta\sigma+\alpha-1}}\left(\int_{B_{\ti{r}_{0}}}\ell_{1}(D\ti{u})^{t}\dx\right)^{\frac{2q-2+\mu}{t}}\left(\int_{B_{\bar{r}_{0}}}\snr{\tau_{h}\ti{u}}^{\frac{t(\sigma-1)}{t+2-\mu-2q}}\dx\right)^{\frac{t+2-\mu-2q}{t}},
\end{eqnarray*}
where we set
$$
\mbox{(I)}_{1}:=\int_{B_{1}}\eta^{2}\snr{\tau_{h}\ti{u}}^{\sigma-1}\mathcal{D}(x,h)^{-\frac{\mu}{2}}\snr{\tau_{h}D\ti{u}}^{2}\dx,
$$
and it is $c\equiv c(\Lambda_{*},\tx{g},\mu,q)$. Concerning term $\mbox{(II)}$, by \eqref{prod} we split:
 \begin{eqnarray*}
 \mbox{(II)}&=&\frac{q\sigma\tx{C}_{\rr}^{2}}{\snr{h}^{1+\vartheta\sigma+\alpha}}\int_{B_{1}}\eta^{2}\ti{\tx{a}}(x)\snr{\tau_{h}\ti{u}}^{\sigma-1}\langle\tau_{h}(\ti{\ell}_{s}(D\ti{u})^{q-2}D\ti{u}),\tau_{h}D\ti{u}\rangle\dx\nonumber \\
 &&+\frac{q\sigma\tx{C}_{\rr}^{2}}{\snr{h}^{1+\vartheta\sigma+\alpha}}\int_{B_{1}}\eta^{2}(\tau_{h}\ti{\tx{a}}(\cdot))\snr{\tau_{h}\ti{u}}^{\sigma-1}\langle\ti{\ell}_{s}(D\ti{u}(x+h))^{q-2}D\ti{u}(x+h),\tau_{h}D\ti{u}\rangle\dx\nonumber \\
 &&+\frac{2q\tx{C}_{\rr}^{2}}{\snr{h}^{1+\vartheta\sigma+\alpha}}\int_{B_{1}}\langle\tau_{h}(\ti{\tx{a}}(\cdot)\ti{\ell}_{s}(D\ti{u})^{q-2}D\ti{u}),(\eta\snr{\tau_{h}\ti{u}}^{\sigma-1}\tau_{h}\ti{u})D\eta\rangle\dx=:\mbox{(II)}_{1}+\mbox{(II)}_{2}+\mbox{(II)}_{3},
 \end{eqnarray*}
 and control via H\"older inequality with conjugate exponents $\left(\frac{t}{2q-2+\mu},\frac{t}{t-2q+2-\mu}\right)$, Cauchy-Schwarz and Young inequalities and standard monotonicity properties of power type integrands,
 \begin{eqnarray*}
 \mbox{(II)}_{1}+\mbox{(II)}_{2}&\ge& \frac{\sigma\tx{C}_{\rr}^{q}}{c\snr{h}^{1+\vartheta\sigma+\alpha}}\int_{B_{1}}\eta^{2}\ti{\tx{a}}(x)\snr{\tau_{h}\ti{u}}^{\sigma-1}\mathcal{D}(x,h)^{\frac{q-2}{2}}\snr{\tau_{h}D\ti{u}}^{2}\dx\nonumber \\
 &&-\frac{c\sigma\tx{C}_{\rr}[\ti{\tx{a}}]_{0,\alpha;B_{1}}}{\snr{h}^{1+\vartheta\sigma}}\int_{B_{1}}\eta^{2}\snr{\tau_{h}\ti{u}}^{\sigma-1}\ti{\ell}_{s}(D\ti{u}(x+h))^{q-1}\snr{\tau_{h}D\ti{u}}\dx\nonumber \\
 &\ge&-\frac{\varepsilon\sigma\tx{C}_{\rr}^{2-\mu}\mbox{(I)}_{1}}{\snr{h}^{1+\vartheta\sigma+\alpha}}-\frac{c\sigma\tx{C}_{\rr}^{2q+\mu-2}[\ti{\tx{a}}]_{0,\alpha;B_{1}}^{2}}{\varepsilon\snr{h}^{1+\vartheta\sigma-\alpha}}\int_{B_{1}}\eta^{2}\snr{\tau_{h}\ti{u}}^{\sigma-1}\mathcal{D}(x,h)^{\frac{2q-2+\mu}{2}}\dx\nonumber \\
 &\ge&-\frac{c\sigma\tx{C}_{\rr}^{2q+\mu-2}\tx{A}_{\rr}^{2}}{\varepsilon\snr{h}^{1+\vartheta\sigma-\alpha}}\left(\int_{B_{\ti{r}_{0}}}\ell_{1}(D\ti{u})^{t}\dx\right)^{\frac{2q-2+\mu}{t}}\left(\int_{B_{\bar{r}_{0}}}\snr{\tau_{h}\ti{u}}^{\frac{t(\sigma-1)}{t+2-2q-\mu}}\dx\right)^{\frac{t+2-2q-\mu}{t}}\nonumber \\
 &&-\frac{\varepsilon\sigma\tx{C}_{\rr}^{2-\mu}\mbox{(I)}_{1}}{\snr{h}^{1+\vartheta\sigma+\alpha}},
 \end{eqnarray*}
 for $c\equiv c(\Lambda_{*},\tx{g},\mu,q)$ and sufficiently small $\varepsilon>0$. Using also Young inequality with conjugate exponents $\left(\frac{2q-2+\mu}{\mu},\frac{2q-2+\mu}{2q-2}\right)$ and \eqref{af}, we further have
 \begin{eqnarray*}
 \mbox{(II)}_{3}&=&-\frac{2q\tx{C}_{\rr}^{2}}{\snr{h}^{\vartheta\sigma+\alpha}}\int_{B_{1}}\int_{0}^{1}\ti{\tx{a}}(x+\beta h)\langle \ti{\ell}_{s}(D\ti{u}(x+\beta h))^{q-2}D\ti{u}(x+\beta h),\partial_{h/\snr{h}}(\eta \snr{\tau_{h}\ti{u}}^{\sigma-1}\tau_{h}\ti{u}D\eta)\rangle\d\beta\dx\nonumber \\
 &\ge&-\frac{c\tx{C}_{\rr}\nr{\ti{\tx{a}}}_{L^{\infty}(B_{1})}}{\snr{h}^{\vartheta\sigma+\alpha}}\int_{B_{1}}\mathcal{G}_{q}(x,h)(\snr{D\eta}^{2}+\eta\snr{D^{2}\eta})\snr{\tau_{h}\ti{u}}^{\sigma}\dx\nonumber \\
 &&-\frac{c\sigma\tx{C}_{\rr}\nr{\ti{\tx{a}}}_{L^{\infty}(B_{1})}}{\snr{h}^{\vartheta\sigma+\alpha}}\int_{B_{1}}\mathcal{G}_{q}(x,h)\eta\snr{D\eta}\snr{\tau_{h}\ti{u}}^{\sigma-1}\snr{\tau_{h}D\ti{u}}\dx\nonumber \\
 &\ge&-\frac{c\tx{C}_{\rr}\mathcal{N}_{\infty}\tx{A}_{\rr}}{\snr{h}^{\vartheta\sigma+\alpha}}\left(\int_{B_{\bar{r}_{0}}}\mathcal{G}_{q}(x,h)^{\frac{t}{q-1}}\dx\right)^{\frac{q-1}{t}}\left(\int_{B_{\bar{r}_{0}}}\snr{\tau_{h}\ti{u}}^{\frac{t\sigma}{t+1-q}}\dx\right)^{\frac{t+1-q}{t}}\nonumber \\
 &&-\frac{c\sigma\tx{C}_{\rr}^{\mu}\mathcal{N}_{\infty}\tx{A}_{\rr}^{2}}{\varepsilon\snr{h}^{\vartheta\sigma+\alpha-1}}\int_{B_{\bar{r}_{0}}}\mathcal{G}_{q}(x,h)^{2}\snr{\tau_{h}\ti{u}}^{\sigma-1}\mathcal{D}(x,h)^{\frac{\mu}{2}}\dx-\frac{\varepsilon\sigma\tx{C}_{\rr}^{2-\mu}\mbox{(I)}_{1}}{\snr{h}^{1+\vartheta\sigma+\alpha}}\nonumber \\
 &\ge&-\frac{c\tx{C}_{\rr}^{q}\mathcal{N}_{\infty}\tx{A}_{\rr}}{\snr{h}^{\vartheta\sigma+\alpha}}\left(\int_{B_{\ti{r}_{0}}}\ell_{1}(D\ti{u})^{t}\dx\right)^{\frac{q-1}{t}}\left(\int_{B_{\bar{r}_{0}}}\snr{\tau_{h}\ti{u}}^{\frac{t\sigma}{t+1-q}}\dx\right)^{\frac{t+1-q}{t}}-\frac{\varepsilon\sigma\tx{C}_{\rr}^{2-\mu}\mbox{(I)}_{1}}{\snr{h}^{1+\vartheta\sigma+\alpha}}\nonumber \\
 &&-\frac{c\sigma\tx{C}_{\rr}^{\mu}\mathcal{N}_{\infty}\tx{A}_{\rr}^{2}}{\varepsilon\snr{h}^{\vartheta\sigma+\alpha-1}}\int_{B_{\bar{r}_{0}}}\mathcal{G}_{q}(x,h)^{\frac{2q-2+\mu}{q-1}}\snr{\tau_{h}\ti{u}}^{\sigma-1}\dx-\frac{c\sigma\tx{C}_{\rr}^{\mu}\mathcal{N}_{\infty}\tx{A}_{\rr}^{2}}{\varepsilon\snr{h}^{\vartheta\sigma+\alpha-1}}\int_{B_{\bar{r}_{0}}}\mathcal{D}(x,h)^{\frac{2q-2+\mu}{2}}\snr{\tau_{h}\ti{u}}^{\sigma-1}\dx\nonumber \\
 &\ge&-\frac{c\tx{C}_{\rr}^{q}\mathcal{N}_{\infty}\tx{A}_{\rr}}{\snr{h}^{\vartheta\sigma+\alpha}}\left(\int_{B_{\ti{r}_{0}}}\ell_{1}(D\ti{u})^{t}\dx\right)^{\frac{q-1}{t}}\left(\int_{B_{\bar{r}_{0}}}\snr{\tau_{h}\ti{u}}^{\frac{t\sigma}{t+1-q}}\dx\right)^{\frac{t+1-q}{t}}-\frac{\varepsilon\sigma\tx{C}_{\rr}^{2-\mu}\mbox{(I)}_{1}}{\snr{h}^{1+\vartheta\sigma+\alpha}}\nonumber \\
 &&-\frac{c\sigma\tx{C}_{\rr}^{2(q-1+\mu)}\mathcal{N}_{\infty}\tx{A}_{\rr}^{2}}{\varepsilon\snr{h}^{\vartheta\sigma+\alpha-1}}\left(\int_{B_{\ti{r}_{0}}}\ell_{1}(D\ti{u})^{t}\dx\right)^{\frac{2q-2+\mu}{t}}\left(\int_{B_{\bar{r}_{0}}}\snr{\tau_{h}\ti{u}}^{\frac{t(\sigma-1)}{t+2-2q-\mu}}\dx\right)^{\frac{t+2-2q-\mu}{t}},
 \end{eqnarray*}
 with $c\equiv c(\Lambda_{*},\tx{g},\mu,q)$, and again for suitably small $\varepsilon>0$. Overall, we have
 \begin{eqnarray*}
 \mbox{(II)}&\ge&-\frac{2\varepsilon\sigma\tx{C}_{\rr}^{2-\mu}\mbox{(I)}_{1}}{\snr{h}^{1+\vartheta\sigma+\alpha}}-\frac{c\tx{C}_{\rr}^{q}\mathcal{N}_{\infty}\tx{A}_{\rr}}{\snr{h}^{\vartheta\sigma+\alpha}}\left(\int_{B_{\ti{r}_{0}}}\ell_{1}(D\ti{u})^{t}\dx\right)^{\frac{q-1}{t}}\left(\int_{B_{\bar{r}_{0}}}\snr{\tau_{h}\ti{u}}^{\frac{t\sigma}{t+1-q}}\dx\right)^{\frac{t+1-q}{t}}\nonumber \\
 &&-\frac{c\sigma\tx{C}_{\rr}^{2(q-1+\mu)}\mathcal{N}_{\infty}\tx{A}_{\rr}^{2}}{\varepsilon\snr{h}^{1+\vartheta\sigma-\alpha}}\left(\int_{B_{\ti{r}_{0}}}\ell_{1}(D\ti{u})^{t}\dx\right)^{\frac{2q-2+\mu}{t}}\left(\int_{B_{\bar{r}_{0}}}\snr{\tau_{h}\ti{u}}^{\frac{t(\sigma-1)}{t+2-2q-\mu}}\dx\right)^{\frac{t+2-2q-\mu}{t}},
 \end{eqnarray*}
 where it is $c\equiv c(\Lambda_{*},\tx{g},\mu,q)$, and $\varepsilon>0$ still needs to be fixed. Next, we take care of the obstacle term. We control using \eqref{p2}, $\eqref{scasca}_{1}$, \eqref{ts}, and \eqref{ass.2},
 \begin{eqnarray*}
 \snr{\mbox{(III)}}&\le&\frac{\sigma}{\snr{h}^{1+\vartheta\sigma+\alpha}}\int_{B_{\bar{r}_{0}}}\left(\snr{\partial\ti{\tx{L}}(D\ti{u}(x+h))}+\snr{\partial \ti{\tx{L}}(D\ti{u}(x))}\right)\snr{\tau_{h}\ti{\psi}}^{\sigma-1}\snr{\tau_{h}D\ti{\psi}}\dx\nonumber \\
 &&+\frac{2\mathcal{N}_{\infty}}{\snr{h}^{1+\vartheta\sigma+\alpha}}\int_{B_{\bar{r}_{0}}}\left(\snr{\partial\ti{\tx{L}}(D\ti{u}(x+h))}+\snr{\partial \ti{\tx{L}}(D\ti{u}(x))}\right)\snr{\tau_{h}\ti{\psi}}^{\sigma}\dx\nonumber \\
 &\le&c\sigma\mathcal{N}_{\infty}\Psi^{\sigma}\snr{h}^{\sigma(1-\vartheta)-(1+\alpha)}\int_{B_{\ti{r}_{0}}}\snr{\partial\ti{\tx{L}}(D\ti{u})}\dx\nonumber \\
 &\le&c\sigma \tx{C}_{\rr}^{2}\mathcal{N}_{\infty}\Psi^{\sigma}\int_{B_{\ti{r}_{0}}}1+\tx{g}(\snr{D\ti{u}})\dx\le c\sigma\tx{C}_{\rr}^{2}\mathcal{N}_{\infty}\Psi^{\sigma}\left(\int_{B_{\ti{r}_{0}}}\ell_{1}(D\ti{u})^{t}\dx\right)^{\frac{q-1}{t}},
 \end{eqnarray*}
 for $c\equiv c(n,\Lambda_{*},\tx{g},\mu)$. Moreover, after splitting with the product rule \eqref{prod},
 \begin{eqnarray*}
 \mbox{(IV)}&=&\frac{q\sigma\tx{C}_{\rr}^{2}}{\snr{h}^{1+\vartheta\sigma+\alpha}}\int_{B_{1}}\eta^{2}\ti{\tx{a}}(x)\snr{\tau_{h}\ti{\psi}}^{\sigma-1}\langle\tau_{h}(\ti{\ell}_{s}(D\ti{u})^{q-2}D\ti{u}),\tau_{h}D\ti{\psi}\rangle\dx\nonumber \\
 &&+\frac{q\sigma\tx{C}_{\rr}^{2}}{\snr{h}^{1+\vartheta\sigma+\alpha}}\int_{B_{1}}\eta^{2}(\tau_{h}\ti{\tx{a}}(\cdot))\snr{\tau_{h}\ti{\psi}}^{\sigma-1}\langle\ti{\ell}_{s}(D\ti{u}(x+h))^{q-2}D\ti{u}(x+h),\tau_{h}D\ti{\psi}\rangle\dx\nonumber \\
 &&+\frac{2q\tx{C}_{\rr}^{2}}{\snr{h}^{1+\vartheta\sigma+\alpha}}\int_{B_{1}}\langle\tau_{h}(\ti{\tx{a}}(\cdot)\ti{\ell}_{s}(D\ti{u})^{q-2}D\ti{u}),(\eta\snr{\tau_{h}\ti{\psi}}^{\sigma-1}\tau_{h}\ti{\psi})D\eta\rangle\dx=:\mbox{(IV)}_{1}+\mbox{(IV)}_{2}+\mbox{(IV)}_{3},
 \end{eqnarray*}
 we bound via \eqref{p2} and \eqref{ts},
 \begin{eqnarray*}
 \snr{\mbox{(IV)}_{1}}+\snr{\mbox{(IV)}_{2}}&\le&c\sigma\tx{C}_{\rr}^{q}\nr{\ti{\tx{a}}}_{L^{\infty}(B_{1})}\Psi^{\sigma}\snr{h}^{\sigma(1-\vartheta)-(1+\alpha)}\int_{B_{\bar{r}_{0}}}\ell_{s}(D\ti{u}(x+h))^{q-1}+\ell_{s}(D\ti{u}(x))^{q-1}\dx\nonumber \\
 &&+c\sigma\tx{C}_{\rr}^{q}[\ti{\tx{a}}]_{0,\alpha;B_{1}}\snr{h}^{\sigma(1-\vartheta)-1}\Psi^{\sigma}\int_{B_{\bar{r}_{0}}}\ell_{s}(D\ti{u}(x+h))^{q-1}\dx\nonumber \\
&\le&c\sigma\tx{C}_{\rr}^{q}\tx{A}_{\rr}\Psi^{\sigma}\int_{B_{\ti{r}_{0}}}\ell_{1}(D\ti{u})^{q-1}\dx\le c\sigma\tx{C}_{\rr}^{q}\tx{A}_{\rr}\Psi^{\sigma}\left(\int_{B_{\ti{r}_{0}}}\ell_{1}(D\ti{u})^{t}\dx\right)^{\frac{q-1}{t}},
 \end{eqnarray*}
 with $c\equiv c(n,q)$, and, by \eqref{af} and \eqref{ts},
 \begin{eqnarray*}
 \snr{\mbox{(IV)}_{3}}&=&\frac{2q\tx{C}_{\rr}^{2}}{\snr{h}^{\vartheta\sigma+\alpha}}\left|\int_{B_{1}}\int_{0}^{1}\ti{\tx{a}}(x+\beta h)\langle \ti{\ell}_{s}(D\ti{u}(x+\beta h))^{q-2}D\ti{u}(x+\beta h),\partial_{h/\snr{h}}(\eta \snr{\tau_{h}\ti{\psi}}^{\sigma-1}\tau_{h}\ti{\psi}D\eta)\rangle\d\beta\dx\right|\nonumber \\
 &\le&\frac{c\sigma\tx{C}_{\rr}\mathcal{N}_{\infty}\nr{\ti{\tx{a}}}_{L^{\infty}(B_{1})}}{\snr{h}^{\vartheta\sigma+\alpha}}\int_{B_{\bar{r}_{0}}}\mathcal{G}_{q}(x,h)\snr{\tau_{h}\ti{\psi}}^{\sigma-1}\left(\snr{\tau_{h}\ti{\psi}}+\snr{\tau_{h}D\ti{\psi}}\right)\dx\nonumber \\
 &\le&c\sigma\tx{C}_{\rr}^{q}\mathcal{N}_{\infty}\Psi^{\sigma}\snr{h}^{\sigma(1-\vartheta)-\alpha}\tx{A}_{\rr}\int_{B_{\ti{r}_{0}}}\ell_{1}(D\ti{u})^{q-1}\dx\le c\sigma\tx{C}_{\rr}^{q}\mathcal{N}_{\infty}\Psi^{\sigma}\tx{A}_{\rr}\left(\int_{B_{\ti{r}_{0}}}\ell_{1}(D\ti{u})^{t}\dx\right)^{\frac{q-1}{t}},
 \end{eqnarray*}
 for $c\equiv c(n,q)$. All in all, it is
 \begin{flalign*}
 \snr{\mbox{(IV)}}\le c\sigma\tx{C}_{\rr}^{q}\Psi^{\sigma}\mathcal{N}_{\infty}\tx{A}_{\rr}\left(\int_{B_{\ti{r}_{0}}}\ell_{1}(D\ti{u})^{t}\dx\right)^{\frac{q-1}{t}},
 \end{flalign*}
 where $c\equiv c(n,q)$. Merging the content of all previous displays we obtain, after routine manipulations,
 \begin{eqnarray}\label{0}
 \frac{1}{\snr{h}^{1+\vartheta\sigma+\alpha}}\mbox{(I)}_{1}&\le&\frac{c\sigma\tx{C}_{\rr}^{\frac{4q}{q-1}}\tx{A}_{\rr}^{2}\mathcal{N}_{\infty}}{\snr{h}^{1+\vartheta\sigma-\alpha}}\left(\int_{B_{\ti{r}_{0}}}\ell_{1}(D\ti{u})^{t}\dx\right)^{\frac{2q-2+\mu}{t}}\left(\int_{B_{\bar{r}_{0}}}\snr{\tau_{h}\ti{u}}^{\frac{t(\sigma-1)}{t+2-2q-\mu}}\dx\right)^{\frac{t+2-2q-\mu}{t}}\nonumber \\
 &&+\frac{c\sigma\tx{C}_{\rr}^{\mu}\tx{A}_{\rr}\mathcal{N}_{\infty}}{\snr{h}^{\vartheta\sigma+\alpha}}\left(\int_{B_{\ti{r}_{0}}}\ell_{1}(D\ti{u})^{t}\dx\right)^{\frac{q-1}{t}}\left(\int_{B_{\bar{r}_{0}}}\snr{\tau_{h}\ti{u}}^{\frac{t\sigma}{t+1-q}}\dx\right)^{\frac{t+1-q}{t}}\nonumber \\
&&+c\sigma\tx{C}_{\rr}^{\mu}\Psi^{\sigma}\tx{A}_{\rr}\mathcal{N}_{\infty}\left(\int_{B_{\ti{r}_{0}}}\ell_{1}(D\ti{u})^{t}\dx\right)^{\frac{q-1}{t}}=:\tx{B}_{\ti{u}},
 \end{eqnarray}
 with $c\equiv c(\data_{*})$. Now we need to estimate the left-hand side of \eqref{0}. By H\"older inequality with conjugate exponents $\left(\frac{2}{2-\mu},\frac{2}{\mu}\right)$ we get
 \begin{eqnarray}\label{1}
 \mbox{(V)}&:=&\frac{1}{\snr{h}^{1+\vartheta\sigma+\alpha}}\int_{B_{1}}\eta^{2}\snr{\tau_{h}\ti{u}}^{\sigma-\ti{\delta}}\snr{\tau_{h}D\ti{u}}^{2-\mu}\dx\nonumber \\
 &\le&\left( \frac{\mbox{(I)}_{1}}{\snr{h}^{1+\vartheta\sigma+\alpha}}\right)^{\frac{2-\mu}{2}}\left(\frac{1}{\snr{h}^{1+\vartheta\sigma+\alpha}}\int_{B_{1}}\eta^{2}\snr{\tau_{h}\ti{u}}^{\sigma+\frac{2(1-\ti{\delta})}{\mu}-1}\mathcal{D}(x,h)^{\frac{2-\mu}{2}}\dx\right)^{\frac{\mu}{2}}\nonumber \\
 &\stackrel{\eqref{0}}{\le}&\tx{B}_{\ti{u}}^{\frac{2-\mu}{2}}\left(\int_{B_{\ti{r}_{0}}}\ell_{1}(D\ti{u})^{t+2\tx{s}_{*}+\theta}\dx\right)^{\frac{\mu(2-\mu)}{2(t+2\tx{s}_{*}+\theta)}}\nonumber \\
 &&\cdot \left(\int_{B_{\bar{r}_{0}}}\left|\frac{\tau_{h}\ti{u}}{\snr{h}^{\frac{\mu(1+\vartheta\sigma+\alpha)}{\mu(\sigma-1)+2(1-\ti{\delta})}}}\right|^{\frac{(\mu(\sigma-1)+2(1-\ti{\delta}))(t+2\tx{s}_{*}+\theta)}{   \mu(t+2\tx{s}_{*}+\theta-2+\mu)}}\dx\right)^{\frac{\mu(t+2\tx{s}_{*}+\theta-2+\mu)}{2(t+2\tx{s}_{*}+\theta)}}.
 \end{eqnarray}
Recalling the definition of $\tx{M}_{*}\ge 1$ in \eqref{m*} and noticing that via \eqref{etaeta} it is $\nr{\tau_{h}D\ti{u}}_{L^{\infty}(\supp(\eta))}\le 2\tx{C}_{\rr}^{-1}\tx{M}_{*}$, and therefore
\begin{eqnarray}\label{mmmm}
\mathds{1}_{\{\supp(\eta)\cap\{\snr{\tau_{h}D\ti{u}}>0\}\}}\snr{\tau_{h}D\ti{u}}^{1-\mu}&\stackrel{\mu\ge 1}{\ge}& \mathds{1}_{\{\supp(\eta)\cap\{\snr{\tau_{h}D\ti{u}}>0\}\}}\nr{\tau_{h}D\ti{u}}_{L^{\infty}({\{\supp(\eta)\cap\{\snr{\tau_{h}D\ti{u}}>0\}\})}}^{1-\mu}\nonumber\\
&\ge& \mathds{1}_{\{\supp(\eta)\cap\{\snr{\tau_{h}D\ti{u}}>0\}\}}\left(\frac{\tx{C}_{\rr}}{2\tx{M}_{*}}\right)^{\mu-1},
\end{eqnarray}
we keep bounding below
 \begin{eqnarray}\label{2}
 \mbox{(V)}&\stackrel{\eqref{mmmm}}{\ge}&\frac{\tx{C}_{\rr}^{\mu-1}}{2^{\mu-1}\tx{M}_{*}^{\mu-1}\snr{h}^{1+\vartheta\sigma+\alpha}}\int_{B_{1}}\eta^{2}\snr{\tau_{h}\ti{u}}^{\sigma-\ti{\delta}}\snr{\tau_{h}D\ti{u}}\dx\nonumber \\
 &\ge&\frac{\tx{C}_{\rr}^{\mu-1}}{c\tx{M}_{*}^{\mu-1}\snr{h}^{1+\vartheta\sigma+\alpha}}\int_{B_{1}}\left|D\left(\eta^{2}\snr{\tau_{h}\ti{u}}^{\sigma-\ti{\delta}}\tau_{h}\ti{u}\right)\right|\dx-\frac{c\tx{C}_{\rr}^{\mu-1}\mathcal{N}_{\infty}}{\tx{M}_{*}^{\mu-1}\snr{h}^{1+\vartheta\sigma+\alpha}}\int_{B_{\bar{r}_{0}}}\snr{\tau_{h}\ti{u}}^{\sigma-\ti{\delta}+1}\dx\nonumber \\
 &\stackrel{\eqref{gh}}{\ge}&\frac{\tx{C}_{\rr}^{\mu-1}}{c\tx{M}_{*}^{\mu-1}\snr{h}^{2+\vartheta\sigma+\alpha}}\int_{B_{\bar{r}_{0}}}\left|\tau_{h}\left(\eta^{2}\snr{\tau_{h}\ti{u}}^{\sigma-\ti{\delta}}\tau_{h}\ti{u}\right)\right|\dx-\frac{c\tx{C}_{\rr}^{\mu-1}\mathcal{N}_{\infty}}{\tx{M}_{*}^{\mu-1}\snr{h}^{1+\vartheta\sigma+\alpha}}\int_{B_{\bar{r}_{0}}}\snr{\tau_{h}\ti{u}}^{\sigma-\ti{\delta}+1}\dx\nonumber \\
 &\stackrel{\eqref{prod}}{\ge}&\frac{\tx{C}_{\rr}^{\mu-1}}{c\tx{M}_{*}^{\mu-1}\snr{h}^{2+\vartheta\sigma+\alpha}}\int_{B_{\bar{r}_{0}}}\eta(x+h)^{2}\left|\tau_{h}\left(\snr{\tau_{h}\ti{u}}^{\sigma-\ti{\delta}}\tau_{h}\ti{u}\right)\right|\dx\nonumber \\
 &&-\frac{c\tx{C}_{\rr}^{\mu-1}}{\tx{M}_{*}^{\mu-1}\snr{h}^{2+\vartheta\sigma+\alpha}}\int_{B_{\bar{r}_{0}}}\snr{\tau_{h}(\eta^{2})}\snr{\tau_{h}\ti{u}}^{\sigma-\ti{\delta}+1}\dx-\frac{c\tx{C}_{\rr}^{\mu-1}\mathcal{N}_{\infty}}{\tx{M}_{*}^{\mu-1}\snr{h}^{1+\vartheta\sigma+\alpha}}\int_{B_{\bar{r}_{0}}}\snr{\tau_{h}\ti{u}}^{\sigma-\ti{\delta}+1}\dx\nonumber \\
 &\ge&\frac{\tx{C}_{\rr}^{\mu-1}}{c\tx{M}_{*}^{\mu-1}\snr{h}^{2+\vartheta\sigma+\alpha}}\int_{B_{\ti{r}_{1}}}\left|\tau_{h}\left(\snr{\tau_{h}\ti{u}}^{\sigma-\ti{\delta}}\tau_{h}\ti{u}\right)\right|\dx-\frac{c\tx{C}_{\rr}^{\mu-1}\mathcal{N}_{\infty}}{\tx{M}_{*}^{\mu-1}\snr{h}^{1+\vartheta\sigma+\alpha}}\int_{B_{\bar{r}_{0}}}\snr{\tau_{h}\ti{u}}^{\sigma-\ti{\delta}+1}\dx\nonumber \\
 &\stackrel{\eqref{mon}}{\ge}&\frac{\tx{C}_{\rr}^{\mu-1}}{c\tx{M}_{*}^{\mu-1}\snr{h}^{2+\vartheta\sigma+\alpha}}\int_{B_{\ti{r}_{1}}}\snr{\tau_{h}^{2}\ti{u}}^{\sigma-\ti{\delta}+1}\dx-\frac{c\tx{C}_{\rr}^{\mu-1}\mathcal{N}_{\infty}}{\tx{M}_{*}^{\mu-1}\snr{h}^{1+\vartheta\sigma+\alpha}}\int_{B_{\bar{r}_{0}}}\snr{\tau_{h}\ti{u}}^{\sigma-\ti{\delta}+1}\dx,
 \end{eqnarray}
 with $c\equiv c(\mu,q,\alpha,t)$ - here and in the reminder of this section, we incorporate any dependency on $\sigma$ into a dependency on $(\mu,q,\alpha,t)$. Combining \eqref{2} with \eqref{1}, we obtain
\begin{eqnarray}\label{3}
    \int_{B_{\ti{r}_{1}}}\left|\frac{\tau_{h}^{2}\ti{u}}{\snr{h}^{\frac{2+\vartheta\sigma+\alpha}{\sigma-\ti{\delta}+1}}}\right|^{\sigma-\ti{\delta}+1}\dx&\le&c\tx{M}_{*}^{\mu-1}\tx{C}_{\rr}^{1-\mu}\tx{B}_{\ti{u}}^{\frac{2-\mu}{2}}\left(\int_{B_{\ti{r}_{0}}}\ell_{1}(D\ti{u})^{t+2\tx{s}_{*}+\theta}\dx\right)^{\frac{\mu(2-\mu)}{2(t+2\tx{s}_{*}+\theta)}}\nonumber \\
    &&\cdot\left(\int_{B_{\bar{r}_{0}}}\left|\frac{\tau_{h}\ti{u}}{\snr{h}^{\frac{\mu(1+\vartheta\sigma+\alpha)}{\mu(\sigma-1)+2(1-\ti{\delta})}}}\right|^{\frac{(\mu(\sigma-1)+2(1-\ti{\delta}))(t+2\tx{s}_{*}+\theta)}{   \mu(t+2\tx{s}_{*}+\theta-2+\mu)}}\dx\right)^{\frac{\mu(t+2\tx{s}_{*}+\theta-2+\mu)}{2(t+2\tx{s}_{*}+\theta)}}\nonumber \\
    &&+\frac{c\mathcal{N}_{\infty}}{\snr{h}^{1+\vartheta\sigma+\alpha}}\int_{B_{\bar{r}_{0}}}\snr{\tau_{h}\ti{u}}^{\sigma-\ti{\delta}+1}\dx,
\end{eqnarray}
for $c\equiv c(\data_{*},\alpha,t)$. To control the difference quotients on the right-hand side of \eqref{3} including those implicit in $\tx{B}_{\ti{u}}$, set $\beta:=\max\{1+\vartheta\sigma+\alpha,t\}+2n+2$, and notice that by \eqref{ts} and \eqref{mumu} it is
\eqn{num}
$$
 \left\{
 \begin{array}{c}
 \displaystyle
 \ 1+\vartheta\sigma+\alpha=\sigma-1+\frac{2(1-\ti{\delta})}{\mu}=t+2\tx{s}_{*}+\theta-2+\mu\\[8pt]\displaystyle
 \ \frac{1+\vartheta\sigma+\alpha}{\sigma-\ti{\delta}+1},\ \ \frac{1+\vartheta\sigma-\alpha}{\sigma-1},\ \ \frac{\vartheta\sigma+\alpha}{\sigma}, \ \ \frac{1+\vartheta\sigma+\alpha}{t+2\tx{s}_{*}+\mu-1},\ \ \frac{\vartheta\sigma+\alpha}{t+1-q}\in (0,1)\\[8pt]\displaystyle
 \ 1+\vartheta\sigma-\alpha=t+2-2q-\mu.
 \end{array}
 \right.
 $$
We first control
\begin{eqnarray*}
\int_{B_{\bar{r}_{0}}}\left|\frac{\tau_{h}\ti{u}}{\snr{h}^{\frac{\mu(1+\vartheta\sigma+\alpha)}{\mu(\sigma-1)+2(1-\ti{\delta})}}}\right|^{\frac{(\mu(\sigma-1)+2(1-\ti{\delta}))(t+2\tx{s}_{*}+\theta)}{   \mu(t+2\tx{s}_{*}+\theta-2+\mu)}}\dx&\stackrel{\eqref{num}_{1}}{=}&
\int_{B_{\bar{r}_{0}}}\left|\frac{\tau_{h}\ti{u}}{\snr{h}}\right|^{t+2\tx{s}_{*}+\theta}\dx\nonumber \\
&\stackrel{\eqref{gh}}{\le}&\int_{B_{\ti{r}_{0}}}\snr{D\ti{u}}^{t+2\tx{s}_{*}+\theta}\dx,
\end{eqnarray*}
while by \eqref{fs.0}, \eqref{bdd} and $\eqref{num}_{2,3}$ we obtain
 \begin{flalign*}
 &\int_{B_{\bar{r}_{0}}}\left|\frac{\tau_{h}\ti{u}}{\snr{h}^{\frac{1+\vartheta\sigma-\alpha}{\sigma-1}}}\right|^{\frac{t(\sigma-1)}{t+2-2q-\mu}}\dx+\int_{B_{\bar{r}_{0}}}\left|\frac{\tau_{h}\ti{u}}{\snr{h}^{\frac{\vartheta\sigma+\alpha}{\sigma}}}\right|^{\frac{t\sigma}{t+1-q}}\dx+\int_{B_{\bar{r}_{0}}}\left|\frac{\tau_{h}\ti{u}}{\snr{h}^{\frac{1+\vartheta\sigma+\alpha}{\sigma-\ti{\delta}+1}}}\right|^{\sigma-\ti{\delta}+1}\dx\nonumber \\
 &\qquad \quad \quad \le \frac{c}{(\ti{\tau}_{1}-\ti{\tau}_{2})^{\beta}}+c\int_{B_{\ti{r}_{0}}}\snr{D\ti{u}}^{t}\dx+c\left(\int_{B_{\ti{r}_{0}}}\snr{D\ti{u}}^{t}\dx\right)^{\frac{\vartheta\sigma+\alpha}{t+1-q}}+c\left(\int_{B_{\ti{r}_{0}}}\snr{D\ti{u}}^{t+2\tx{s}_{*}+\mu-1}\dx\right)^{\frac{1+\vartheta\sigma+\alpha}{t+2\tx{s}_{*}+\mu-1}},
 \end{flalign*}
 with $c\equiv c(n,\mu,q,\alpha,t)$. Keeping in mind \eqref{0}, the content of the above display yields in particular that
 \begin{flalign*}
\tx{B}_{\ti{u}}\le \frac{c\tx{T}_{\rr;\sigma}}{(\ti{\tau}_{1}-\ti{\tau}_{2})^{2+\beta}}\left(\int_{B_{\ti{r}_{0}}}\ell_{1}(D\ti{u})^{t}\dx\right),
 \end{flalign*}
 where we set $\tx{T}_{\rr;\sigma}:=\tx{C}_{\rr}^{\frac{4q}{q-1}}\Psi^{\sigma}\tx{A}_{\rr}^{2}$, used \eqref{etaeta}, $\tx{C}_{\rr}\ge 1$, $\mu\ge 1$, and it is $c\equiv c(\data_{*},\alpha,t)$. Plugging the content of the previous three displays into \eqref{3} we obtain
  \begin{eqnarray}\label{4}
 \left\|\frac{\tau_{h}^{2}\ti{u}}{\snr{h}^{\frac{2+\vartheta\sigma+\alpha}{\sigma-\ti{\delta}+1}}}\right\|_{L^{\sigma-\ti{\delta}+1}(B_{\ti{r}_{1}})}^{\sigma-\ti{\delta}+1}&\le&\frac{c\tx{T}_{\rr;\sigma}^{\frac{2-\mu}{2}}\tx{M}_{*}^{\mu-1}}{(\ti{\tau}_{1}-\ti{\tau}_{2})^{\frac{(2+\beta)(2-\mu)}{2}}}\left(\int_{B_{\ti{r}_{0}}}\ell_{1}(D\ti{u})^{t}\dx\right)^{\frac{2-\mu}{2}}\left(\int_{B_{\ti{r}_{0}}}\ell_{1}(D\ti{u})^{t+2\tx{s}_{*}+\theta}\dx\right)^{\frac{\mu}{2}}\nonumber \\
 &&+\frac{c}{(\ti{\tau}_{1}-\ti{\tau}_{2})^{\beta}}\left(\int_{B_{\ti{r}_{0}}}\ell_{1}(D\ti{u})^{t+2\tx{s}_{*}+\mu-1}\dx\right)^{\frac{1+\vartheta\sigma+\alpha}{t+2\tx{s}_{*}+\mu-1}}\nonumber \\
 &\le&\frac{c\tx{T}_{\rr;\sigma}^{\frac{2-\mu}{2}}\tx{M}_{*}^{\frac{2(\mu-1)+\mu(\theta+1-\mu)_{+}}{2}}}{(\ti{\tau}_{1}-\ti{\tau}_{2})^{\frac{(2+\beta)(2-\mu)}{2}}}\left(\int_{B_{\ti{r}_{0}}}\ell_{1}(D\ti{u})^{t}\dx\right)^{\frac{2-\mu}{2}}\left(\int_{B_{\ti{r}_{0}}}\ell_{1}(D\ti{u})^{t+2\tx{s}_{*}+\mu-1}\dx\right)^{\frac{\mu}{2}}\nonumber \\
 &&+\frac{c}{(\ti{\tau}_{1}-\ti{\tau}_{2})^{\beta}}\left(\int_{B_{\ti{r}_{0}}}\ell_{1}(D\ti{u})^{t+2\tx{s}_{*}+\mu-1}\dx\right)^{\frac{1+\vartheta\sigma+\alpha}{t+2\tx{s}_{*}+\mu-1}},
 \end{eqnarray}
 for $c\equiv c(\data_{*},\alpha,t)$. Next, by \eqref{mumu} and \eqref{kkk} we have
 $$
 \left\{
 \begin{array}{c}
 \displaystyle
 \ \sigma-\ti{\delta}+1=t+2(-q+2-\mu+\alpha-\theta/2)+\mu-1=t+2\tx{s}_{*}+\mu-1>t\\[10pt]\displaystyle
 \ \frac{2+\vartheta\sigma+\alpha}{\sigma-\ti{\delta}+1}=1+\frac{\omega(2-\mu)}{\sigma-\ti{\delta}+1}=:1+\epsilon\in (1,2),
 \end{array}
 \right.
 $$
 therefore \eqref{immersione2} gives
 \begin{eqnarray*}
    \nr{D\ti{u}}_{L^{t+2\tx{s}_{*}+\mu-1}(B_{\ti{\tau}_{2}})}^{t+2\tx{s}_{*}+\mu-1}
    &\le&\frac{c}{\mathcal{h}_{0}^{(1+\epsilon)(t+2\tx{s}_{*}+\mu-1)}}+c\sup_{0<\snr{h}<\mathcal{h}_{0}}\left\|\frac{\tau_{h}^{2}\ti{u}}{\snr{h}^{1+\epsilon}}\right\|_{L^{t+2\tx{s}_{*}+\mu-1}(B_{\ti{r}_{1}})}^{t+2\tx{s}_{*}+\mu-1}\nonumber \\
    &\stackrel{\eqref{4}}{\le}&\frac{c\tx{T}_{\rr;\sigma}^{\frac{2-\mu}{2}}\tx{M}_{*}^{\frac{2(\mu-1)+\mu(\theta+1-\mu)_{+}}{2}}}{(\ti{\tau}_{1}-\ti{\tau}_{2})^{\frac{(2+\beta)(2-\mu)}{2}}}\left(1+\nr{D\ti{u}}_{L^{t}(B_{\ti{r}_{0}})}^{\frac{t(2-\mu)}{2}}\nr{D\ti{u}}_{L^{t+2\tx{s}_{*}+\mu-1}(B_{\ti{r}_{0}})}^{\frac{\mu(t+2\tx{s}_{*}+\mu-1)}{2}}\right)\nonumber \\
    &&+\frac{c}{(\ti{\tau}_{1}-\ti{\tau}_{2})^{\beta}}\left(1+\nr{D\ti{u}}_{L^{t+2\tx{s}_{*}+\mu-1}(B_{\ti{r}_{0}})}^{1+\vartheta\sigma+\alpha}\right)+\frac{c}{\mathcal{h}_{0}^{(1+\epsilon)(t+2\tx{s}_{*}+\mu-1)}}\nonumber \\
    &\stackrel{\eqref{mumu}}{\le}&\frac{c\tx{T}_{\rr;\sigma}^{\frac{2-\mu}{2}}\tx{M}_{*}^{\ti{\gamma}}}{(\ti{\tau}_{1}-\ti{\tau}_{2})^{\frac{(2+\beta)(2-\mu)}{2}}}\nr{D\ti{u}}_{L^{t}(B_{\ti{r}_{0}})}^{\frac{t(2-\mu)}{2}}\nr{D\ti{u}}_{L^{t+2\tx{s}_{*}+\mu-1}(B_{\ti{r}_{0}})}^{\frac{\mu(t+2\tx{s}_{*}+\mu-1)}{2}}\nonumber \\
    &&+\frac{c}{(\ti{\tau}_{1}-\ti{\tau}_{2})^{\beta}}\nr{D\ti{u}}_{L^{t+2\tx{s}_{*}+\mu-1}(B_{\ti{r}_{0}})}^{1+\vartheta\sigma+\alpha}+\frac{c\tx{M}_{*}^{\ti{\gamma}}}{\mathcal{h}_{0}^{(3+\beta+\epsilon)(t+2\tx{s}_{*}+\mu-1)}}\nonumber \\
    &\stackrel{\eqref{num}}{\le}&\frac{1}{4}\nr{D\ti{u}}_{L^{t+2\tx{s}_{*}+\mu-1}(B_{\ti{\tau}_{1}})}^{t+2\tx{s}_{*}+\mu-1}+\frac{c\tx{T}_{\rr;\sigma}\tx{M}_{*}^{\frac{2\ti{\gamma}}{2-\mu}}}{(\ti{\tau}_{1}-\ti{\tau}_{2})^{\tx{d}}}\left(1+\nr{D\ti{u}}_{L^{t}(B_{\ti{r}_{0}})}^{t}\right),
    \end{eqnarray*}
 where $\ti{\gamma}$ has been defined in \eqref{kkk}, we set $\tx{d}:=(4+2\beta+\epsilon)(\sigma-\ti{\delta}+1)$,
 and it is $c\equiv c(\data_{*},\alpha,t,\omega)$. Via Lemma \ref{iterlem} and standard manipulations we conclude with
 \begin{flalign}\label{1-0}
    \nr{D\ti{u}}_{L^{t+2\tx{s}_{*}+\mu-1}(B_{\rr_{1}})}^{t+2\tx{s}_{*}+\mu-1}+1\le\frac{c\tx{T}_{\rr;\sigma}\tx{M}_{*}^{\frac{2\ti{\gamma}}{2-\mu}}}{(\rr_{0}-\rr_{1})^{\tx{d}}}\left(1+\nr{D\ti{u}}_{L^{t}(B_{\rr_{0}})}^{t}\right).
    \end{flalign}
For simplicity from now on we allow the value of $\tx{d}$ to change from line to line, maintaining always at most a dependency on $(n,\mu,q,\alpha,t)$. At this stage, we only need to iterate \eqref{1-0}. To this aim, we introduce an increasing sequence of numbers $\{t_{i}\}_{i\in \mathbb{N}\cup \{0\}}\subset [1,\infty)$ so that $t_{0}=t_{*}$ and for $i\ge 0$, $t_{i+1}=t_{i}+2\tx{s}_{*}+\mu-1=t_{*}+(i+1)(2\tx{s}_{*}+\mu-1)$ (of course $t_{i}\nearrow \infty$), parameters $0<r_{2}\le \tau_{2}<\tau_{1}\le r_{1}\le 1$, and a sequence of shrinking, concentric balls $\{B_{\rr_{i}}\}_{i\in \N\cup\{0\}}$ such that $\rr_{i}:=\tau_{2}+(\tau_{1}-\tau_{2})2^{-(i+2)}$, therefore $B_{\tau_{2}}\subset B_{\rr_{i+1}}\subset B_{\rr_{i}}\subseteq B_{\tau_{1}}$ and $\bigcap_{i\in \N\cup\{0\}}B_{\rr_{i}}=\bar{B}_{\tau_{2}}$. 
Given that \eqref{1-0} is true for all $t\ge t_{*}$, we can repeat all the above procedure with $t\equiv t_{i}$, $\rr_{0}$, $\rr_{1}$ replaced by $\rr_{i}$, $\rr_{i+1}$ respectively to derive 
\begin{flalign}\label{i+1-i}
    \nr{D\ti{u}}_{L^{t_{i+1}}(B_{\rr_{i+1}})}^{t_{i+1}}+1\le \frac{c_{i}\tx{T}_{\rr;\sigma_{i}}\tx{M}_{*}^{\frac{2\ti{\gamma}}{2-\mu}}}{(\tau_{1}-\tau_{2})^{\tx{d}_{i}}}\left(\nr{D\ti{u}}_{L^{t_{i}}(B_{\rr_{i}})}^{t_{i}}+1\right),
    \end{flalign}
    for $c_{i}\equiv c_{i}(\data_{*},\alpha,i,\omega)$, $\tx{d}_{i}(n,\mu,q,\alpha,i,\omega)$ - here $\sigma_{i}$ is defined as in \eqref{ts} with $t=t_{i}$, and we incorporated any dependency of constant $c_{i}$ in \eqref{i+1-i} on $t$ into that on $i$. We then fix $i\in \N$ and iterate \eqref{i+1-i} for all $j\in \{0,\cdots,i\}$ to get
    \begin{flalign}\label{14}
    \nr{D\ti{u}}_{L^{t_{i+1}}(B_{\rr_{i+1}})}^{t_{i+1}}+1\le \tx{M}_{*}^{\frac{2(i+1)\ti{\gamma}}{2-\mu}}\left(\prod_{j=0}^{i}\frac{c_{j}\tx{T}_{\rr;\sigma_{j}}}{(\tau_{1}-\tau_{2})^{\tx{d}_{j}}}\right)\left(\nr{D\ti{u}}_{L^{t_{0}}(B_{\rr_{0}})}^{t_{0}}+1\right).
    \end{flalign}
   The very definition of exponents $t_{i}$'s yield that $t_{i+1}>t_{0}=t_{*}>1$, so in \eqref{14} we apply the interpolation inequality
    $$
\nr{D\ti{u}}_{L^{t_{0}}(B_{\rr_{0}})}\le  \nr{D\ti{u}}_{L^{t_{i+1}}(B_{\rr_{0}})}^{\chi_{i+1}}\nr{D\ti{u}}_{L^{1}(B_{\rr_{0}})}^{1-\chi_{i+1}},
$$
with numbers $\chi_{i+1}\in (0,1)$ solving
$$
\frac{1}{t_{0}}=\frac{\chi_{i+1}}{t_{i+1}}+1-\chi_{i+1} \ \Longrightarrow \ \chi_{i+1}=\frac{t_{i+1}(t_{0}-1)}{t_{0}(t_{i+1}-1)}.
$$
We then recall that $\tau_{2}\le \rr_{i+1}< \rr_{0}\le \tau_{1}$, so we get
\begin{eqnarray*}
\nr{D\ti{u}}_{L^{t_{i+1}}(B_{\tau_{2}})}^{t_{i+1}}+1&\le&\nr{D\ti{u}}_{L^{t_{i+1}}(B_{\rr_{i+1}})}^{t_{i+1}}+1\nonumber \\
&\le& \tx{M}_{*}^{\frac{2(i+1)\ti{\gamma}}{2-\mu}}\left(\prod_{j=0}^{i}\frac{c_{j}\tx{T}_{\rr;\sigma_{j}}}{(\tau_{1}-\tau_{2})^{\tx{d}_{j}}}\right)\left(\nr{D\ti{u}}_{L^{t_{i+1}}(B_{\rr_{0}})}^{t_{0}\chi_{i+1}}\nr{D\ti{u}}_{L^{1}(B_{\rr_{0}})}^{t_{0}(1-\chi_{i+1})}+1\right)\nonumber \\
&\le& \tx{M}_{*}^{\frac{2(i+1)\ti{\gamma}}{2-\mu}}\left(\prod_{j=0}^{i}\frac{c_{j}\tx{T}_{\rr;\sigma_{j}}}{(\tau_{1}-\tau_{2})^{\tx{d}_{j}}}\right)\left(\nr{D\ti{u}}_{L^{t_{i+1}}(B_{\tau_{1}})}^{\frac{t_{i+1}(t_{0}-1)}{t_{i+1}-1}}\nr{D\ti{u}}_{L^{1}(B_{\tau_{1}})}^{\frac{t_{i+1}-t_{0}}{t_{i+1}-1}}+1\right).
\end{eqnarray*}
Now, being $t_{i+1}>t_{0}$ it is $(t_{0}-1)/(t_{i+1}-1)<1$, so we can use Young inequality with conjugate exponents $\left(\frac{t_{i+1}-1}{t_{0}-1},\frac{t_{i+1}-1}{t_{i+1}-t_{0}}\right)$ to derive
\begin{eqnarray}\label{14.14}
\nr{D\ti{u}}_{L^{t_{i+1}}(B_{\tau_{2}})}^{t_{i+1}}&\le& \frac{1}{4}\nr{D\ti{u}}_{L^{t_{i+1}}(B_{\tau_{1}})}^{t_{i+1}}\nonumber \\
&&+\tx{M}_{*}^{\frac{2\ti{\gamma}(i+1)(t_{i+1}-1)}{(2-\mu)(t_{i+1}-t_{0})}}\left(\prod_{j=0}^{i}\frac{c_{j}\tx{T}_{\rr;\sigma_{j}}}{(\tau_{1}-\tau_{2})^{\tx{d}_{j}}}\right)^{\frac{t_{i+1}-1}{t_{i+1}-t_{0}}}\left(\nr{D\ti{u}}_{L^{1}(B_{\tau_{1}})}+1\right).
\end{eqnarray}
Lemma \ref{iterlem} then gives
\begin{flalign*}
\nr{D\ti{u}}_{L^{t_{i+1}}(B_{r_{2}})}\le \frac{c_{i}\tx{T}_{\rr;\sigma_{i}}^{\frac{(i+1)(t_{i+1}-1)}{t_{i+1}(t_{i+1}-t_{0})}}\tx{M}_{*}^{\frac{2\ti{\gamma}(i+1)(t_{i+1}-1)}{t_{i+1}(2-\mu)(t_{i+1}-t_{0})}}}{(r_{1}-r_{2})^{\tx{d}_{i}}}\left(\nr{D\ti{u}}_{L^{1}(B_{r_{1}})}+1\right)^{\frac{1}{t_{i+1}}},
\end{flalign*}
for constants $c_{i}\equiv c_{i}(\data_{*},\alpha,i,\omega)$, $\tx{d}_{i}\equiv \tx{d}_{i}(n,\mu,q,\alpha,i,\omega)$ that are possibly different from those displayed in \eqref{14.14} but share the same dependencies. In the previous inequality we bound
$$
\frac{(i+1)(t_{i+1}-1)}{t_{i+1}(t_{i+1}-t_{0})}=\frac{1}{2\tx{s}_{*}+\mu-1}\left(1-\frac{1}{t_{i+1}}\right)\le \frac{1}{2\tx{s}_{*}+\mu-1},
$$
and routine manipulations eventually yield 
\begin{flalign}\label{iii}
\nr{D\ti{u}}_{L^{t_{i+1}}(B_{r_{2}})}\le\frac{c_{i}\tx{T}_{\rr;\sigma_{i}}^{\frac{1}{2\tx{s}_{*}}}\tx{M}_{*}^{\frac{2\ti{\gamma}}{(2-\mu)(2\tx{s}_{*}+\mu-1)}}}{(r_{1}-r_{2})^{\tx{d}_{i}}}\left(\nr{D\ti{u}}_{L^{1}(B_{r_{1}})}+1\right)^{\frac{1}{t_{i+1}}},
\end{flalign}
for $c_{i}\equiv c_{i}(\data_{*},\alpha,i,\omega)$, $\tx{d}_{i}\equiv \tx{d}_{i}(n,\mu,q,\alpha,i,\omega)$. Next, observe that by \eqref{mumu}, $\eqref{thre}$ and \eqref{kkk} it is
\eqn{isis}
$$
\frac{2\ti{\gamma}}{2\tx{s}_{*}+\mu-1}\stackrel{\eqref{ts}}{=}\frac{2\max\{\mu-1,(2-\mu)\omega\}}{3-\mu+2\alpha-2q-\omega(2-\mu)}<\frac{2\max\{\mu-1,\omega\}}{1+\alpha-q}=\gamma_{\mu;\omega}.
$$
Moreover, thanks to the same arguments in \eqref{isis}, we can easily make the integrability gain in \eqref{iii} independent of both, $\mu$ and $\omega$. In fact, by the definition of the $t_{i}$'s we have
$$
t_{i+1}=t_{*}+(i+1)(2\tx{s}_{*}+\mu-1)\stackrel{\eqref{thre}}{\ge} t_{*}+(i+1)(1+\alpha-q)=:\ti{t}_{i+1}\equiv \ti{t}_{i+1}(q,\alpha,i),
$$
thus by \eqref{isis} and H\"older inequality, \eqref{iii} becomes
\begin{eqnarray*}
\nr{D\ti{u}}_{L^{\ti{t}_{i+1}}(B_{r_{2}})}&\lesssim_{n}& \nr{D\ti{u}}_{L^{t_{i+1}}(B_{r_{2}})}\nonumber \\
&\le&\frac{c_{i}\tx{T}_{\rr;\sigma_{i}}^{\frac{1}{2\tx{s}_{*}}}\tx{M}_{*}^{\frac{\gamma_{\mu;\omega}}{2-\mu}}}{(r_{1}-r_{2})^{\tx{d}_{i}}}\left(\nr{D\ti{u}}_{L^{1}(B_{r_{1}})}+1\right)^{\frac{1}{t_{i+1}}}\le \frac{c_{i}\tx{T}_{\rr;\sigma_{i}}^{\frac{1}{2\tx{s}_{*}}}\tx{M}_{*}^{\frac{\gamma_{\mu;\omega}}{2-\mu}}}{(r_{1}-r_{2})^{\tx{d}_{i}}}\left(\nr{D\ti{u}}_{L^{1}(B_{r_{1}})}+1\right)^{\frac{1}{\ti{t}_{i+1}}}.
\end{eqnarray*}
Finally, whenever $1\le p<\infty$ is a number, we can find $i_{p}\equiv i_{p}(q,\alpha,p)\in \mathbb{N}$ such that $\ti{t}_{i_{p}+1}\ge p$, and \eqref{15} holds true by H\"older inequality, after setting $r_{1}=1$, $r_{2}=\tau\in (0,1)$, and scaling back on $B_{\rr}(x_{0})$.
 \begin{proposition}\label{p31}
Under assumptions \eqref{1qa}$_{1}$, \eqref{ass.2}, \eqref{assfr} and \eqref{sasa}, let 
$B\Subset 2B\Subset \Omega$ be concentric balls, and $u\in (u_{0}+W^{1,q}_{0}(B))\cap \tx{K}^{\psi}(B)\cap C^{1}_{\loc}(B)$ be the solution of Dirichlet problem \eqref{pdreg} with obstacle function $\psi$ satisfying \eqref{p2}, boundary datum $u_{0}$ as in \eqref{u0}, exponents $(\mu,q)$ verifying \eqref{mumu}, and threshold $\omega_{*}\equiv \omega_{*}(\mu,q,\alpha)\in (0,1/2]$ from \eqref{thre}$_{2}$. For every ball $B_{\rr}(x_{0})\subset B_{2\rr}(x_{0})\Subset B$, constants
\eqn{m*}
$$
\max\left\{1,\nr{Du}_{L^{\infty}(B_{\rr}(x_{0}))}\right\}\le \tx{M}_{*},
$$
and any $p\in [1,\infty)$, $\tau\in (0,1)$, $\omega\in (0,\omega_{*})$ it holds
\eqn{15}
$$
\nr{Du}_{L^{p}(B_{\tau\rr}(x_{0}))}\le \frac{c\tx{M}_{*}^{\frac{\gamma_{\mu;\omega}}{2-\mu}}}{\rr^{\tx{d}}\tau^{\tx{d}}(1-\tau)^{\tx{d}}}\left(1+\nr{u_{0}}_{L^{\infty}(B)}^{\tx{d}}\right)\left(1+\nr{Du}_{L^{1}(B_{\rr}(x_{0}))}^{\frac{1}{p}}\right),
$$
where $\gamma_{\mu;\omega}\equiv \gamma_{\mu;\omega}(\mu,q,\alpha,\omega)$ is defined in $\eqref{kkk}$, and it is $c\equiv c(\data_{*}(B),p,\omega)$, $\tx{d}\equiv \tx{d}(n,\mu,q,\alpha,p)$.
\end{proposition}
\noindent We conclude this section with some comments on our hybrid fractional Moser technique, possible variants and extensions.
\begin{remark}\label{r32}
\emph{The iteration built in Section \ref{fm} is probably the most extreme version of fractional Moser iteration available, as it exploits till the very end, and in a quantitative way, the interpolation principle of trading control on the oscillation for size properties of functions. In fact, the ultimate goal of the whole machinery designed above is estimating certain Besov norms of minima with arbitrarily high integrability power and fractional differentiability rate slightly larger than one, i.e.,
\eqn{fff.f}
$$
\int\left|\frac{\tau_{h}^{2}\ti{u}}{\snr{h}^{1+\epsilon}}\right|^{t+2\tx{s}_{*}+\mu-1}\dx<\infty\qquad\quad  \mbox{for all} \ \ 1\le t<\infty,
$$
that eventually embed into all Sobolev spaces (except of course $W^{1,\infty}$). This immersion requires "sacrificing" the very small $\epsilon>0$ in \eqref{fff.f} which makes the differentiability rate larger than one. Such a tiny differentiability amount gets lost in the iteration previously employed in \cite{jw}, while here it is preserved to maximise the nonuniformity range at the cost of the appearance of a power of the $L^{\infty}$-norm of the gradient among the bounding constants. The drawback is that to achieve meaningful estimates, our linear iteration needs to be coupled with a De Giorgi type iteration that grants the possibility of reabsorbing constants depending on $\nr{D\ti{u}}_{L^{\infty}}$ and eventually leads to uniform Lipschitz bounds. It is anyway possible to directly apply \cite[Proposition 7.1]{jw} with the formal choice $p=2-\mu$ there and follow our strategy to deal with the obstacle constraint to derive gradient $L^{t}$-$L^{1}$ estimates for all $1<t<\infty$ and no dependency on $\nr{D\ti{u}}_{L^{\infty}}$. The price to pay is a rather severe restriction on the nonuniformity range: $q<1+\alpha/2$ is required not only in this case, cf. \cite[Display (7.9)]{jw}, but whenever fractional differentiability is involved in subquadratic problems, cf. \cite[Section 3.6, Step 5]{CM2} and \cite[Theorem 3]{dmon}. Nonetheless it is not difficult to see that our approach to nondifferentiable obstacle problems combined with the fractional Moser's iteration in \cite{jw} extends \cite[Theorem 1]{jw} to (quantitatively superlinear) $(p,q)$-nonuniformly elliptic obstacle problems with H\"older continuous coefficients as well as to nonautonomous, nonuniformly elliptic variational inequalities, see \cite[Section 11]{ciccio} and references therein for the Sobolev-differentiable case.}
\end{remark}
\section{Schauder estimates}\label{sce}
\noindent In this section we design a hybrid comparison scheme eventually leading to uniform Lipschitz estimates for solutions to problem \eqref{pdreg}. Throughout the whole section, we work within the setting described at the beginning of Section \ref{ncz} - in particular, \eqref{1qa}, \eqref{ass.2}, \eqref{p2}, \eqref{assfr}, \eqref{sasa}, \eqref{u0} will always be in force together with the first restriction $1\le \mu\le 3/2$.
\subsection{First order scaling}\label{fos} 
Since we are after gradient boundedness, we need to adopt a different scaling than the one used in Section \ref{sca} in order to preserve first order information. Let $u\in (u_{0}+W^{1,q}_{0}(B))\cap \tx{K}^{\psi}(B)$ be the solution to problem \eqref{pdreg}, fix a ball $B_{\rr}(x_{0})\Subset B$ with radius $\rr\in (0,1]$, blow up $u$ and $\psi$ on $B_{\rr}(x_{0})$ by letting $\ti{u}(x):=\rr^{-1}u(x_{0}+\rr x)$, $\ti{\psi}(x):=\rr^{-1}\psi(x_{0}+\rr x)$, set this time $\ti{\tx{H}}(x,z):=\tx{L}_{*}(z)+\ti{\tx{a}}(x)\ell_{s}(z)^{q}$, where $\ti{\tx{a}}$ is the same function defined in Section \ref{sca}, integrand $\tx{L}_{*}$ satisfies \eqref{assfr}, and notice that by scaling and \eqref{areg}, $\ti{u}\in C^{1}(B_{1}(0))\cap\tx{K}^{\ti{\psi}}(B_{1}(0))$ is a constrained local minimizer of functional \eqref{fbw} solving variational inequality \eqref{vler}, of course with the current definition of $\ti{\tx{H}}$, $\ti{u}$, $\ti{\psi}$ replacing those given in Section \ref{sca}.
\subsection{Auxiliary integrands and their eigenvalues}\label{aux}
Before proceeding with our regularity estimates, following \cite[Section 3]{dm} we need to introduce some auxiliary quantities. Let $B\subseteq B_{1}$ be a ball, and, with reference to Section \ref{fos}, we set
$$
\ti{\tx{a}}_{i}(B):=\inf_{x\in B}\ti{\tx{a}}(x),\qquad \quad \tx{H}_{i}(z;B):=\tx{L}_{*}(z)+\ti{\tx{a}}_{i}(B)\ell_{s}(z)^{q},
$$
and introduce
$$
\begin{cases}
\ \ti{\lambda}_{1}(x,\snr{z}):=\ell_{1}(z)^{-\mu}+\ti{\tx{a}}(x)\ell_{s}(z)^{q-2}\\
\ \ti{\lambda}_{2}(x,\snr{z}):=\ell_{1}(z)^{-1}(\tx{g}(\snr{z})+1)+\ti{\tx{a}}(x)\ell_{s}(z)^{q-2}\\
\ \lambda_{1;i}(\snr{z};B):=\ell_{1}(z)^{-\mu}+\ti{\tx{a}}_{i}(B)\ell_{s}(z)^{q-2}\\
\ \lambda_{2;i}(\snr{z};B):=\ell_{1}(z)^{-1}(\tx{g}(\snr{z})+1)+\ti{\tx{a}}_{i}(B)\ell_{s}(z)^{q-2}\\
\ \mathcal{V}_{i}^{2}(z_{1},z_{2};B):=\snr{V_{1,2-\mu}(z_{1})-V_{1,2-\mu}(z_{2})}^{2}+\ti{\tx{a}}_{i}(B)\snr{V_{s,q}(z_{1})-V_{s,q}(z_{2})}^{2}.
\end{cases}
$$
From \eqref{assfr}, \cite[Remark 2]{dm}, and the definition of $\tx{H}$, $\tx{H}_{i}$, $\mathcal{V}_{i}^{2}$ it follows that
\begin{flalign}\label{eig.4}
\left\{
\begin{array}{c}
\displaystyle
\ \frac{1}{c}\lambda_{1;i}(\snr{z};B)\snr{\xi}^{2}\leq \langle\partial^{2}\tx{H}_{i}(z;B)\xi,\xi\rangle, \qquad \qquad \snr{\partial^{2}\tx{H}_{i}(z;B)}\le c\lambda_{2;i}(\snr{z};B)\\ [10pt]\displaystyle
\ \frac{1}{c}\mathcal{V}^{2}_{i}(z_{1},z_{2};B)\le \langle\partial \tx{H}_{i}(z_{1};B)-\partial\tx{H}_{i}(z_{2};B),z_{1}-z_{2}\rangle,
\end{array}
\right.
\end{flalign}
for all $z,\xi\in \mathbb{R}^{n}$, with $c\equiv c(\texttt{data}_{*})$. In particular it is \eqn{elrat}
$$
\frac{\lambda_{2;i}(\snr{z};B)}{\lambda_{1;i}(\snr{z};B)}\le \ell_{1}(z)^{\mu-1}(\tx{g}(\snr{z})+1)\qquad \mbox{for all} \ \ z\in \mathbb{R}^{n}.
$$
A crucial role in the forthcoming computations will be played by the primitives of $\ti{\lambda}_{1}(x,t)t$ and $\lambda_{1;i}(t;B)t$. For $(x,z)\in\Omega\times \mathbb{R}^{n}$, we indeed set
\eqn{eii}
$$
\left\{
\begin{array}{c}
\displaystyle
\ \ti{\tx{E}}(x,\snr{z}):=\int_{0}^{\snr{z}}\ti{\lambda}_{1}(x,t)t\dtt=\frac{1}{2-\mu}\left(\ell_{1}(z)^{2-\mu}-1\right)+\frac{\ti{\tx{a}}(x)}{q}\left(\ell_{s}(z)^{q}-s^{q}\right)\\[10pt]\displaystyle
\ \tx{E}_{i}(\snr{z};B):=\int_{0}^{\snr{z}}\lambda_{1;i}(t;B)t\dtt=\frac{1}{2-\mu}\left(\ell_{1}(z)^{2-\mu}-1\right)+\frac{\ti{\tx{a}}_{i}(B)}{q}\left(\ell_{s}(z)^{q}-s^{q}\right),
\end{array}
\right.
$$
keep in mind the definitions in \eqref{hlhl}, and recall \cite[Lemma 3.1]{dm}, containing the basic properties of $\ti{\tx{E}}$, $\tx{E}_{*}$, and $\tx{E}_{i}$ - recall that $\tx{E}_{*}$ has been defined in \eqref{hlhl}.
\begin{lemma}\label{leee} The following holds about the functions in \eqref{eii}, for all $z,z_{1},z_{2}\in \mathbb{R}^{n}$, $x\in B$:
\begin{itemize}
\item the Lipschitz estimate
\begin{flalign}\label{0.000}
&\snr{\tx{E}_{i}(\snr{z_{1}};B)-\tx{E}_{i}(\snr{z_{2}};B)}\nonumber \\
&\qquad \qquad \quad \lesssim_{\mu,q}\left[(\snr{z_{1}}^2+\snr{z_{2}}^2+1)^{(1-\mu)/2}+\ti{\tx{a}}_{i}(B)(\snr{z_{1}}^2+\snr{z_{2}}^2+s^2)^{(q-1)/2}\right]\left|\snr{z_{1}}-\snr{z_{2}}\right|
\end{flalign}
holds true;
\item the oscillation estimate
\eqn{0.00}
$$
\snr{\ti{\tx{E}}(x,\snr{z})-\tx{E}_{i}(\snr{z};B)}\lesssim_{q}  \snr{\ti{\tx{a}}(x)-\ti{\tx{a}}_{i}(B)}\left(\ell_{s}(z)^{q}-s^q\right)
$$
is verified;
\item there exists constant $\tx{T}_{\mu}\equiv \tx{T}_{\mu}(\mu)\geq 1$ such that 
\eqn{0.001}
$$
\snr{z}\le 2 \tx{E}_{*}(x,\snr{z})^{\frac{1}{2-\mu}}, \qquad  \quad \snr{z}\le 2\ti{\tx{E}}(x,\snr{z})^{\frac{1}{2-\mu}},\qquad \quad \snr{z}\le  2\tx{E}_{i}(\snr{z};B)^{\frac{1}{2-\mu}},
$$
whenever $\snr{z}\ge \tx{T}_{\mu}$;
\item there exists a constant $c\equiv c(\mu, q,\tx{g})$ such that 
\eqn{0.002}
$$
\begin{cases}
\ |z|+\tx{E}_{*}(x,\snr{z})\leq c \tx{H}_{*}(x,z)+c\\
\ |z|+\ti{\tx{E}}(x,\snr{z}) \leq c \ti{\tx{H}}(x,z)+c\\
\ \snr{z}+\tx{E}_{i}(\snr{z};B)\le c\tx{H}_{i}(z;B)+c,
 \end{cases}
$$
is true. 
\end{itemize}
\end{lemma}
\subsection{Nearly homogeneous reference estimates for nonhomogeneous integrals}\label{nnh}
In this section, we derive some $L^{\infty}$ and energy estimates for minima of autonomous, nonhomogeneous variational integrals. More precisely, let $B_{\rrr}(\equiv B_{\rrr}(x_{B}))\Subset B_{1}$ be a ball with radius $\rrr\in (0,\rrr_{*}]$, $\tx{f}\in L^{\infty}(B_{1})$, and $v_{0}\in W^{1,\infty}(B_{\rrr})$ be functions. The threshold radius $\rrr_{*}\in (0,1]$ is taken so small that $\tx{c}_{\tx{p}}\rrr_{*}\Lambda_{*}\nr{\tx{f}}_{L^{\infty}(B_{1})}\le 1/2,$ where symbol $\tx{c}_{\tx{p}}\equiv \tx{c}_{\tx{p}}(n)$ denotes the constant from Poincar\'e inequality in $W^{1,1}(B_{1})$. Within this setting, in the next proposition we prove some Lipschitz and energy type reference estimates for the solution $v\in v_{0}+W^{1,q}_{0}(B_{\rrr})$ to Dirichlet problem
\eqn{pdnh}
$$
v_{0}+W^{1,q}_{0}(B_{\rrr})\ni w\mapsto \int_{B_{\rrr}}\tx{H}_{i}(Dw;B_{\rrr})-\tx{f}w\dx,
$$
which exists and it is unique by superlinearity, standard direct methods and strict convexity arguments. 
\begin{proposition}\label{p41}
Let $v\in v_{0}+W^{1,q}_{0}(B_{\rrr})$ be the solution of Dirichlet problem \eqref{pdnh}. The energy estimate
\eqn{enes}
$$
\int_{B_{\rrr}}\tx{H}_{i}(Dv;B_{\rrr})\dx\le c\int_{B_{\rrr}}\tx{H}_{i}(Dv_{0};B_{\rrr})+1\dx,
$$
is verified with $c\equiv c(n,\Lambda_{*},\tx{g})$. Moreover $v\in W^{1,\infty}_{\loc}(B_{\rrr})\cap W^{2,2}_{\loc}(B_{\rrr})$, and for every $\delta_{0},\delta_{1},\delta_{2}\in (0,(2-q)/(80n))$, there exists a threshold $\bar{\mu}\equiv \bar{\mu}(n,\delta_{0},\delta_{1},\delta_{2})>1$ such that 
$$
\begin{cases}
\ (\delta_{0},\delta_{1},\delta_{2})\mapsto \bar{\mu}(n,\delta_{0},\delta_{1},\delta_{2}) \ \ \mbox{is increasing}\\
\ \bar{\mu}\searrow 1 \ \ \mbox{if at least one of} \ \delta_{0},\delta_{1},\delta_{2} \  \mbox{vanishes},
\end{cases}
$$
and if $1\le \mu<\bar{\mu}$, the $L^{\infty}$-estimates
\begin{eqnarray}\label{nh.2}
\nr{Dv}_{L^{\infty}(B_{3\rrr/4})}&\le& c\left(1+\nr{\tx{f}}_{L^{\infty}(B_{\rrr})}^{\tx{m}_{0}}\right)\tx{E}_{i}\left(\nr{Dv_{0}}_{L^{\infty}(B_{\rrr})};B_{\rrr}\right)^{\frac{\delta_{0}}{2-\mu}}\nr{Dv_{0}}_{L^{\infty}(B_{\rrr})}\nonumber \\
&&+c\left(1+\nr{\tx{f}}_{L^{\infty}(B_{\rrr})}^{\tx{m}_{0}}\right),
\end{eqnarray}
and
\eqn{nh.2.1}
$$
\nr{Dv}_{L^{\infty}(B_{3\rrr/4})}\le c\left(1+\nr{\tx{f}}_{L^{\infty}(B_{\rrr})}^{\tx{m}_{0}}\right)\left(1+\nr{Dv_{0}}_{L^{\infty}(B_{\rrr})}^{4q}\right),
$$
hold with $c\equiv c(\data_{*},\nr{a}_{L^{\infty}(B_{\rr}(x_{0}))},\delta_{0},\delta_{1},\delta_{2})$, $\tx{m}_{0}\equiv \tx{m}_{0}(n,\delta_{0},\delta_{1},\delta_{2})$. Finally, whenever $B\Subset B_{\rrr}(x_{B})$ is a ball, and $\tx{M}$ is a positive number such that
\eqn{mmm}
$$
 \max\left\{\nr{Dv}_{L^{\infty}(B)},10+\tx{T}_{\tx{g}}+\tx{T}_{\mu}\right\}\le \tx{M},
$$
the Caccioppoli type inequality
\begin{eqnarray}\label{nh.1}
\nr{D(\tx{E}_{i}(\snr{Dv};B_{\rrr})-\kk)_{+}}_{L^{2}(3B/4)}&\le&\frac{c\tx{M}^{\delta_{1}}}{\snr{B}^{1/n}}\nr{(\tx{E}_{i}(\snr{Dv};B_{\rrr})-\kk)_{+}}_{L^{2}(B)}\nonumber \\
&&+c\tx{M}^{1-\delta_{2}}\nr{\tx{f}}_{L^{\infty}(B_{\rrr})}\nr{\ell_{1}(Dv)^{\delta_{2}}}_{L^{2}(B)}
\end{eqnarray}
is satisfied for $c\equiv c(\data_{*},\delta_{0},\delta_{1},\delta_{2})$. 
\end{proposition}
\begin{proof}
    The minimality of $v\in v_{0}+W^{1,q}_{0}(B_{\rrr})$ grants the validity of the integral identity
\eqn{iid}
$$
\int_{B_{\rrr}}\langle\partial \tx{H}_{i}(Dv;B_{\rrr}),Dw\rangle-\tx{f}w\dx=0\qquad \mbox{for all} \ \ w\in W^{1,q}_{0}(B_{\rrr}),
$$
as well as of estimate
\begin{eqnarray*}
\int_{B_{\rrr}}\tx{H}_{i}(Dv;B_{\rrr})\dx&\le&\int_{B_{\rrr}}\tx{H}_{i}(Dv_{0};B_{\rrr})\dx+\int_{B_{\rrr}}\snr{\tx{f}}\snr{v-v_{0}}\dx\nonumber \\
&\stackrel{\eqref{gt}}{\le}&\left(1+\tx{c}_{\tx{p}}\Lambda_{*}\nr{\tx{f}}_{L^{\infty}(B_{\rrr})}\rrr\right)\int_{B_{\rrr}}\tx{H}_{i}(Dv_{0};B_{\rrr})\dx+4\rrr\tx{c}_{\tx{p}}\Lambda_{*}^{2}\tx{T}_{\tx{g}}\nr{\tx{f}}_{L^{\infty}(B_{\rrr})}\snr{B_{\rrr}}\nonumber \\
&&+\tx{c}_{\tx{p}}\Lambda_{*}\nr{\tx{f}}_{L^{\infty}(B_{\rrr})}\rrr\int_{B_{\rrr}}\tx{L}_{*}(Dv)\dx,
\end{eqnarray*}
so thanks to the restrictions imposed on the threshold radius $\rrr_{*}$ we obtain \eqref{enes}. Growth conditions \eqref{assfr.2}$_{1,2,3}$, that obviously fit the frozen integrand $\tx{H}_{i}$, and standard regularity theory \cite[Chapter 8]{giu}, \cite[Section 8]{BM} imply that 
\eqn{aar}
$$
v\in W^{1,\infty}_{\loc}(B_{\rrr})\cap W^{2,2}_{\loc}(B_{\rrr})\qquad \mbox{and}\qquad \partial \tx{H}_{i}(Dv;B_{\rrr})\in W^{1,2}_{\loc}(B_{\rrr},\mathbb{R}^{n}),
$$
therefore equation \eqref{iid} can be further differentiated and
\eqn{2v}
$$
\int_{B_{\rrr}}\langle\partial^{2}\tx{H}_{i}(Dv;B_{\rrr})D_{s}Dv,Dw\rangle \dx=-\int_{B_{\rrr}}\tx{f}D_{s}w\dx,
$$
holds true for all $s\in \{1,\cdots,n\}$, $w\in W^{1,2}(B_{\rrr})$ such that $\supp(w)\Subset B_{\rrr}$. We then fix any ball $B\Subset B_{\rrr}$, a cut-off function $\eta\in C^{1}_{c}(B)$ with $\mathds{1}_{3B/4}\le \eta\le \mathds{1}_{5B/6}$ and $\snr{D\eta}\lesssim \snr{B}^{-1/n}$, and nonnegative numbers $\kk\ge \kk_{0}\ge 0$, and test \eqref{2v} against map $w_{\kk}:=\eta^{2}(\tx{E}_{i}(\snr{Dv};B_{\rrr})-\kk)_{+}D_{s}v$, admissible thanks to \eqref{aar}, and sum over $s\in \{1,\cdots,n\}$ to get via \eqref{eig.4}-\eqref{eii}, \eqref{ass.2} and Young inequality:
\begin{eqnarray}\label{16}
\int_{B}\eta^{2}\snr{D(\tx{E}_{i}(\snr{Dv};B_{\rrr})-\kk)_{+}}^{2}\dx&\le& c\int_{B}\left(\frac{\lambda_{2;i}(\snr{Dv};B_{\rrr})}{\lambda_{1;i}(\snr{Dv};B_{\rrr})}\right)^{2}(\tx{E}_{i}(\snr{Dv};B_{\rrr})-\kk)_{+}^{2}\snr{D\eta}^{2}\dx\nonumber \\
&&+c\int_{B}\tx{f}^{2}\eta^{2}\ell_{1}(Dv)^{2}\dx\nonumber \\
&\le& c\int_{B}(\tx{g}(\snr{Dv})+1)^{2}\ell_{1}(Dv)^{2(\mu-1)}(\tx{E}_{i}(\snr{Dv};B_{\rrr})-\kk)_{+}^{2}\snr{D\eta}^{2}\dx\nonumber \\
&&+c\int_{B}\tx{f}^{2}\eta^{2}\ell_{1}(Dv)^{2}\dx\nonumber \\
&\le&\frac{c\tx{M}^{2(\mu-1+\varepsilon_{1})}}{\snr{B}^{2/n}}\int_{B}(\tx{E}_{i}(\snr{Dv};B_{\rrr})-\kk)_{+}^{2}\dx\nonumber \\
&&+c\tx{M}^{2(1-\varepsilon_{2})}\nr{\tx{f}}_{L^{\infty}(B_{\rrr})}^{2}\int_{B}\eta^{2}\ell_{1}(Dv)^{2\varepsilon_{2}}\dx,
\end{eqnarray}
with $\tx{M}$ as in \eqref{mmm}, $\varepsilon_{1},\varepsilon_{2}\in (0,1)$ and $c\equiv c(\data_{*},\varepsilon_{1},\varepsilon_{2})$, see also \cite[Proposition 4.1]{dm}. Choosing
\eqn{md}
$$
1\le \mu<1+\frac{\delta_{1}}{2},\qquad \quad\quad 0<\varepsilon_{1}<\frac{\delta_{1}}{2},\qquad \quad\quad  \varepsilon_{2}=\delta_{2},
$$
and recalling that $\eta\equiv 1$ on $3B/4$ we obtain \eqref{nh.1} with constants now depending on $(\data_{*},\delta_{1},\delta_{2})$. Back to \eqref{16}, Sobolev embedding theorem eventually yields
\begin{eqnarray}\label{17}
\left(\mint_{B/2}(\tx{E}_{i}(\snr{Dv};B_{\rrr})-\kk)_{+}^{2\chi_{0}}\dx\right)^{\frac{1}{2\chi_{0}}}&\le&c\tx{M}^{(\mu-1+\varepsilon_{1})}\left(\mint_{B}(\tx{E}_{i}(\snr{Dv};B_{\rrr})-\kk)_{+}^{2}\dx\right)^{\frac{1}{2}}\nonumber \\
&&+c\tx{M}^{1-\delta_{2}}\nr{\tx{f}}_{L^{\infty}(B_{\rrr})}\snr{B}^{\frac{\gamma}{n}}\left(\mint_{B}\ell_{1}(Dv)^{2\delta_{2}}\dx\right)^{\frac{1}{2}},
\end{eqnarray}
where $1<\chi_{0}\equiv \chi_{0}(n)<4$ comes from Sobolev exponent, $\gamma=1$ if $n\ge 3$, and $\gamma=1/2$ if $n=2$, and it is $c\equiv c(\data_{*},\varepsilon_{1},\delta_{2})$. With $\varsigma>0$, we then fix parameters $3\varsigma/4\le \tau_{2}<\tau_{1}\le 5\varsigma/6$ and related concentric balls $B_{3\varsigma/4}(x_{\rm c})\subset B_{\tau_2}(x_{\rm c})\subset B_{\tau_1}(x_{\rm c})\subset B_{5\varsigma/6}(x_{\rm c})\subset B_{\varsigma}(x_{\rm c})\subseteq B_{\rrr}$. Pick an arbitrary $x_{0}\in B_{\tau_{2}}(x_{\ccc})$ and set $r_{0}:=(\tau_{1}-\tau_{2})/8$, so that $ B_{r_{0}}(x_0) \subset B_{\tau_1}(x_{\ccc})$, observe that there is no loss of generality in assuming $\nr{Dv}_{L^{\infty}(B_{3\varsigma/4}(x_{\ccc}))}\ge 10+\tx{T}_{\tx{g}}+\tx{T}_{\mu}$ (otherwise there would be nothing to prove), choose $\tx{M}\equiv \nr{Dv}_{L^{\infty}(B_{\tau_1}(x_{\ccc}))}$ and take $B\equiv B_{\rr}(x_0)\subseteq B_{r_0}(x_0)$ being an arbitrary concentric ball in \eqref{17}.~Lemma \ref{revlem} applies with $M_{0}=\tx{M}^{\mu-1+\varepsilon_{1}}$, $M_{1}= \tx{M}^{1-\delta_{2}}\nr{\tx{f}}_{L^{\infty}(B_{\rrr})}$, $f=\ell_{1}(Dv)^{2\delta_{2}}$, $\sigma=\gamma$, $\vartheta=1/2$, $v\equiv \tx{E}_{i}(\snr{Dv};B_{\rrr})$, and gives
\begin{eqnarray*}
    \tx{E}_{i}(\snr{Dv(x_{0})};B_{\rrr})&\le&c+ c\tx{M}^{\frac{\chi_{0}(\mu-1+\varepsilon_{1})}{\chi_{0}-1}}\nra{\tx{E}_{i}(\snr{Dv};B_{\rrr})}_{L^{2}(B_{r_{0}}(x_{0}))}\nonumber \\
    &&+c\tx{M}^{\frac{\mu-1+\varepsilon_{1}}{\chi_{0}-1}+(1-\delta_{2})}\nr{\tx{f}}_{L^{\infty}(B_{\rrr})}\mathbf{P}^{\frac{1}{2}}_{\gamma}\left(\ell_{1}(Dv)^{2\delta_{2}};x_{0},2r_{0}\right),
\end{eqnarray*}
for $c\equiv c(\data_{*},\varepsilon_{1},\delta_{2})$. By \eqref{stimazza} with $m=n+1>n/(2\gamma)> 1$, the arbitrariety of $x_{0}\in B_{\tau_{2}}(x_{\ccc})$, and \eqref{0.001}$_{3}$ implies
\begin{flalign}\label{19}
    &\tx{E}_{i}\left(\nr{Dv}_{L^{\infty}(B_{\tau_{2}}(x_{\ccc}))};B_{\rrr}\right)\nonumber \\
    &\qquad \qquad \quad\le \frac{c}{(\tau_{1}-\tau_{2})^{n/2}}\nr{Dv}_{L^{\infty}(B_{\tau_{1}}(x_{\ccc}))}^{\frac{\chi_{0}(\mu-1+\varepsilon_{1})}{\chi_{0}-1}}\tx{E}_{i}\left(\nr{Dv}_{L^{\infty}(B_{\tau_{1}}(x_{\ccc}))};B_{\rrr}\right)^{\frac{1}{2}}\nr{\tx{E}_{i}(\snr{Dv};B_{\rrr})}_{L^{1}(B_{5\varsigma/6}(x_{\ccc}))}^{\frac{1}{2}}\nonumber \\
    &\qquad \qquad \quad\quad +c\nr{Dv}_{L^{\infty}(B_{\tau_{1}}(x_{\ccc}))}^{\frac{\mu-1+\varepsilon_{1}}{\chi_{0}-1}+1-\delta_{2}}\nr{\tx{f}}_{L^{\infty}(B_{\rrr})}\left\|\mathbf{P}^{\frac{1}{2}}_{\gamma}\left(\ell_{1}(Dv)^{2\delta_{2}};\cdot,(\tau_{1}-\tau_{2})/4\right)\right\|_{L^{\infty}(B_{\tau_{2}}(x_{\ccc}))}+c\nonumber \\
    &\qquad \qquad \quad\le \frac{c}{(\tau_{1}-\tau_{2})^{n/2}}\tx{E}_{i}\left(\nr{Dv}_{L^{\infty}(B_{\tau_{1}}(x_{\ccc}))};B_{\rrr}\right)^{\frac{\chi_{0}(\mu-1+\varepsilon_{1})}{(\chi_{0}-1)(2-\mu)}+\frac{1}{2}}\nr{\tx{E}_{i}(\snr{Dv};B_{\rrr})}_{L^{1}(B_{5\varsigma/6}(x_{\ccc}))}^{\frac{1}{2}}\nonumber \\
     &\qquad \qquad \quad\quad +c\tx{E}_{i}\left(\nr{Dv}_{L^{\infty}(B_{\tau_{1}}(x_{\ccc}))};B_{\rrr}\right)^{\frac{\mu-1+\varepsilon_{1}+(\chi_{0}-1)(1-\delta_{2})}{(\chi_{0}-1)(2-\mu)}}\nr{\tx{f}}_{L^{\infty}(B_{\rrr})}\nr{\ell_{1}(Dv)^{2-\mu}}_{L^{1}(B_{5\varsigma/6}(x_{\ccc}))}^{\frac{\delta_{2}}{2-\mu}}+c,
\end{flalign}
where we used that $1\le \mu\le 3/2$ and $\delta_{2}<(2-q)/(80n)<1/(80n)$ thus $2\delta_{2}(n+1)<1/2\le 2-\mu$, and it is $c\equiv c(\data_{*},\varepsilon_{1},\delta_{2})$. Enforcing the bounds imposed on the size of $\mu,\varepsilon_{1}$ as
\eqn{ddd}
$$
\left\{
\begin{array}{c}
\displaystyle
\ 1\le \mu<1+\min\left\{\frac{\delta_{1}}{80},\frac{\delta_{2}(\chi_{0}-1)}{4\chi_{0}},\frac{\delta_{0}(\chi_{0}-1)}{4(4\chi_{0} +\delta_{0}(3\chi_{0}-1))},\frac{\delta_{0}\delta_{2}(\chi_{0}-1)}{4\chi_{0}(2+\delta_{0})},\frac{\delta^{2}_{0}}{2^{10}}\right\}=:\bar{\mu}\\[10pt]\displaystyle
\ 0<\varepsilon_{1}<\min\left\{\frac{\delta_{1}}{80},\frac{\delta_{2}(\chi_{0}-1)}{4},\frac{\delta_{0}(\chi_{0}-1)}{4\chi_{0}(2+\delta_{0})},\frac{\delta_{0}\delta_{2}(\chi_{0}-1)}{4(2+\delta_{0})}\right\},
\end{array}
\right.
$$
we obtain:\footnote{Notice that since $\delta_{0},\delta_{1},\delta_{2}\in (0,(2-q)/(80n))$ all quantities involved in \eqref{ddd} are well-defined.}
\begin{flalign*}
\begin{array}{c}
\displaystyle
\frac{\chi_{0}(\mu-1+\varepsilon_{1})}{(\chi_{0}-1)(2-\mu)}+\frac{1}{2}<\frac{1+\delta_{0}}{2+\delta_{0}},\qquad\quad \frac{\mu-1+\varepsilon_{1}+(\chi_{0}-1)(1-\delta_{2})}{(\chi_{0}-1)(2-\mu)}<1,\\[10pt]\displaystyle
\frac{\delta_{2}(\chi_{0}-1)}{(\chi_{0}-1)\delta_{2}-\varepsilon_{1}-\chi_{0}(\mu-1)}<1+\frac{\delta_{0}}{2},
\end{array}
\end{flalign*}
so in \eqref{19} we can apply Young inequality with conjugate exponents $$\left(\frac{2+\delta_{0}}{1+\delta_{0}},2+\delta_{0}\right),\qquad \quad \left(\frac{(\chi_{0}-1)(2-\mu)}{(\chi_{0}-1)(1-\delta_{2})+\mu-1+\varepsilon_{1}},\frac{(\chi_{0}-1)(2-\mu)}{\delta_{2}(\chi_{0}-1)-\chi_{0}(\mu-1)-\varepsilon_{1}}\right),$$
and recall that
$$
\nr{Dv}_{L^{\infty}(B_{\tau_{2}}(x_{\ccc}))}\ge 10+\tx{T}_{\tx{g}}+\tx{T}_{\mu} \ \stackrel{\eqref{0.001}_{3}}{\Longrightarrow} \ \tx{E}_{i}\left(\nr{Dv}_{L^{\infty}(B_{\tau_{2}}(x_{\ccc}))};B_{\rrr}\right)\ge 1
$$
to get
\begin{flalign}\label{20}
 &\tx{E}_{i}\left(\nr{Dv}_{L^{\infty}(B_{\tau_{2}}(x_{\ccc}))};B_{\rrr}\right)\nonumber \\
 &\qquad \qquad \quad\le \frac{c}{(\tau_{1}-\tau_{2})^{n/2}}\tx{E}_{i}\left(\nr{Dv}_{L^{\infty}(B_{\tau_{1}}(x_{\ccc}))};B_{\rrr}\right)^{\frac{1+\delta_{0}}{2+\delta_{0}}}\nr{\tx{E}_{i}(\snr{Dv};B_{\rrr})}_{L^{1}(B_{5\varsigma/6}(x_{\ccc}))}^{\frac{1}{2}}+c\nonumber \\
 &\qquad \qquad \quad \quad +c\nr{\tx{f}}_{L^{\infty}(B_{\rrr})}\tx{E}_{i}\left(\nr{Dv}_{L^{\infty}(B_{\tau_{1}}(x_{\ccc}))};B_{\rrr}\right)^{\frac{\mu-1+\varepsilon_{1}+(\chi_{0}-1)(1-\delta_{2})}{(\chi_{0}-1)(2-\mu)}}\nr{\ell_{1}(Dv)^{2-\mu}}_{L^{1}(B_{5\varsigma/6}(x_{\ccc}))}^{\frac{\delta_{2}}{2-\mu}}\nonumber \\
 &\qquad \qquad \quad\le \frac{1}{4}\tx{E}_{i}\left(\nr{Dv}_{L^{\infty}(B_{\tau_{1}}(x_{\ccc}))};B_{\rrr}\right)+\frac{c}{(\tau_{1}-\tau_{2})^{n(1+\delta_{0}/2)}}\nr{\tx{E}_{i}(\snr{Dv};B_{\rrr})}_{L^{1}(B_{5\varsigma/6}(x_{\ccc}))}^{1+\frac{\delta_{0}}{2}}\nonumber \\
 &\qquad \qquad \quad\quad +c\left(\nr{\tx{f}}_{L^{\infty}(B_{\rrr}(x_{B}))}\nr{\ell_{1}(Dv)^{2-\mu}}_{L^{1}(B_{5\varsigma/6}(x_{\ccc}))}^{\frac{\delta_{2}}{2-\mu}}\right)^{\frac{(\chi_{0}-1)(2-\mu)}{(\chi_{0}-1)\delta_{2}-\chi_{0}(\mu-1)-\varepsilon_{1}}}+c\nonumber \\
 &\qquad \qquad \quad \le \frac{1}{4}\tx{E}_{i}\left(\nr{Dv}_{L^{\infty}(B_{\tau_{1}}(x_{\ccc}))};B_{\rrr}\right)+c\left(1+\nr{\tx{f}}_{L^{\infty}(B_{\rrr})}^{\frac{(\chi_{0}-1)(2-\mu)}{(\chi_{0}-1)\delta_{2}-\chi_{0}(\mu-1)-\varepsilon_{1}}}\right)\nonumber \\
 &\qquad \qquad \quad \quad +\frac{c}{(\tau_{1}-\tau_{2})^{n(1+\delta_{0}/2)}}\left(1+\nr{\tx{f}}_{L^{\infty}(B_{\rrr})}^{\frac{(\chi_{0}-1)(2-\mu)}{(\chi_{0}-1)\delta_{2}-\chi_{0}(\mu-1)-\varepsilon_{1}}}\right)\nr{\tx{E}_{i}(\snr{Dv};B_{\rrr})}_{L^{1}(B_{5\varsigma/6}(x_{\ccc}))}^{1+\frac{\delta_{0}}{2}},
\end{flalign}
with $c\equiv c(\data_{*},\delta_{0},\delta_{1},\delta_{2})$. An application of Lemma \ref{iterlem} yields
\begin{flalign}\label{21}
\tx{E}_{i}\left(\nr{Dv}_{L^{\infty}(B_{3\varsigma/4}(x_{\ccc}))};B_{\rrr}\right)\le c\left(1+\nr{\tx{f}}_{L^{\infty}(B_{\rrr})}^{\tx{m}_{0}(2-\mu)}\right)\nra{\tx{E}_{i}(\snr{Dv};B_{\rrr})}_{L^{1}(B_{5\varsigma/6}(x_{\ccc}))}^{1+\frac{\delta_{0}}{2}}+c\left(1+\nr{\tx{f}}_{L^{\infty}(B_{\rrr})}^{\tx{m}_{0}(2-\mu)}\right),
\end{flalign}
where we set $\tx{m}_{0}:=(\chi_{0}-1)((\chi_{0}-1)\delta_{2}-\chi_{0}(\mu-1)-\varepsilon_{1})^{-1}$, and it is $c\equiv c(\data_{*},\delta_{0},\delta_{1},\delta_{2})$. Letting in particular $B_{\varsigma}(x_{\ccc})\equiv B_{\rrr}$ and setting for simplicity $\tx{F}:=1+\nr{\tx{f}}_{L^{\infty}(B_{\rrr})}^{\tx{m}_{0}(2-\mu)}$, we can further manipulate \eqref{21} to get:
\begin{eqnarray*}
\tx{E}_{i}\left(\nr{Dv}_{L^{\infty}(B_{3\rrr/4})};B_{\rrr}\right)&\stackrel{\eqref{0.002}_{3}}{\le}& c\tx{F}\nra{\tx{H}_{i}(Dv;B_{\rrr})}_{L^{1}(B_{\rrr})}^{1+\frac{\delta_{0}}{2}}+c\tx{F}\nonumber \\
&\stackrel{\eqref{enes}}{\le}&c\tx{F}\nra{\tx{H}_{i}(Dv_{0};B_{\rrr})}_{L^{1}(B_{\rrr})}^{1+\frac{\delta_{0}}{2}}+c\tx{F}\nonumber \\
&\stackrel{\eqref{assfr}_{2}}{\le}&c\tx{F}\nr{\tx{H}_{i}(Dv_{0};B_{\rrr})}_{L^{\infty}(B_{\rrr})}^{\frac{\delta_{0}}{2}}\tx{g}\left(\nr{Dv_{0}}_{L^{\infty}(B_{\rrr})}\right)\nr{Dv_{0}}_{L^{\infty}(B_{\rrr})}\nonumber \\
&&+c\tx{F}\nr{\tx{H}_{i}(Dv_{0};B_{\rrr})}_{L^{\infty}(B_{\rrr})}^{\frac{\delta_{0}}{2}}\ti{\tx{a}}_{i}(B_{\rrr})\ell_{s}\left(\nr{Dv_{0}}_{L^{\infty}(B_{\rrr})}\right)^{q}+c\tx{F}\nonumber \\
&\stackrel{\eqref{ass.2},\eqref{0.001}_{3}}{\le}&c\tx{F}\nr{\tx{H}_{i}(Dv_{0};B_{\rrr})}_{L^{\infty}(B_{\rrr})}^{\frac{3\delta_{0}}{4}+\mu-1}\tx{E}_{i}\left(\nr{Dv_{0}}_{L^{\infty}(B_{\rrr})};B_{\rrr}\right)\nonumber \\
&&+c\tx{F}\nr{\tx{H}_{i}(Dv_{0};B_{\rrr})}_{L^{\infty}(B_{\rrr})}^{\frac{\delta_{0}}{2}}\tx{E}_{i}\left(\nr{Dv_{0}}_{L^{\infty}(B_{\rrr})};B_{\rrr}\right)+c\tx{F}\nonumber \\
&\stackrel{\eqref{ddd}}{\le}&c\tx{F}\tx{E}_{i}\left(\nr{Dv_{0}}_{L^{\infty}(B_{\rrr})};B_{\rrr}\right)^{\delta_{0}}\tx{E}_{i}\left(\nr{Dv_{0}}_{L^{\infty}(B_{\rrr})};B_{\rrr}\right)+c\tx{F},
\end{eqnarray*}
for $c\equiv c(\data_{*},\delta_{0},\delta_{1},\delta_{2})$. Observe that $\tx{E}_{i}$ is monotone increasing, unbounded, $\tx{E}_{i}(0;B_{\tx{r}})=0$, and $\tx{E}_{i}(mt;B_{\tx{r}})\ge m^{2-\mu}\tx{E}_{i}(t;B_{\tx{r}})$ for all $m\ge 1$, $t\ge 0$ cf. \eqref{eii}, so its inverse $\tx{E}_{i}^{-1}$ is increasing with $\tx{E}_{i}^{-1}(0;B_{\tx{r}})=0$. Moreover, for all $t,t_{1},t_{2}\ge 0$, it holds $\tx{E}_{i}^{-1}(t_{1}+t_{2};B_{\tx{r}})\lesssim_{\mu} \tx{E}_{i}^{-1}(t_{1};B_{\tx{r}})+\tx{E}_{i}^{-1}(t_{2};B_{\tx{r}})$ and $\tx{E}_{i}^{-1}(t;B_{\tx{r}})\lesssim_{\mu}1+t^{1/(2-\mu)}$. We can therefore apply on both sides of the previous inequality $\tx{E}_{i}^{-1}$ to derive \eqref{nh.2}. Recalling that,
$$
1\le \mu\stackrel{\eqref{ddd}}{\le} 3/2 \ \Longrightarrow \ \tx{E}_{i}(\snr{z};B_{\rrr})\le c+c\snr{z}^{2q},
$$
with $c\equiv c(\mu,q,\nr{a}_{L^{\infty}(B_{\rr}(x_{0}))})$, we further obtain another (very rough) $L^{\infty}$-estimate for $Dv$, i.e. \eqref{nh.2.1}, that is a direct consequence of \eqref{nh.2} and of the previous display. 
The proof is complete.
\end{proof}

\subsection{Reverse H\"older inequality on level sets} Our main goal is the proof of the following lemma, that will be carried out in several steps accounting for various hybrid comparison arguments aimed at establishing the membership of the solution of problem \eqref{pdreg} to suitable fractional De Giorgi classes.
\begin{lemma}
Under assumptions \eqref{1qa}, \eqref{ass.2}-\eqref{p2}, \eqref{assfr}, \eqref{sasa} and \eqref{u0}, let $B\subset 2B\Subset \Omega$ be concentric balls, and $u\in\left(u_{0}+W^{1,q}_{0}(B)\right)\cap \tx{K}^{\psi}(B)$ be the solution of problem \eqref{pdreg}. For any $\delta_{0},\delta_{1},\delta_{2}\in (0,(2-q)/(80n))$, $\kk\ge 0$, there is a threshold $\mu_{\textnormal{max}}\equiv \mu_{\textnormal{max}}(n,q,\alpha,\delta_{0},\delta_{1},\delta_{2})>1$ and positive exponents $\tx{t}_{1}\equiv \tx{t}_{1}(\delta_{1})$, $\tx{t}_{2}\equiv \tx{t}_{2}(\mu,\delta_{0},\delta_{2})$ such that if $1\le \mu<\mu_{\textnormal{max}}$ and $\bar{\tx{M}}\ge 1$ is any number satisfying
\eqn{mmm.2}
$$
\bar{\tx{M}}\ge \max\left\{\nr{\tx{E}_{*}(\cdot,\snr{Du})}_{L^{\infty}(B_{\rr}(x_{0}))},10+\tx{T}_{\tx{g}}+\tx{T}_{\mu}\right\},
$$
the reverse H\"older inequality
    \begin{flalign}\label{frca}
    \nra{(\tx{E}_{*}(\cdot,\snr{Du})-\kk)_{+}}_{L^{2\chi}(B_{\rr/2}(x_{0}))}\le c\bar{\tx{M}}^{\tx{t}_{1}}\nra{(\tx{E}_{*}(\cdot,\snr{Du})-\kk)_{+}}_{L^{2}(B_{\rr}(x_{0}))}+c\bar{\tx{M}}^{\tx{t}_{2}}\rr^{\frac{\alpha}{2}}\nra{\ell_{1}(Du)}_{L^{4}(B_{\rr}(x_{0}))}^{2}
    \end{flalign}
    holds true for some $\chi \equiv \chi(n,\alpha)>1$ and $c\equiv c(\data(B),\delta_{0},\delta_{1},\delta_{2})$. 
\end{lemma}
\begin{proof}
    Let $\ti{u}\in C^{1}(B_{1})\cap \tx{K}^{\ti{\psi}}(B_{1})$ be the rescaled map defined in Section \ref{fos}, set
    $$
    \begin{array}{c}
    \displaystyle
    \tx{f}_{\infty}:=\nr{\diver (\partial \tx{L}_{*}(D\psi))}_{L^{\infty}(B_{\rr}(x_{0}))}+4nq\left(\nr{a}_{L^{\infty}(B_{\rr}(x_{0}))}+1\right)\nr{\snr{D\psi}^{q-2}D^{2}\psi}_{L^{\infty}(B_{\rr}(x_{0}))},\\[8pt]\displaystyle
    \tx{h}_{\infty}:= 1+40(\tx{c}_{\tx{p}}\Lambda_{*}+1)\tx{f}_{\infty},
    \end{array}
    $$
    and observe that
    \eqn{fff}
    $$
    \nr{\diver(\partial \tx{H}_{i}(D\ti{\psi};B))}_{L^{\infty}(B_{1})}\le \rr\tx{f}_{\infty}\le \tx{f}_{\infty}\qquad \mbox{for all balls} \ \ B\subseteq B_{1}.
    $$
     Fix $\beta_{0}\in (0,1)$, pick a vector $h\in \mathbb{R}^{n}$ such that $\snr{h}\in \left(0,\min\{2^{-8/\beta_{0}},\tx{h}_{\infty}^{-1/\beta_{0}}\}\right)$ and with $x_{\ccc}\in B_{\frac{1}{2}+2\snr{h}^{\beta_{0}}}$ introduce balls $B_{h}:=B_{\snr{h}^{\beta_{0}}}(x_{\ccc})$ that, by construction, verify inclusion $8B_{h}\Subset B_{3/4}\subset B_{1}$. We stress that \eqref{areg} assures that numbers
    \eqn{mmm.1}
$$
\tx{M}_{1}:= \max\left\{10+\tx{T}_{\tx{g}}+\tx{T}_{\mu},\nr{Du}_{L^{\infty}(B_{\rr}(x_{0}))}\right\},\quad \quad \tx{M}_{2}:= \max\left\{10+\tx{T}_{\tx{g}}+\tx{T}_{\mu},\nr{\tx{E}_{*}(\cdot,\snr{Du})}_{L^{\infty}(B_{\rr}(x_{0}))}\right\}
$$
are finite. Finally, notice that \eqref{1qa} is equivalent to $1<2+\alpha-q$, which guarantees that a restriction on the size of $\mu$ of the form $1\le \mu<2+\alpha-q$ makes sense. As a consequence,
\eqn{bbbb}
$$
q<2-\mu+\alpha,
$$
so we set\footnote{Position \eqref{muma} is valid thanks to the restriction on $\delta_{0},\delta_{1},\delta_{2}$ in Proposition \ref{p41}. Specifically, while Proposition \ref{p41} only requires $\delta_{0},\delta_{1},\delta_{2}<1/(80n)$, the stronger bound $\delta_{0},\delta_{1},\delta_{2}<(2-q)/(80n)$ is already in place, adopted from the outset to avoid repetition.} 
\eqn{muma}
$$
\mu_{\textnormal{max}}:=\min\left\{\bar{\mu},2-q+\alpha\right\}>1,
$$
where $3/2>\bar{\mu}\equiv \bar{\mu}(n,\delta_{0},\delta_{1},\delta_{2})>1$ is the same limiting quantity devised in Proposition \ref{p41}, and permanently work under condition $1\le \mu<\mu_{\textnormal{max}}$. For the sake of exposition, the reminder of the proof is split in three steps eventually yielding \eqref{frca}.
\subsubsection*{Step 1: comparison} Let $v\in\ti{u} +W^{1,q}_{0}(8B_{h})$ be the solution of problem
\eqn{com.0}
$$
\ti{u}+W^{1,q}_{0}(8B_{h})\ni w\mapsto \min \int_{8B_{h}}\tx{H}_{i}(Dw;8B_{h})-\tx{f}w\dx,
$$
where $\tx{f}:=-\diver(\partial\tx{H}_{i}(D\ti{\psi};8B_{h}))\in L^{\infty}(B_{1})$ by means of \eqref{p2} and \eqref{fff}. Existence and uniqueness of $v$ are granted by standard direct methods and strict convexity arguments, and the Euler-Lagrange equation
\eqn{com.3}
$$
\int_{8B_{h}}\langle \partial \tx{H}_{i}(Dv;8B_{h}),Dw\rangle\dx= \int_{8B_{h}}\tx{f}w\dx
$$
is verified by all $w\in W^{1,q}_{0}(8B_{h})$. Since $\tx{H}_{i}$ satisfies $\eqref{eig.4}_{1}$ and $v=\ti{u}\ge \ti{\psi}$ on $\partial(8B_{h})$, a straightforward manipulation of \cite[Lemma 2.1]{sch} yields that $v\ge \ti{\psi}$ on $8B_{h}$, so $v\in (\ti{u}+W^{1,q}_{0}(8B_{h}))\cap \tx{K}^{\ti{\psi}}(8B_{h})$. Observe also that 
\eqn{com.9}
$$
(\tx{E}_{i}(\snr{D\ti{u}};8B_{h})-\kk)_{+}\le (\ti{\tx{E}}(x,\snr{D\ti{u}})-\kk)_{+}\qquad \mbox{for all} \ \ x\in 8B_{h}, \  \kk\ge 0.
$$
Recalling the bound imposed on the size of $\snr{h}$, we have:
$$
8\tx{c}_{\tx{p}}\Lambda_{*}\nr{\tx{f}}_{L^{\infty}(8B_{h})}\snr{h}^{\beta_{0}}\stackrel{\eqref{fff}}{\le} 8\tx{c}_{\tx{p}}\Lambda_{*}\tx{f}_{\infty}\tx{h}_{\infty}^{-1}\le \frac{1}{2},
$$
so the reference estimates in Proposition \ref{p41} apply to $v$: energy estimate (recall the definition of $\ti{\tx{a}}_{i}$ in Section \ref{aux})
\begin{flalign}\label{enes.1}
\int_{8B_{h}}\tx{H}_{i}(Dv;8B_{h})\dx\le c\int_{8B_{h}}\tx{H}_{i}(D\ti{u};8B_{h})+1\dx\le  c\int_{8B_{h}}\ti{\tx{H}}(x,D\ti{u})+1\dx
\end{flalign}
is verified with $c\equiv c(n,\Lambda_{*},\tx{g})$. Moreover, for all $\delta_{0},\delta_{1},\delta_{2}\in (0,(2-q)/(80n))$ the Lipschitz estimates
\begin{eqnarray}\label{com.1}
\nr{Dv}_{L^{\infty}(6B_{h})}&\stackrel{\eqref{nh.2}}{\le}&c\left(1+\nr{\tx{f}}_{L^{\infty}(8B_{h})}^{\tx{m}_{0}}\right)\nr{\tx{E}_{i}(\snr{D\ti{u}};8B_{h})}_{L^{\infty}(8B_{h})}^{\frac{\delta_{0}}{2-\mu}}\nr{D\ti{u}}_{L^{\infty}(8B_{h})}\nonumber \\
&&+c\left(1+\nr{\tx{f}}_{L^{\infty}(8B_{h})}^{\tx{m}_{0}}\right)\nonumber \\
&\stackrel{\eqref{fff}}{\le}&c\nr{\tx{E}_{i}(\snr{D\ti{u}};8B_{h})}_{L^{\infty}(8B_{h})}^{\frac{\delta_{0}}{2-\mu}}\nr{D\ti{u}}_{L^{\infty}(8B_{h})}+c\nonumber \\
&\stackrel{\eqref{com.9}}{\le}&c\nr{\ti{\tx{E}}(\cdot,\snr{D\ti{u}})}_{L^{\infty}(8B_{h})}^{\frac{\delta_{0}}{2-\mu}}\nr{D\ti{u}}_{L^{\infty}(8B_{h})}+c\stackrel{\eqref{0.001}_{2}}{\le}c\tx{M}_{2}^{\frac{1+\delta_{0}}{2-\mu}},
\end{eqnarray}
with $c\equiv c(\data_{*},\tx{f}_{\infty},\delta_{0},\delta_{1},\delta_{2})$, and
\eqn{com.1.1}
$$
\nr{Dv}_{L^{\infty}(6B_{h})}\stackrel{\eqref{nh.2.1},\eqref{fff}}{\le} c\nr{D\ti{u}}_{L^{\infty}(8B_{h})}^{4q}+c\stackrel{\eqref{1qa},\eqref{mmm.1}}{\le}c\tx{M}_{1}^{10}\stackrel{\eqref{0.001}_{1}}{\le}c\tx{M}_{2}^{20},
$$
with $c\equiv c(\data_{*},\tx{f}_{\infty},\delta_{0},\delta_{1},\delta_{2})$ hold true. Finally, given any ball $B\Subset 8B_{h}$, the Caccioppoli type inequality
\begin{eqnarray}\label{com.2}
\nr{D(\tx{E}_{i}(\snr{Dv};8B_{h})-\kk)_{+}}_{L^{2}(3B/4)}&\stackrel{\eqref{nh.1}}{\le}&\frac{c\tx{M}^{\delta_{1}}}{\snr{B}^{1/n}}\nr{(\tx{E}_{i}(\snr{Dv};8B_{h})-\kk)_{+}}_{L^{2}(B)}\nonumber \\
&&+c\tx{M}^{1-\delta_{2}}\nr{\tx{f}}_{L^{\infty}(8B_{h})}\nr{\ell_{1}(Dv)^{\delta_{2}}}_{L^{2}(B)}
\end{eqnarray}
is satisfied for $c\equiv c(\data_{*},\delta_{0},\delta_{1},\delta_{2})$, and all numbers $\tx{M}\ge \max\left\{\nr{Dv}_{L^{\infty}(B)},10+\tx{T}_{\tx{g}}+\tx{T}_{\mu}\right\}$. 
Taking $\tx{M}= \max\left\{\nr{Dv}_{L^{\infty}(B)},10+\tx{T}_{\tx{g}}+\tx{T}_{\mu}\right\}$, the positions in \eqref{mmm.1} and \eqref{com.1} grant the bounds
\begin{eqnarray}\label{com.11}
    \ti{\tx{a}}_{i}(8B_{h})\nr{Dv}_{L^{\infty}(6B_{h})}^{q}&\le& c\ti{\tx{a}}_{i}(8B_{h})\nr{D\ti{u}}_{L^{\infty}(8B_{h})}^{q}\nr{\tx{E}_{i}(\snr{D\ti{u}};8B_{h})}_{L^{\infty}(8B_{h})}^{\frac{q\delta_{0}}{2-\mu}}+c\nonumber \\
    &\le&c\tx{M}_{2}^{\frac{q\delta_{0}}{2-\mu}}\left\|\ti{\tx{a}}(\cdot)\snr{D\ti{u}}^{q}\right\|_{L^{\infty}(8B_{h})}+c\le c\tx{M}_{2}^{1+\frac{q\delta_{0}}{2-\mu}},
\end{eqnarray}
for $c\equiv c(\data_{*},\tx{f}_{\infty},\delta_{0},\delta_{1},\delta_{2})$.
Thanks to \eqref{com.1}-\eqref{com.1.1}, estimate \eqref{com.2} becomes
\begin{eqnarray}\label{com.22}
\nr{D(\tx{E}_{i}(\snr{Dv};8B_{h})-\kk)_{+}}_{L^{2}(3B/4)}&\le&\frac{c\tx{M}_{2}^{20\delta_{1}}}{\snr{B}^{1/n}}\nr{(\tx{E}_{i}(\snr{Dv};8B_{h})-\kk)_{+}}_{L^{2}(B)}\nonumber \\
&&+c\tx{M}_{2}^{\frac{(1-\delta_{2})(1+\delta_{0})}{2-\mu}}\nr{\tx{f}}_{L^{\infty}(8B_{h})}\nr{\ell_{1}(Dv)^{\delta_{2}}}_{L^{2}(B)},
\end{eqnarray}
for all balls $B\subseteq 6B_{h}$. Let us quantify the distance between $D\ti{u}$ and $Dv$. After extending $v\equiv \ti{u}$ in $B_{1}\setminus 8B_{h}$, we see that function $w:=\ti{u}-v$ is admissible in \eqref{com.3}, while thanks the considerations made immediately below \eqref{com.3}, function $v$ is a valid test function in \eqref{vler} - more precisely
\eqn{com.6.1}
$$
-\int_{8B_{h}}\langle\partial\ti{\tx{H}}(x,D\ti{u}),D\ti{u}-Dv\rangle\dx\ge 0,
$$
so we get
\begin{eqnarray}\label{com.6}
\int_{8B_{h}}\mathcal{V}^{2}_{i}(D\ti{u},Dv;8B_{h})\dx &\stackrel{\eqref{eig.4}_{2}}{\le}&c\int_{8B_{h}}\langle\partial \tx{H}_{i}(D\ti{u};8B_{h})-\partial \tx{H}_{i}(Dv;8B_{h}),D\ti{u}-Dv\rangle\dx\nonumber \\
&\stackrel{\eqref{com.3}}{=}&c\int_{8B_{h}}\langle\partial \tx{H}_{i}(D\ti{u};8B_{h}),D\ti{u}-Dv\rangle-\tx{f}(\ti{u}-v)\dx\nonumber \\
&\stackrel{\eqref{com.6.1}}{\le}&c\int_{8B_{h}}\langle\partial \tx{H}_{i}(D\ti{u};8B_{h})-\partial\ti{\tx{H}}(x,D\ti{u}),D\ti{u}-Dv\rangle\dx\nonumber \\
&&+c\nr{\tx{f}}_{L^{\infty}(8B_{h})}\int_{8B_{h}}\snr{\ti{u}-v}\dx\nonumber \\
&\le&c\left([\ti{\tx{a}}]_{0,\alpha;B_{1}}\snr{h}^{\beta_{0}\alpha}+\nr{\tx{f}}_{L^{\infty}(B_{1})}\snr{h}^{\beta_{0}}\right)\int_{8B_{h}}\ell_{1}(D\ti{u})^{q-1}\snr{Dv-D\ti{u}}\dx\nonumber\\
&\stackrel{\eqref{mmm.1}}{\le}&c\tx{M}_{1}^{q-1}\left([\ti{\tx{a}}]_{0,\alpha;B_{1}}\snr{h}^{\beta_{0}\alpha}+\nr{\tx{f}}_{L^{\infty}(B_{1})}\snr{h}^{\beta_{0}}\right)\int_{8B_{h}}\snr{Dv}+\snr{D\ti{u}}\dx\nonumber \\
&\stackrel{\eqref{enes.1}}{\le}&c\tx{M}_{1}^{q-1}\left([\ti{\tx{a}}]_{0,\alpha;B_{1}}\snr{h}^{\beta_{0}\alpha}+\nr{\tx{f}}_{L^{\infty}(B_{1})}\snr{h}^{\beta_{0}}\right)\int_{8B_{h}}\ti{\tx{H}}(x,D\ti{u})+1\dx\nonumber \\
&\stackrel{\eqref{0.001}_{1}}{\le}&c\tx{M}_{2}^{\frac{q-1}{2-\mu}}\left([\ti{\tx{a}}]_{0,\alpha;B_{1}}\snr{h}^{\beta_{0}\alpha}+\nr{\tx{f}}_{L^{\infty}(B_{1})}\snr{h}^{\beta_{0}}\right)\int_{8B_{h}}\ti{\tx{H}}(x,D\ti{u})+1\dx,
\end{eqnarray}
for $c\equiv c(\data_{*})$.
\subsubsection*{Step 2: transfer of regularity} Before turning to the key estimate, let us add extra notation. We define $\tx{T}:=[\ti{\tx{a}}]_{0,\alpha;B_{1}}+\nr{\tx{f}}_{L^{\infty}(B_{1})}$, and introduce exponents
\eqn{t1t2}
$$
\tx{t}_{1}:=20\delta_{1},\qquad\qquad\quad   \tx{t}_{2}:=\frac{(\delta_{0}+1)(1-\delta_{2})}{2-\mu}.
$$
Notice that
\eqn{noti}
$$
0<\delta_{0},\delta_{1},\delta_{2}<\frac{2-q}{80n} \ \Longrightarrow \ \frac{q+1-\mu+q\delta_{0}}{2-\mu}+40\delta_{1}<\frac{2(1+\delta_{0})(1-\delta_{2})}{2-\mu}=2\tx{t}_{2}.
$$
This will be helpful in a few lines. A quick manipulation of \eqref{com.6} yields
\begin{flalign}\label{com.666}
 \int_{8B_{h}}\mathcal{V}_{i}^{2}(D\ti{u},Dv;8B_{h})\dx\le c\tx{T}\snr{h}^{\beta_{0}\alpha}\tx{M}_{2}^{\frac{q-1}{2-\mu}}\int_{8B_{h}}\ti{\tx{H}}(x,D\ti{u})+1\dx,
\end{flalign}
with $c\equiv c(\data_{*})$. By scaling, \eqref{assfr} and \eqref{p2} we also record
\begin{eqnarray}\label{ap}
\tx{T}&=&\rr^{\alpha}[a]_{0,\alpha;B_{\rr}(x_{0})}+\rr\left\|\diver\left(\partial \tx{L}_{*}(D\psi)+q\left(\inf_{x\in B_{\rr}(x_{0})}a(x)+a_{*}\right)\ell_{s}(D\psi)^{q-2}D\psi\right)\right\|_{L^{\infty}(B_{\rr}(x_{0}))}\nonumber \\
&\le &\rr^{\alpha}c(\data_{*}(B_{\rr}(x_{0}))).
\end{eqnarray}
Next, recalling the $1$-Lipschitz character of truncations, \eqref{0.000}, \eqref{Vm}$_{1}$, \eqref{com.1} and \eqref{com.11} on $6B_{h}$ we control
\begin{flalign}\label{com.7}
&\snr{(\tx{E}_{i}(\snr{D\ti{u}};8B_{h})-\kk)_{+}-(\tx{E}_{i}(\snr{Dv};8B_{h})-\kk)_{+}}\le \snr{\tx{E}_{i}(\snr{D\ti{u}};8B_{h})-\tx{E}_{i}(\snr{Dv};8B_{h})}\nonumber \\
&\qquad \qquad \quad \le c(1+\snr{D\ti{u}}^{2}+\snr{Dv}^{2})^{\frac{2-\mu}{4}}\snr{V_{1,2-\mu}(D\ti{u})-V_{1,2-\mu}(Dv)}\nonumber \\
&\qquad \qquad \quad \quad +c\ti{\tx{a}}_{i}(8B_{h})^{\frac{1}{2}}(s^{2}+\snr{D\ti{u}}^{2}+\snr{Dv}^{2})^{\frac{q}{4}}\ti{\tx{a}}_{i}(8B_{h})^{\frac{1}{2}}\snr{V_{s,q}(D\ti{u})-V_{s,q}(Dv)}\nonumber \\
&\qquad \qquad \quad \le c\left(\tx{M}_{2}^{\frac{1+\delta_{0}}{2}}+\tx{M}_{2}^{\frac{1}{2}+\frac{q\delta_{0}}{2(2-\mu)}}\right)\mathcal{V}_{i}(D\ti{u},Dv;8B_{h})\le c\tx{M}_{2}^{\frac{1}{2}+\frac{q\delta_{0}}{2(2-\mu)}}\mathcal{V}_{i}(D\ti{u},Dv;8B_{h}),
\end{flalign}
and in a similar fashion, by \eqref{0.00} we have
\begin{flalign}\label{com.8}
&\snr{(\ti{\tx{E}}(x,\snr{D\ti{u}})-\kk)_{+}-(\tx{E}_{i}(\snr{D\ti{u}};8B_{h})-\kk)_{+}}\le \snr{\ti{\tx{E}}(x,\snr{D\ti{u}})-\tx{E}_{i}(\snr{D\ti{u}};8B_{h})}\nonumber \\
&\qquad \qquad \quad \lesssim_{q}\snr{\ti{\tx{a}}_{i}(8B_{h})-\ti{\tx{a}}(x)}\ell_{s}(D\ti{u})^{q}\lesssim_{q}[\ti{\tx{a}}]_{0,\alpha;B_{1}}\snr{h}^{\beta_{0}\alpha}\ell_{s}(D\ti{u})^{q}.
\end{flalign}
Basic properties of translations, Jensen inequality, \eqref{com.9}, \eqref{enes.1}, \eqref{com.22}, \eqref{noti} and \eqref{com.666}-\eqref{com.8} then yield
\begin{flalign}\label{com.10}
&\int_{B_{h}}\snr{\tau_{h}(\ti{\tx{E}}(\cdot,\snr{D\ti{u}})-\kk)_{+}}^{2}\dx\nonumber \\
 &\qquad \quad \quad \le c\int_{2B_{h}}\snr{\ti{\tx{E}}(x,\snr{D\ti{u}})-\tx{E}_{i}(\snr{D\ti{u}};8B_{h})}^{2}\dx+c\int_{B_{h}}\snr{\tau_{h}(\tx{E}_{i}(\snr{D\ti{u}};8B_{h})-\kk)_{+}}^{2}\dx\nonumber \\
 &\qquad \quad \quad \le c[\ti{\tx{a}}]_{0,\alpha;B_{1}}^{2}\snr{h}^{2\alpha\beta_{0}}\int_{2B_{h}}\ell_{s}(D\ti{u})^{2q}\dx+c\int_{B_{h}}\snr{\tau_{h}(\tx{E}_{i}(\snr{Dv};8B_{h})-\kk)_{+}}^{2}\dx\nonumber \\
 &\qquad \quad \quad \quad +c\int_{2B_{h}}\snr{\tx{E}_{i}(\snr{D\ti{u}};8B_{h})-\tx{E}_{i}(\snr{Dv};8B_{h})}^{2}\dx\nonumber \\
 &\qquad \quad \quad \le c[\ti{\tx{a}}]_{0,\alpha;B_{1}}^{2}\snr{h}^{2\alpha\beta_{0}}\int_{2B_{h}}\ell_{s}(D\ti{u})^{2q}\dx+c\snr{h}^{2}\int_{2B_{h}}\snr{D(\tx{E}_{i}(\snr{Dv};8B_{h})-\kk)_{+}}^{2}\dx\nonumber \\
  &\qquad \quad \quad \quad +c\tx{M}_{2}^{1+\frac{q\delta_{0}}{2-\mu}}\int_{2B_{h}}\mathcal{V}_{i}^{2}(D\ti{u},Dv;8B_{h})\dx\nonumber \\
  &\qquad \quad \quad\le c[\ti{\tx{a}}]_{0,\alpha;B_{1}}^{2}\snr{h}^{2\alpha\beta_{0}}\int_{2B_{h}}\ell_{s}(D\ti{u})^{2q}\dx+c\snr{h}^{2(1-\beta_{0})}\tx{M}_{2}^{40\delta_{1}}\int_{(8/3)B_{h}}(\tx{E}_{i}(\snr{Dv};8B_{h})-\kk)_{+}^{2}\dx\nonumber \\
  &\qquad \quad \quad\quad+c\tx{M}_{2}^{\frac{2(1-\delta_{2})(1+\delta_{0})}{2-\mu}}\nr{\tx{f}}_{L^{\infty}(B_{1})}^{2}\snr{h}^{2}\int_{(8/3)B_{h}}\ell_{1}(Dv)^{2\delta_{2}}\dx +c\tx{M}_{2}^{1+\frac{q\delta_{0}}{2-\mu}}\int_{2B_{h}}\mathcal{V}_{i}^{2}(D\ti{u},Dv;8B_{h})\dx\nonumber \\
  &\qquad \quad \quad\le c\tx{T}^{2}\snr{h}^{2\alpha\beta_{0}}\int_{2B_{h}}\ell_{s}(D\ti{u})^{2q}\dx+c\snr{h}^{2(1-\beta_{0})}\tx{M}_{2}^{40\delta_{1}}\int_{(8/3)B_{h}}(\ti{\tx{E}}(x,\snr{D\ti{u}})-\kk)_{+}^{2}\dx\nonumber \\
  &\qquad \quad \quad\quad+c\tx{M}_{2}^{\frac{2(1-\delta_{2})(1+\delta_{0})}{2-\mu}}\tx{T}^{2}\snr{h}^{2}\int_{(8/3)B_{h}}\ell_{1}(Dv)\dx +c\tx{M}_{2}^{1+\frac{q\delta_{0}}{2-\mu}}\int_{2B_{h}}\mathcal{V}_{i}^{2}(D\ti{u},Dv;8B_{h})\dx\nonumber \\
  &\qquad \quad \quad\quad+c\snr{h}^{2(1-\beta_{0})}\tx{M}_{2}^{40\delta_{1}}\int_{(8/3)B_{h}}\snr{\tx{E}_{i}(\snr{D\ti{u}};8B_{h})-\tx{E}_{i}(\snr{Dv};8B_{h})}^{2}\dx\nonumber \\
  &\qquad \quad \quad \le c\snr{h}^{2(1-\beta_{0})}\tx{M}_{2}^{40\delta_{1}}\int_{8B_{h}}(\ti{\tx{E}}(x,\snr{D\ti{u}})-\kk)_{+}^{2}\dx+c\tx{M}_{2}^{\frac{q+1-\mu+q\delta_{0}}{2-\mu}+40\delta_{1}}\tx{T}\snr{h}^{\beta_{0}\alpha}\int_{8B_{h}}\ti{\tx{H}}(x,D\ti{u})+1\dx\nonumber \\
   &\qquad \quad \quad\quad+c\tx{T}^{2}\snr{h}^{2\alpha\beta_{0}}\left(\int_{8B_{h}}\ell_{1}(D\ti{u})^{2q}\dx+\tx{M}_{2}^{\frac{2(1-\delta_{2})(1+\delta_{0})}{2-\mu}}\int_{8B_{h}}\ti{\tx{H}}(x,D\ti{u})+1\dx\right)\nonumber \\
&\qquad \quad \quad \le c\snr{h}^{2(1-\beta_{0})}\tx{M}_{2}^{2\tx{t}_{1}}\int_{8B_{h}}(\ti{\tx{E}}(x,\snr{D\ti{u}})-\kk)_{+}^{2}\dx+c\tx{T}\snr{h}^{\alpha\beta_{0}}\tx{M}_{2}^{2\tx{t}_{2}}\int_{8B_{h}}\ell_{1}(D\ti{u})^{4}\dx,
\end{flalign}
with $c\equiv c(\data_{*},\tx{f}_{\infty},\delta_{0},\delta_{1},\delta_{2})$. In \eqref{com.10} we equalize
$$
2(1-\beta_{0})=\beta_{0}\alpha \ \Longrightarrow \ \beta_{0}=\frac{2}{2+\alpha}
$$
so \eqref{com.10} becomes
\begin{eqnarray}\label{com.12}
    \int_{B_{h}}\snr{\tau_{h}(\ti{\tx{E}}(\cdot,\snr{D\ti{u}})-\kk)_{+}}^{2}\dx&\le&c\snr{h}^{\frac{2\alpha}{2+\alpha}}\tx{M}_{2}^{2\tx{t}_{1}}\int_{8B_{h}}(\ti{\tx{E}}(x,\snr{D\ti{u}})-\kk)_{+}^{2}\dx\nonumber \\
    &&+c\snr{h}^{\frac{2\alpha}{2+\alpha}}\tx{T}\tx{M}_{2}^{2\tx{t}_{2}}\int_{8B_{h}}\ell_{1}(D\ti{u})^{4}\dx,
\end{eqnarray}
for $c\equiv c(\data_{*},\tx{f}_{\infty},\delta_{0},\delta_{1},\delta_{2})$. 
\subsubsection*{Step 3: covering} To conclude, we only need to patch up \eqref{com.12} via a dyadic covering argument. More precisely, we take a lattice $\mathscr{L}_{\snr{h}^{\beta_{0}}/\sqrt{n}}$ of open, disjoint cubes $\{Q_{\snr{h}^{\beta_{0}}/\sqrt{n}}(y)\}_{y\in (2\snr{h}^{\beta_{0}}/\sqrt{n})\mathbb{Z}^{n}}$. From this lattice, we pick $\tx{n}\approx_{n}\snr{h}^{-n\beta_{0}}$ cubes centered at points $\{x_{\ccc}\}_{\ccc\le \tx{n}}\subset (2\snr{h}^{\beta_{0}}/\sqrt{n})\mathbb{Z}^{n}$ such that $\snr{x_{\ccc}}\le 1/2+2\snr{h}^{\beta_{0}}$, thus determining the corresponding family $\{Q_{\ccc}\}_{\ccc\le \tx{n}}\equiv \{Q_{\snr{h}^{\beta_{0}}/\sqrt{n}}(x_{\ccc})\}_{\ccc\le \tx{n}}$. Observe that, in general, if $\snr{x}>1/2+2\snr{h}^{\beta_{0}}$, then $Q_{\snr{h}^{\beta_{0}}/\sqrt{n}}(x)\cap B_{1/2}=\emptyset$ as $Q_{\snr{h}^{\beta_{0}}/\sqrt{n}}(x)\subset B_{\snr{h}^{\beta_{0}}}(x)$ and $B_{\snr{h}^{\beta_{0}}}(x)\cap B_{1/2}=\emptyset$. We indeed have
\begin{flalign}\label{com.13}
\left| \  B_{1/2}\setminus \bigcup_{\ccc\le \tx{n}}Q_{\ccc} \ \right|=0,\qquad\qquad  Q_{\ccc_{1}}\cap Q_{\ccc_{2}}=\emptyset \ \Longleftrightarrow \ \ccc_{1}\not =\ccc_{2}.
\end{flalign}
Such a family of cubes corresponds to a family of balls $\{B_{\ccc}\}_{\ccc\le \tx{n}}:=\{B_{\snr{h}^{\beta_{0}}}(x_{\ccc})\}_{\ccc\le \tx{n}}$ in the sense that $Q_{\ccc}$ is the largest hypercube concentric to $B_{\ccc}$, with sides parallel to the coordinate axes. By construction it is $8B_{\ccc}\Subset B_{1}$ for all $\ccc\le \tx{n}$. Moreover, each of the dilated balls $8B_{\ccc_{t}}$ intersects the similar ones $8B_{\ccc_{s}}$, $\ccc_{t}\not =\ccc_{s}$ a finite, quantified number of times, depending only on $n$ (uniform finite intersection property). In fact, notice that the family of outer cubes $\{Q_{\snr{h}^{\beta_{0}}}(x_{\ccc})\}_{\ccc\le \tx{n}}$ has the same property and $B_{\ccc}\subset Q_{\snr{h}^{\beta_{0}}}(x_{\ccc})$. This yields:
\eqn{com.13.1}
$$
\sum_{\ccc\le \tx{n}}\phi(8B_{\ccc})\lesssim_{n}\phi(B_{1}),
$$
for all Borel measure $\phi$ defined on $B_{1}$. By \eqref{com.13} it turns out that also $\{B_{\ccc}\}_{\ccc\le \tx{n}}$ is a measure covering of $B_{1/2}$, i.e.:
\eqn{com.13.2}
$$
\left| \  B_{1/2}\setminus \bigcup_{\ccc\le \tx{n}}B_{\ccc} \ \right|=0.
$$
We can then estimate
\begin{eqnarray}\label{comm.14}
\int_{B_{1/2}}\snr{\tau_{h}(\ti{\tx{E}}(\cdot,\snr{D\ti{u}})-\kk)_{+}}^{2}\dx&\stackrel{\eqref{com.13.2}}{\le}&\sum_{\ccc\le \tx{n}}\int_{B_{\ccc}}\snr{\tau_{h}(\ti{\tx{E}}(\cdot,\snr{D\ti{u}})-\kk)_{+}}^{2}\dx\nonumber \\
&\stackrel{\eqref{com.12}}{\le}&c\snr{h}^{\frac{2\alpha}{2+\alpha}}\tx{M}_{2}^{2\tx{t}_{1}}\sum_{\ccc\le\tx{n}}\int_{8B_{\ccc}}(\ti{\tx{E}}(\cdot,\snr{D\ti{u}})-\kk)_{+}^{2}\dx\nonumber \\
&&+c\snr{h}^{\frac{2\alpha}{2+\alpha}}\tx{T}\tx{M}_{2}^{2\tx{t}_{2}}\sum_{\ccc\le \tx{n}}\int_{8B_{\ccc}}\ell_{1}(D\ti{u})^{4}\dx\nonumber\\
&\stackrel{\eqref{com.13.1}}{\le}&c\snr{h}^{\frac{2\alpha}{2+\alpha}}\tx{M}_{2}^{2\tx{t}_{1}}\int_{B_{1}}(\ti{\tx{E}}(\cdot,\snr{D\ti{u}})-\kk)_{+}^{2}\dx\nonumber \\
&&+c\snr{h}^{\frac{2\alpha}{2+\alpha}}\tx{T}\tx{M}_{2}^{2\tx{t}_{2}}\int_{B_{1}}\ell_{1}(D\ti{u})^{4}\dx,
\end{eqnarray}
with $c\equiv c(\data_{*},\tx{f}_{\infty},\delta_{0},\delta_{1},\delta_{2})$. Estimate \eqref{comm.14}, Lemma \ref{l4} and immersion \eqref{immersione} then give
\begin{flalign}\label{com.15}
&[(\ti{\tx{E}}(\cdot,\snr{D\ti{u}})-\kk)_{+}]_{\gamma,2;B_{1/2}}+\nr{(\ti{\tx{E}}(\cdot,\snr{D\ti{u}})-\kk)_{+}}_{L^{2\chi}(B_{1/2})}\nonumber \\
&\qquad \qquad \quad \le c\tx{M}_{2}^{\tx{t}_{1}}\nr{(\ti{\tx{E}}(\cdot,\snr{D\ti{u}})-\kk)_{+}}_{L^{2}(B_{1})}+c\tx{T}^{\frac{1}{2}}\tx{M}_{2}^{\tx{t}_{2}}\nr{\ell_{1}(D\ti{u})}_{L^{4}(B_{1})}^{2},
\end{flalign}
for all $\gamma\in (0,\alpha/(2+\alpha))$, say $\gamma=\alpha/(4+2\alpha)$, and some $c\equiv c(\data_{*},\tx{f}_{\infty},\delta_{0},\delta_{1},\delta_{2})$. In \eqref{com.15} it is $\chi:=n/(n-2\gamma)>1$. Scaling back on $B_{\rr}(x_{0})$, setting $\bar{\tx{M}}\equiv\tx{M}_{2}$, recalling the definition of the various quantities appearing in \eqref{com.15} and of $\tx{E}$ in \eqref{hlhl}, and the scaling features of $\tx{T}$ in \eqref{ap}, we obtain \eqref{frca} after standard manipulations.
\end{proof}
\subsection{Lipschitz regularity} Let $B_{r}\subset B_{2r}\subset B$ be any ball with radius $r\in (0,1]$, consider concentric balls $B_{r/2}\subseteq B_{\tau_{2}}\subset B_{\tau_{1}}\subseteq B_{3r/4}$ and notice that there is no loss of generality in assuming that $\nr{\tx{E}_{*}(\cdot,\snr{Du})}_{L^{\infty}(B_{r/2})}\ge (10+\tx{T}_{\tx{g}}+\tx{T}_{\mu})$ otherwise there would be nothing to prove. By \eqref{1qa} and \eqref{areg}, all $x_{0}\in B_{\tau_{2}}$ are Lebesgue point for $\tx{E}_{*}(\cdot,\snr{Du})$, see \cite[Section 6.6]{ciccio}, so we set $r_{0}:=(\tau_{1}-\tau_{2})/8$ in such a way that $B_{2r_{0}}(x_{0})\Subset B_{\tau_{1}}$, and, via \eqref{frca} applied on $B_{r_{0}}(x_{0})$ with $\bar{\tx{M}}:=\nr{\tx{E}_{*}(\cdot,\snr{Du})}_{L^{\infty}(B_{\tau_{1}})}$ (that verifies \eqref{mmm.2}) we can apply Lemma \ref{revlem} choosing $v(x):=\tx{E}_{*}(x,\snr{Du(x)})$, $M_{0}= \bar{\tx{M}}^{\tx{t}_{1}}$, $M_{1}\equiv \bar{\tx{M}}^{\tx{t}_{2}}$, $f(x):=\ell_{1}(Du)^{4}$, $\sigma:=\alpha/2$, $\vartheta=1/2$, $\kk_{0}=0$, $\chi:=n/(n-2\gamma)\in (1,2)$, to get
\begin{eqnarray*}
\tx{E}_{*}(x_{0},\snr{Du(x_{0})})&\le&c\nr{\tx{E}_{*}(\cdot,\snr{Du})}_{L^{\infty}(B_{\tau_{1}})}^{\frac{\tx{t}_{1}\chi}{\chi-1}}\nra{\tx{E}_{*}(\cdot,\snr{Du})}_{L^{2}(B_{r_{0}}(x_{0}))}\nonumber \\
&&+c\nr{\tx{E}_{*}(\cdot,\snr{Du})}_{L^{\infty}(B_{\tau_{1}})}^{\frac{\tx{t}_{1}}{\chi-1}+\tx{t}_{2}}\mathbf{P}^{\frac{1}{2}}_{\frac{\alpha}{2}}\left(\ell_{1}(Du)^{4};x_{0},2r_{0}\right),
\end{eqnarray*}
for $c\equiv c(\data_{*}(B),\delta_{0},\delta_{1},\delta_{2})$. The arbitrariety of $x_{0}\in B_{\tau_{2}}$ allows passing to the supremum on both sides of the above inequality, and derive 
\begin{eqnarray}\label{l.0}
\nr{\tx{E}_{*}(\cdot,\snr{Du})}_{L^{\infty}(B_{\tau_{2}})}&\le&\frac{c}{(\tau_{1}-\tau_{2})^{n/2}}\nr{\tx{E}_{*}(\cdot,\snr{Du})}_{L^{\infty}(B_{\tau_{1}})}^{\frac{\tx{t}_{1}\chi}{\chi-1}+\frac{1}{2}}\nr{\tx{E}_{*}(\cdot,\snr{Du})}_{L^{1}(B_{3r/4})}^{\frac{1}{2}}\nonumber \\
&&+c\nr{\tx{E}_{*}(\cdot,\snr{Du})}_{L^{\infty}(B_{\tau_{1}})}^{\frac{\tx{t}_{1}}{\chi-1}+\tx{t}_{2}}\left\|\mathbf{P}^{\frac{1}{2}}_{\frac{\alpha}{2}}\left(\ell_{1}(Du)^{4};\cdot,(\tau_{1}-\tau_{2})/4\right)\right\|_{L^{\infty}(B_{\tau_{2}})}\nonumber \\
&\stackrel{\eqref{stimazza}}{\le}&\frac{c}{(\tau_{1}-\tau_{2})^{n/2}}\nr{\tx{E}_{*}(\cdot,\snr{Du})}_{L^{\infty}(B_{\tau_{1}})}^{\frac{\tx{t}_{1}\chi}{\chi-1}+\frac{1}{2}}\nr{\tx{E}_{*}(\cdot,\snr{Du})}_{L^{1}(B_{3r/4})}^{\frac{1}{2}}\nonumber \\
&&+c\nr{\tx{E}_{*}(\cdot,\snr{Du})}_{L^{\infty}(B_{\tau_{1}})}^{\frac{\tx{t}_{1}}{\chi-1}+\tx{t}_{2}}\nr{\ell_{1}(Du)}_{L^{\frac{4(n+1)}{\alpha}}(B_{(3\tau_{2}+\tau_{1})/4})}^{2},
\end{eqnarray}
with $c\equiv c(\data_{*}(B),\delta_{0},\delta_{1},\delta_{2})$. At this stage we only need to tame the last term in \eqref{l.0}. To this end, thanks to \eqref{bbbb}, we apply \eqref{15} on $B_{\tau_{1}}$ with $\tau=(3\tau_{2}+\tau_{1})/(4\tau_{1})$, $p=4(n+1)/\alpha$ and $\tx{M}_{*}=\max\{\nr{Du}_{L^{\infty}(B_{\tau_{1}})},10+\tx{T}_{\tx{g}}+\tx{T}_{\mu}\}$ to get
\begin{eqnarray}\label{l.0.0}
\nr{\tx{E}_{*}(\cdot,\snr{Du})}_{L^{\infty}(B_{\tau_{2}})}&\le&\frac{c}{(\tau_{1}-\tau_{2})^{n/2}}\nr{\tx{E}_{*}(\cdot,\snr{Du})}_{L^{\infty}(B_{\tau_{1}})}^{\frac{\tx{t}_{1}\chi}{\chi-1}+\frac{1}{2}}\nr{\tx{E}_{*}(\cdot,\snr{Du})}_{L^{1}(B_{3r/4})}^{\frac{1}{2}}\nonumber \\
&&+\frac{c}{(\tau_{1}-\tau_{2})^{\tx{d}}}\nr{\tx{E}_{*}(\cdot,\snr{Du})}_{L^{\infty}(B_{\tau_{1}})}^{\frac{\tx{t}_{1}}{\chi-1}+\tx{t}_{2}}\nr{Du}_{L^{\infty}(B_{\tau_{1}})}^{4\gamma_{\mu;\omega}}\left(\nr{Du}_{L^{1}(B_{\tau_{1}})}^{\frac{\alpha}{2(n+1)}}+1\right)\nonumber \\
&\stackrel{\eqref{0.001}_{1}}{\le}&\frac{c}{(\tau_{1}-\tau_{2})^{n/2}}\nr{\tx{E}_{*}(\cdot,\snr{Du})}_{L^{\infty}(B_{\tau_{1}})}^{\frac{\tx{t}_{1}\chi}{\chi-1}+\frac{1}{2}}\nr{\tx{E}_{*}(\cdot,\snr{Du})}_{L^{1}(B_{3r/4})}^{\frac{1}{2}}\nonumber \\
&&+\frac{c}{(\tau_{1}-\tau_{2})^{\tx{d}}}\nr{\tx{E}_{*}(\cdot,\snr{Du})}_{L^{\infty}(B_{\tau_{1}})}^{\frac{\tx{t}_{1}}{\chi-1}+\tx{t}_{2}+8\gamma_{\mu;\omega}}\left(\nr{Du}_{L^{1}(B_{3r/4})}^{\frac{\alpha}{2(n+1)}}+1\right),
\end{eqnarray}
for any $\omega\in (0,\omega_{*})$ - see \eqref{thre}-\eqref{kkk} - and some $c\equiv c(\data_{*}(B),\nr{u_{0}}_{L^{\infty}(B)},\delta_{0},\delta_{1},\delta_{2},\omega)$. In \eqref{l.0.0} there are four degrees of freedom, $\delta_{0},\delta_{1},\delta_{2},\omega$ that still need to be fixed as to reabsorbe the $L^{\infty}$-terms on the right-hand side of \eqref{l.0.0}, and ultimately fix threshold $\bar{\mu}>1$ as a sole function of $(n,q,\alpha)$. Recalling the explicit expression of exponents $\tx{t}_{1}$, $\tx{t}_{2}$ in \eqref{t1t2}, we set
\begin{flalign}\label{posi}
\begin{array}{c}
\displaystyle
\tx{m}_{+}:=4\max\left\{\frac{5\chi}{\chi-1},\frac{4}{1+\alpha-q}\right\},\qquad \quad \quad\tx{m}_{*}:=\frac{\chi-1}{20\chi},\\[10pt]\displaystyle
\omega:=\delta_{1}\stackrel{\eqref{ddd}}{>}\mu-1,\qquad \delta_{2}=2\delta_{0},\qquad \delta_{1}:=\frac{\delta_{0}^{2}}{4\tx{m}_{+}},\qquad \delta_{0}:=\tx{m}_{*}\min\left\{\frac{2-q}{2^{10}n},\frac{1}{4}\right\},
\end{array}
\end{flalign}
thus fixing dependencies $\delta_{0},\delta_{1},\delta_{2},\omega\equiv \delta_{0},\delta_{1},\delta_{2},\omega(n,q,\alpha)$, and, via \eqref{muma}, \eqref{ddd}, also $\bar{\mu},\mu_{\textnormal{max}}\equiv \bar{\mu},\mu_{\textnormal{max}}(n,q,\alpha)$. The positions in \eqref{posi} and \eqref{t1t2}, together with the restrictions in \eqref{muma} and \eqref{ddd}, immediately imply that
\begin{flalign*}
\left\{
\begin{array}{c}
\displaystyle
\ 2-\mu\stackrel{\eqref{ddd}}{>}1-\delta_{0}^{2} \ \stackrel{\eqref{t1t2}}{\Longrightarrow} \ 1-\tx{t}_{2}>\frac{2\delta_{0}^{2}}{1-\delta_{0}^{2}}\stackrel{\eqref{posi}}{>}\frac{\chi\tx{t}_{1}}{\chi-1}+8\gamma_{\mu;\omega} \ \Longrightarrow \ \frac{\tx{t}_{1}}{\chi-1}+\tx{t}_{2}+8\gamma_{\mu;\omega}<1\\[10pt]\displaystyle
\ \frac{\tx{t}_{1}\chi}{\chi-1}+\frac{1}{2}\stackrel{\eqref{posi}}{<}1.
\end{array}
\right.
\end{flalign*}
We can then apply Young inequality in \eqref{l.0.0} with conjugate exponents 
$$
\begin{array}{c}
\displaystyle
(\tx{d}_{1},\tx{d}_{2}):=\left(\frac{2(\chi-1)}{\chi(2\tx{t}_{1}+1)-1},\frac{2(\chi-1)}{\chi(1-2\tx{t}_{1})-1}\right),\\[10pt]\displaystyle  (\tx{d}_{3},\tx{d}_{4}):=\left(\frac{\chi-1}{\tx{t}_{1}+(\chi-1)(\tx{t}_{2}+8\gamma_{\mu;\omega})},\frac{\chi-1}{(\chi-1)(1-\tx{t}_{2}-8\gamma_{\mu;\omega})-\tx{t}_{1}}\right)
\end{array}
$$
to get
\begin{eqnarray*}
\nr{\tx{E}_{*}(\cdot,\snr{Du})}_{L^{\infty}(B_{\tau_{2}})}&\le&\frac{1}{4}\nr{\tx{E}_{*}(\cdot,\snr{Du})}_{L^{\infty}(B_{\tau_{1}})}+\frac{c}{(\tau_{1}-\tau_{2})^{n\tx{d}_{2}/2}}\nr{\tx{E}_{*}(\cdot,\snr{Du})}_{L^{1}(B_{3r/4})}^{\frac{\tx{d}_{2}}{2}}\nonumber \\
&&+\frac{c}{(\tau_{1}-\tau_{2})^{\tx{d}\tx{d}_{4}}}\left(\nr{Du}_{L^{1}(B_{3r/4})}^{\frac{\alpha\tx{d}_{4}}{2(n+1)}}+1\right)
\end{eqnarray*}
for $c\equiv c(\data_{*}(B),\nr{u_{0}}_{L^{\infty}(B)})$. Finally, Lemma \ref{iterlem}, $\eqref{0.002}_{1}$, and $\eqref{assfr}_{2}$ allow concluding after standard manipulations,
\begin{flalign*}
\nr{\tx{E}_{*}(\cdot,\snr{Du})}_{L^{\infty}(B_{r/2})}\le c\nra{\tx{E}_{*}(\cdot,\snr{Du})}_{L^{1}(B_{3r/4})}^{\tx{b}}+c\nra{\ell_{1}(Du)}_{L^{1}(B_{3r/4})}^{\tx{b}}+cr^{-\tx{b}}\le c\nra{\tx{H}_{*}(\cdot,Du)}_{L^{1}(B_{r})}^{\tx{b}}+cr^{-\tx{b}},
\end{flalign*}
with $c\equiv c(\data_{*}(B),\nr{u_{0}}_{L^{\infty}(B)})$ and $\tx{b}\equiv \tx{b}(n,\mu,q,\alpha)$. A standard covering argument together with the minimality of $u$ in class $(u_{0}+W^{1,q}_{0}(B))\cap \tx{K}^{\psi}(B)$ grants for any open set $\ti{B}\Subset B$,
$$
\nr{\tx{E}_{*}(\cdot,\snr{Du})}_{L^{\infty}(\ti{B})}\le  c\nr{\tx{H}_{*}(\cdot,Du)+1}_{L^{1}(B)}^{\tx{b}}\le c\nr{\tx{H}_{*}(\cdot,Du_{0})}_{L^{1}(B)}^{\tx{b}}+c,
$$
for $c\equiv c(\data_{*}(B),\tx{d}(\ti{B},B),\nr{u_{0}}_{L^{\infty}(B)})$. This proves the following statement.
\begin{proposition}\label{lipb}
    Under assumptions \eqref{1qa}, \eqref{ass.2}-\eqref{p2}, \eqref{assfr}, \eqref{sasa} and \eqref{u0}, let $B\subset 2B\Subset \Omega$ be concentric balls, and $u\in (u_{0}+W^{1,q}_{0}(B))\cap \tx{K}^{\psi}(B)$ be the solution of Dirichlet problem \eqref{pdreg}. There exists a threshold parameter $\mu_{\textnormal{max}}\equiv \mu_{\textnormal{max}}(n,q,\alpha)>1$ such that if $1\le \mu<\mu_{\textnormal{max}}$, then $u\in W^{1,\infty}_{\loc}(B)$, and whenever $\ti{B}\Subset B$ is an open set, the intrinsic Lipschitz estimate
    \eqn{lipb.1}
    $$
    \nr{\tx{E}_{*}(\cdot,\snr{Du})}_{L^{\infty}(\ti{B})}\le c\nr{\tx{H}_{*}(\cdot,Du_{0})}_{L^{1}(B)}^{\tx{b}}+c,
    $$
    holds true with $c\equiv c(\data_{*}(B),\tx{d}(\ti{B},B),\nr{u_{0}}_{L^{\infty}(B)})$ and $\tx{b}\equiv \tx{b}(n,\mu,q,\alpha)$.
\end{proposition}
\subsection{Gradient H\"older continuity}\label{gghc}
Once derived uniform Lipschitz bounds for solutions to \eqref{pdreg}, the approach to gradient H\"older continuity gets closer to the classical, uniformly elliptic case, cf. \cite[Section 5.9]{dm}, and \cite[Section 3]{FM}. We fix an open set $\ti{B}\Subset B$ and observe that given any ball $B_{r}\subseteq \ti{B}$, and constant $\tx{H}_{0}\ge 1$ such that
\eqn{hhh}
$$1+\nr{u_{0}}_{L^{\infty}(B)}+\nr{\tx{H}_{*}(\cdot,Du_{0})}_{L^{1}(B)}\le \tx{H}_{0},$$
via \eqref{0.001}, estimate \eqref{lipb.1} can be turned into
\begin{eqnarray*}
\nr{Du}_{L^{\infty}(B_{r})}&\le& c\nr{\tx{E}_{*}(\cdot,\snr{Du})}_{L^{\infty}(B_{r})}^{\frac{1}{2-\mu}}+c\le c\nr{\tx{E}_{*}(\cdot,\snr{Du})}_{L^{\infty}(\ti{B})}^{\frac{1}{2-\mu}}+c\nonumber \\
&\le& c\nr{\tx{H}_{*}(\cdot,Du_{0})}_{L^{1}(B)}^{\frac{\tx{b}}{2-\mu}}+c\le c\tx{H}_{0}^{\frac{\tx{b}}{2-\mu}}=:\tx{c}_{B},
\end{eqnarray*}
for $c\equiv c(\data_{*}(B),\tx{d}(\ti{B},B),\nr{u_{0}}_{L^{\infty}(B)})$. It is evident that $\tx{c}_{B}\equiv \tx{c}_{B}(\data_{*}(B),\tx{d}(\ti{B},B),\tx{H}_{0})$ is increasing with respect to $\tx{H}_{0}$, that in turn grows proportionally to $\nr{u_{0}}_{L^{\infty}(B)}$ and $\nr{\tx{H}_{*}(\cdot,Du_{0})}_{L^{1}(B)}$, cf. \eqref{15}, \eqref{lipb.1} and \eqref{hhh}. In the next lines constant $\tx{c}_{B}$ will appear possibly magnified by a multiplicative factor depending on $\data_{*}(B)$, or raised to a positive power depending at most on $(n,\mu,q,\alpha)$, but maintaining the same monotonicity with respect to $\tx{H}_{0}$ - specifically, we shall keep on denoting it $\tx{c}_{B}$. Briefly,
\eqn{gh.1}
$$
\left\{
\begin{array}{c}
\displaystyle
\ \nr{Du}_{L^{\infty}(B_{r})}\le \tx{c}_{B}\qquad \mbox{for all balls} \ \ B_{r}\subseteq \ti{B}\\[8pt]\displaystyle
\ \tx{H}_{0}\mapsto \tx{c}_{B}(\data_{*}(B),\tx{d}(\ti{B},B),\tx{H}_{0}) \ \ \mbox{is increasing}.
\end{array}
\right.
$$
We then fix a threshold
\eqn{**}
$$
\tx{r}_{*}:=\left(40\tx{c}_{\tx{p}}\Lambda_{*}\left(\nr{\diver (\partial \tx{L}_{*}(D\psi))}_{L^{\infty}(B)}+4nq\left(\nr{a}_{L^{\infty}(B)}+1\right)\nr{\snr{D\psi}^{q-2}D^{2}\psi}_{L^{\infty}(B)}\right)+1\right)^{-1}
$$
so that $\tx{r}_{*}\equiv \tx{r}_{*}(\data_{*}(B))$, take balls $B_{\tau}\subset B_{2\tau}\subseteq \ti{B}$, $\tau\in (0,\tx{r}_{*}/2)$, set $\tx{f}:=-\diver(\partial \tx{H}_{i}(D\psi;B_{2\tau}))\in L^{\infty}(B)$, and let $v_{1}\in u+W^{1,q}_{0}(B_{2\tau})$ be the solution to Dirichlet problem
\eqn{vp.1}
$$
u+W^{1,q}_{0}(B_{2\tau})\ni w\mapsto \min \int_{B_{2\tau}}\tx{H}_{i}(Dw;B_{2\tau})-\tx{f}w\dx,
$$
and $v_{2}\in v_{1}+W^{1,q}_{0}(B_{\tau})$ be the solution of problem
\eqn{vp.2}
$$
v_{1}+W^{1,q}_{0}(B_{\tau})\ni w\mapsto \min \int_{B_{\tau}}\tx{H}_{i}(Dw;B_{2\tau})\dx.
$$
Thanks to the restriction in \eqref{**}, Proposition \ref{p41} applies twice, the first time to $v_{1}$ with $\tx{f}$ defined above, and the second to $v_{2}$ with $\tx{f}\equiv 0$, to get 
\eqn{gh.2}
$$
\nr{Dv_{2}}_{L^{\infty}(B_{3\tau/4})}+\nr{Dv_{1}}_{L^{\infty}(B_{3\tau/2})}+\nr{Du}_{L^{\infty}(\ti{B})}\le \tx{c}_{B}.
$$
Via \eqref{com.666} with $B_{2\tau}$ instead of $8B_{h}$ and no scaling, \eqref{gh.2}, \eqref{Vm} and $\eqref{eig.4}$ we obtain
\begin{eqnarray}\label{gh.h}
\nra{Du-Dv_{1}}_{L^{2}(B_{3\tau/2})}&\le& \tx{c}_{B}\nra{V_{1,2-\mu}(Du)-V_{1,2-\mu}(Dv_{1})}_{L^{2}(B_{2\tau})}\nonumber \\
&\le& \tx{c}_{B}\tau^{\frac{\alpha}{2}}\nra{1+\tx{H}_{*}(\cdot,Du)}_{L^{1}(B_{2\tau})}^{\frac{1}{2}}\le\tx{c}_{B}\tau^{\frac{\alpha}{2}}\left(1+\nr{\tx{H}_{*}(\cdot,Du)}_{L^{\infty}(B_{2\tau})}^{\frac{1}{2}}\right)\le \tx{c}_{B}\tau^{\frac{\alpha}{2}},
\end{eqnarray}
and, recalling the expression of the Euler-Lagrange equations related to problems \eqref{vp.1}-\eqref{vp.2}, we further obtain
\begin{eqnarray}\label{gh.h1}
\nra{Dv_{1}-Dv_{2}}_{L^{2}(B_{3\tau/4})}&\le&\tx{c}_{B}\nra{V_{1,2-\mu}(Dv_{1})-V_{1,2-\mu}(Dv_{2})}_{L^{2}(B_{\tau})}\nonumber \\
&\stackrel{\eqref{eig.4}_{2}}{\le}&\tx{c}_{B}\nra{\langle\partial\tx{H}_{i}(Dv_{1};B_{2\tau})-\partial \tx{H}_{i}(Dv_{2};B_{2\tau}),Dv_{1}-Dv_{2}\rangle}_{L^{1}(B_{\tau})}^{\frac{1}{2}}\nonumber \\
&\le&\tx{c}_{B}\nr{\tx{f}}_{L^{\infty}(B)}^{\frac{1}{2}}\nra{v_{1}-v_{2}}_{L^{1}(B_{\tau})}^{\frac{1}{2}}\nonumber \\
&\le& \tx{c}_{B}\tau^{\frac{1}{2}}\nra{Dv_{1}-Dv_{2}}_{L^{1}(B_{\tau})}^{\frac{1}{2}}\nonumber \\
&\le& \tx{c}_{B}\tau^{\frac{1}{2}}\nra{1+\tx{H}_{i}(Dv_{1};B_{2\tau})}_{L^{1}(B_{\tau})}^{\frac{1}{2}}\le \tx{c}_{B}\tau^{\frac{1}{2}},
\end{eqnarray}
so, merging the two previous displays we obtain
\eqn{gh.4}
$$
\nra{Du-Dv_{2}}_{L^{2}(B_{3\tau/4})}\le \tx{c}_{B}\tau^{\frac{\alpha}{2}}.
$$
 Moreover, by \eqref{aar}-\eqref{2v} applied on $v_{2}$ in \eqref{vp.2}, we can apply \cite[Proposition 5.6]{dm}\footnote{Our assumptions \eqref{ass.1}-\eqref{ass.2} correspond to \cite[(1.11)-(1.12)]{dm}, but we do not need to impose \cite[(1.13)]{dm} thanks to the first point in Remark \ref{r21}.} to deduce the existence of some 
 \eqn{gh.6}
$$\gamma_{*}\equiv \gamma_{*}(\data_{*}(B),\tx{d}(\ti{B},B),\tx{H}_{0})\in (0,1) \ \ \mbox{nonincreasing in}\ \ \tx{H}_{0},$$ such that
\eqn{gh.5}
$$
\osc_{B_{3\theta\tau/4}}Dv_{2}\le \tx{c}_{B}\theta^{\gamma_{*}}\qquad \mbox{for all} \ \ \theta\in (0,1).
$$
Combining \eqref{gh.4}-\eqref{gh.5} on any ball $B_{\sigma}\subset B_{3\tau/4}$ we obtain
\begin{eqnarray*}
\nra{Du-(Du)_{B_{\sigma}}}_{L^{2}(B_{\sigma})}&\le&2\nra{Du-Dv_{2}}_{L^{2}(B_{\sigma})}+2\nra{Dv_{2}-(Dv_{2})_{B_{\sigma}}}_{L^{2}(B_{\sigma})}\nonumber \\
&\le&\tx{c}_{B}\left(\tau/\sigma\right)^{\frac{n}{2}}\tau^{\frac{\alpha}{2}}+\tx{c}_{B}\left(\sigma/\tau\right)^{\gamma_{*}},
\end{eqnarray*}
so we can equalize above with the choice $\sigma=(3\tau/4)^{\frac{n+2\gamma_{*}+\alpha}{n+2\gamma_{*}}}$ to derive the integral oscillation estimate
\eqn{hol.1}
$$
\nra{Du-(Du)_{B_{\sigma}}}_{L^{2}(B_{\sigma})}\le \tx{c}_{B}\sigma^{\frac{\alpha\gamma_{*}}{n+2\gamma_{*}+\alpha}},
$$
where $\tx{c}_{B}$ and $\gamma_{*}$ behave as indicated in \eqref{gh.1}$_{2}$ and \eqref{gh.6} respectively. Let us point out that if $3\tau/4\le \sigma\le 2\tau$, estimate \eqref{hol.1} trivially follows. Set $\beta_{*}:=\alpha\gamma_{*}/(n+2\gamma_{*}+\alpha)$, and notice that it shares the same monotonicity as $\gamma_{*}$ in \eqref{gh.6}. A standard covering argument and Campanato-Meyers characterization of H\"older continuity yields that $Du$ is locally H\"older continuous in $\ti{B}$, and, thanks to the arbitrariety of $\ti{B}\Subset B$ we further deduce the local H\"older continuity of $Du$ in $B$. Specifically, given any open subset $\ti{B}\Subset B$ it is $[Du]_{0,\beta_{*};\ti{B}}\le c$, with H\"older exponent $\beta_{*}\equiv \beta_{*}(\data_{*}(B),\tx{d}(\ti{B},B),\tx{H}_{0})\in (0,1)$ featuring the same monotonicity with respect to $\tx{H}_{0}$ in \eqref{gh.6}, and bounding constant $c\equiv c(\data_{*}(B),\tx{d}(\ti{B},B),\tx{H}_{0})$ that in $\tx{H}_{0}$ behaves as in $\eqref{gh.1}_{2}$. The content of this section is summarized in the next proposition.
\begin{proposition}\label{hphp}
    Under assumptions \eqref{1qa}, \eqref{ass.2}-\eqref{p2}, \eqref{assfr}, \eqref{sasa} and \eqref{u0}, let $B\subset 2B\Subset \Omega$ be concentric balls, $u\in (u_{0}+W^{1,q}_{0}(B))\cap \tx{K}^{\psi}(B)$ be the solution of Dirichlet problem \eqref{pdreg}, and $\tx{H}_{0}\ge 1$ be any constant satisfying \eqref{hhh}. There exists a threshold parameter $\mu_{\textnormal{max}}\equiv \mu_{\textnormal{max}}(n,q,\alpha)>1$ such that if $1\le \mu<\mu_{\textnormal{max}}$, then $Du$ is locally H\"older continuous in $B$. More precisely, given any open subset $\ti{B}\Subset B$, it is $[Du]_{0,\beta_{*};\ti{B}}\le c$, with H\"older exponent $\beta_{*}\equiv \beta_{*}(\data_{*}(B),\tx{d}(\ti{B},B),\tx{H}_{0})\in (0,1)$, and bounding constant $c\equiv c(\data_{*}(B),\tx{d}(\ti{B},B),\tx{H}_{0})$ featuring the following monotonicity:
    \eqn{hhh.1}
    $$
    \begin{cases}
        \ \tx{H}_{0}\mapsto \beta_{*}(\data_{*}(B),\tx{d}(\ti{B},B),\tx{H}_{0}) \ \ &\mbox{nonincreasing},\\
        \ \tx{H}_{0}\mapsto c(\data_{*}(B),\tx{d}(\ti{B},B),\tx{H}_{0}) \ \ &\mbox{nondecreasing}.
    \end{cases}
    $$
\end{proposition}
\section{Approximation scheme and proof of Theorems \ref{mt1}-\ref{mt2}, and Corollary \ref{cor1}}\label{app} 
\noindent In this section we collect the content of Sections \ref{ncz}-\ref{sce}, and design a suitable approximation scheme leading to the proof of Theorem \ref{mt2} (that in turn implies Theorem \ref{mt1}) and of Corollary \ref{cor1}.
\subsubsection*{Approximation scheme}
Let $B\subset 2B\Subset  \Omega$ be a ball, $\psi$ be as in \eqref{p2}, and $u\in \tx{K}^{\psi}_{\loc}(\Omega)$ be a local minimizer of functional $\mathcal{H}$. By Corollary \ref{alav.2}, \eqref{p2} and Remark \ref{r21}, there exists a sequence $\{\ti{u}_{\varepsilon}\}_{\varepsilon>0}\in W^{2,\infty}(B)\cap \tx{K}^{\psi}(B)$ - of course we can assume with no loss of generality that $0<\varepsilon<\min\{1,\dist(B,\partial \Omega)/4\}$ - such that
\eqn{5.1.0}
$$
\begin{array}{c}
\ \ti{u}_{\varepsilon}\to u\ \ \mbox{strongly in} \ \ W^{1,1}(B),\qquad \qquad\mathcal{H}(\ti{u}_{\varepsilon};B)\to \mathcal{H}(u;B),\\[8pt]\displaystyle
\ \nr{\ti{u}_{\varepsilon}}_{L^{\infty}(B)}\le 4\max\left\{\nr{u}_{L^{\infty}(2B)},\nr{\psi}_{L^{\infty}(2B)}\right\}=:\tx{S}.
\end{array}
$$
We then regularize integrand $\tx{L}$ in \eqref{ass.0}-\eqref{ass.2} by convolution in the gradient variable, i.e. given a sequence $\{\phi_{\delta}(x)\}_{\delta>0}:=\{\delta^{-n}\phi(x/\delta)\}_{\delta>0}\subset C^{\infty}(\mathbb{R}^{n})$ of nonnegative, radially symmetric mollifiers, we define $\tx{L}_{\delta}(z):=(\tx{L}*\phi_{\delta})(z)$, and observe that, after straightforward manipulations cf. \cite[Section 4.5]{dmon}, the newly defined integrand satisfies \eqref{assfr}. We then set 
\eqn{5.3}
$$\sigma_{\varepsilon}:=\left(1+\varepsilon^{-1}+\varepsilon^{-1}\nr{D\ti{u}_{\varepsilon}}_{L^{q}(B)}^{q}\right)^{-1} \ \Longrightarrow \ \sigma_{\varepsilon}\left(1+\nr{D\ti{u}_{\varepsilon}}_{L^{q}(B)}^{q}\right)= \texttt{o}(\varepsilon)\to 0,$$ 
define $\tx{a}_{\varepsilon}(x):=a(x)+\sigma_{\varepsilon}$, $\ell_{\delta}(z):=\ell_{s+\delta}(z)$, name
$$
\begin{array}{c}
\tx{H}_{\delta;\varepsilon}(x,z):=\tx{L}_{\delta}(z)+\tx{a}_{\varepsilon}(x)\ell_{\delta}(z)^{q},\qquad  \quad \lambda_{1;\delta,\varepsilon}(x,\snr{z}):=\ell_{1}(z)^{-\mu}+\tx{a}_{\varepsilon}(x)\ell_{\delta}(z)^{q-2},\\[8pt]\displaystyle
\tx{H}_{\varepsilon}(x,z):=\tx{L}(z)+\tx{a}_{\varepsilon}(x)\ell_{s}(z)^{q}, \\[8pt]\displaystyle
\tx{E}_{\delta;\varepsilon}(x,\snr{z}):=\int_{0}^{\snr{z}}\lambda_{1;\delta,\varepsilon}(x,t)t\dt=\frac{1}{2-\mu}\left(\ell_{1}(z)^{2-\mu}-1\right)+\frac{\tx{a}_{\varepsilon}(x)}{q}\left(\ell_{\delta}(z)^{q}-(s+\delta)^{q}\right),
\end{array}
$$
and introduce the corresponding regularized functionals, well defined whenever $w\in W^{1,q}(B)$,
$$
\mathcal{H}_{\delta;\varepsilon}(w;B):=\int_{B}\tx{H}_{\delta;\varepsilon}(x,Dw)\dx\qquad \mbox{and}\qquad \mathcal{H}_{\varepsilon}(w;B):=\int_{B}\tx{H}_{\varepsilon}(x,Dw)\dx.
$$
Integrand $\tx{H}_{\delta;\varepsilon}$ is of the same type considered in Section \ref{ncz}, i.e.~\eqref{assfr}, \eqref{sasa} and \eqref{assfr.2} hold with $\Lambda_{*}\equiv \Lambda_{*}(n,\Lambda,\tx{g},\mu,q)$, $s$ replaced by $s+\delta$ and $\ti{c}\equiv \ti{c}(\Lambda_{*},\tx{g},q,\nr{a}_{L^{\infty}(\Omega)},s,\delta)$, therefore there exists $u_{\delta;\varepsilon}\in (\ti{u}_{\varepsilon}+W^{1,q}_{0}(B))\cap \tx{K}^{\psi}(B)$, unique solution to Dirichlet problem
\eqn{pded}
$$
(\ti{u}_{\varepsilon}+W^{1,q}_{0}(B))\cap \tx{K}^{\psi}(B)\ni w\mapsto \min \mathcal{H}_{\delta;\varepsilon}(w;B),
$$
to whom all the regularity results obtained in Sections \ref{ncz}-\ref{sce} apply. In particular, the content of Section \ref{gb} yields that 
\eqn{bdd.1}
$$
\nr{u_{\delta;\varepsilon}}_{L^{\infty}(B)}\le \tx{S}.
$$
Notice also that by \eqref{d.1}, the mean value theorem and \eqref{ass.2} it is
\begin{flalign}\label{5.1}
\snr{\tx{L}_{\delta}(z)-\tx{L}(z)}+\snr{\ell_{\delta}(z)^{q}-\ell_{s}(z)^{q}}\le c\delta(\tx{g}(\snr{z})+1)+c\delta\ell_{1}(z)^{q-1}\le c\delta\ell_{1}(z)^{q-1},
\end{flalign}
for $c\equiv c(\data_{0})$, so keeping in mind that $D\ti{u}_{\varepsilon}\in L^{\infty}(B,\mathbb{R}^{n})$, and $Du_{\delta;\varepsilon}\in L^{q}(B,\mathbb{R}^{n})$, thanks to \eqref{5.1}, \eqref{5.3}, and $\eqref{assfr}_{2}$ it follows that
\begin{flalign}\label{5.2}
\begin{array}{c}
\displaystyle
 \nr{\tx{H}_{\delta;\varepsilon}(\cdot,D\ti{u}_{\varepsilon})-\tx{H}(\cdot,D\ti{u}_{\varepsilon})}_{L^{1}(B)}\lesssim \delta\left(\nr{\ell_{1}(D\ti{u}_{\varepsilon})}^{q-1}_{L^{\infty}(B)}+1\right)+\sigma_{\varepsilon}\nr{\ell_{1}(D\ti{u}_{\varepsilon})}_{L^{q}(B)}^{q}\equiv \texttt{o}_{\varepsilon}(\delta)+\texttt{o}(\varepsilon),\\[8pt]\displaystyle
 \snr{\mathcal{H}_{\varepsilon}(u_{\delta;\varepsilon};B)-\mathcal{H}_{\delta;\varepsilon}(u_{\delta;\varepsilon};B)}\lesssim \delta+\delta\nr{\tx{g}(\snr{Du_{\delta;\varepsilon}})}_{L^{1}(B)}+\delta\nr{\tx{a}_{\varepsilon}(\cdot)\snr{Du_{\delta;\varepsilon}}^{q}}_{L^{1}(B)}^{\frac{q-1}{q}}\lesssim \delta+\delta\mathcal{H}_{\delta;\varepsilon}(u_{\delta;\varepsilon};B)^{\frac{q-1}{q}}
\end{array}
\end{flalign}
up to constants depending on $(\data_{0},\nr{a}_{L^{\infty}(B)})$. By \eqref{pded}, $\eqref{5.2}$ and \eqref{5.1.0} we then have
\begin{flalign}\label{5.6}
\left\{
\begin{array}{c}
\displaystyle
\mathcal{H}_{\delta;\varepsilon}(u_{\delta;\varepsilon};B)\le\mathcal{H}_{\delta;\varepsilon}(\ti{u}_{\varepsilon};B)
\le\mathcal{H}(\ti{u}_{\varepsilon};B)+\texttt{o}_{\varepsilon}(\delta)+\texttt{o}(\varepsilon)\le\mathcal{H}(u;B)+\texttt{o}_{\varepsilon}(\delta)+\texttt{o}(\varepsilon),\\[8pt]\displaystyle
\mathcal{H}_{\varepsilon}(u_{\delta;\varepsilon};B)\le\mathcal{H}(u;B)+\texttt{o}_{\varepsilon}(\delta)+\texttt{o}(\varepsilon)+\texttt{o}(\delta),
\end{array}
\right.
\end{flalign}
thus for fixed $\varepsilon>0$ we can find a subsequence $\{u_{\delta;\varepsilon}\}_{\delta>0}\subset (\ti{u}_{\varepsilon}+W^{1,q}_{0}(B))\cap \tx{K}^{\psi}(B)$ such that
\eqn{5.7}
$$u_{\delta;\varepsilon}\rightharpoonup_{\delta\to 0}u_{\varepsilon}\in (\ti{u}_{\varepsilon}+W^{1,q}_{0}(B))\cap \tx{K}^{\psi}(B) \ \  \mbox{weakly in} \ \  W^{1,q}(B).$$
In \eqref{5.6}$_{2}$ we first send $\delta\to 0$, use \eqref{5.7} and $W^{1,q}$-weak lower semicontinuity, to deduce
\eqn{5.8}
$$
\mathcal{H}_{\varepsilon}(u_{\varepsilon};B)\le \mathcal{H}(u;B)+\texttt{o}(\varepsilon).
$$
Estimate \eqref{5.8}, $\eqref{ass.1}_{1}$ and de la Vall\'ee Poussin theorem yield that 
\eqn{5.9}
$$
\mbox{there exists} \ \ \bar{u}\in (u+W^{1,1}_{0}(B))\cap \tx{K}^{\psi}(B) \ \ \mbox{such that} \ \ u_{\varepsilon}\rightharpoonup \bar{u} \ \ \mbox{weakly in} \ \ W^{1,1}(B).
$$
As $\varepsilon\to 0$ in \eqref{5.8}, by weak lower semicontinuity and minimality we derive
$$
\mathcal{H}(\bar{u};B)\le \liminf_{\varepsilon\to 0}\mathcal{H}_{\varepsilon}(u_{\varepsilon};B)\le \limsup_{\varepsilon\to 0}\mathcal{H}_{\varepsilon}(u_{\varepsilon};B)\le \mathcal{H}(u;B)\le \mathcal{H}(\bar{u};B),
$$
which by standard strict convexity arguments (recall that $\tx{K}^{\psi}(B)$ is a convex subset of $W^{1,1}(B)$) implies that $u\equiv \bar{u}$ on $B$.
\subsection{Proof of Theorem \ref{mt2}} By construction, functional $\mathcal{H}_{\delta;\varepsilon}$ satisfies \eqref{assfr}-\eqref{assfr.2}, and obstacle $\psi$ is as in \eqref{p2}, so Propositions \ref{lipb} and \ref{hphp} apply. Specifically, letting $\ti{B}\Subset B$ be an open subset, 
\eqn{hhh.2}
$$\tx{H}_{0}:=4+4\max\{\nr{u}_{L^{\infty}(2B)},\nr{\psi}_{L^{\infty}(2B)}\}+\mathcal{H}(u;B),$$
cf. \eqref{hhh}, \eqref{5.1.0}$_{3}$ and $\eqref{5.6}_{2}$ - estimates 
\eqn{51.0}
$$
\left\{
\begin{array}{c}
\displaystyle
\nr{\tx{E}_{\delta;\varepsilon}(\cdot,\snr{Du_{\delta;\varepsilon}})}_{L^{\infty}(\ti{B})}\le c\mathcal{H}_{\delta;\varepsilon}(\ti{u}_{\varepsilon};B)^{\tx{b}}+c,\qquad \quad [Du_{\delta;\varepsilon}]_{0,\beta_{*};\ti{B}}\le \tx{c}_{B}\\[8pt]\displaystyle
c\equiv c(\data(B),\tx{d}(\ti{B},B),\nr{\ti{u}_{\varepsilon}}_{L^{\infty}(B)})\\[8pt]\displaystyle
\beta_{*},\tx{c}_{B}\equiv \beta_{*},\tx{c}_{B}(\data(B),\tx{d}(\ti{B},B),\nr{\ti{u}_{\varepsilon}}_{L^{\infty}(B)},\mathcal{H}_{\delta;\varepsilon}(\ti{u}_{\varepsilon};B)),
\end{array}
\right.
$$
hold true. Looking at the monotonicity of $c$ and $\beta_{*},\tx{c}_{B}$ in $\nr{\ti{u}_{\varepsilon}}_{L^{\infty}(B)}$ and $(\nr{\ti{u}_{\varepsilon}}_{L^{\infty}(B)},\mathcal{H}_{\delta;\varepsilon}(\ti{u}_{\varepsilon};B))$ respectively, cf. \eqref{hhh.1}, and keeping in mind \eqref{hhh.2}, by $\eqref{5.1.0}_{3}$ and $\eqref{5.6}$ we can update the dependencies of $c,\tx{c}_{B},\beta_{*}$ as $c\equiv c(\data(2B),\tx{d}(\ti{B},B))$ and $\beta_{*},\tx{c}_{B}\equiv \beta_{*},\tx{c}_{B}(\data(2B),\tx{d}(\ti{B},B),\mathcal{H}(u;B))$ by slightly increasing $c$ and $\tx{c}_{B}$, and slightly decreasing $\beta_{*}$ (without relabelling). This together with \eqref{5.6} turns \eqref{51.0} into
\eqn{51.2}
$$
\nr{\tx{E}_{\delta;\varepsilon}(\cdot,\snr{Du_{\delta;\varepsilon}})}_{L^{\infty}(\ti{B})}\le c\mathcal{H}(u;B)^{\tx{b}}+c\qquad \mbox{and}\qquad [Du_{\delta;\varepsilon}]_{0,\beta_{*};\ti{B}}\le \tx{c}_{B},
$$
with the stated dependencies of bounding constants and H\"older exponent. We can then (locally) update the convergence in \eqref{5.7} as $u_{\delta;\varepsilon}\rightharpoonup^{*}u_{\varepsilon}$ weak$^{*}$ in $W^{1,\infty}(\ti{B})$, observe that $$\snr{z}^{2-\mu}+a(x)\ell_{s}(z)^{q}\lesssim_{\mu,q,\nr{a}_{L^{\infty}(B)}}\tx{E}_{\delta;\varepsilon}(x,\snr{z})+1,$$ and send $\delta\to 0$ in \eqref{51.2} to get
\begin{flalign*}
\nr{Du_{\varepsilon}}_{L^{\infty}(\ti{B})}^{2-\mu}+\nr{a(\cdot)\ell_{s}(Du_{\varepsilon})^{q}}_{L^{\infty}(\ti{B})}\le c\mathcal{H}(u;B)^{\texttt{b}}+c.
\end{flalign*}
Thanks to \eqref{5.9} and weak$^{*}$-convergence we obtain \eqref{fin.1}. 
The local H\"older continuity of $Du$ then follows via Arzela-Ascoli theorem after passing to the limit in \eqref{51.2}$_{2}$ on $\ti{B}$ and $Du\in C^{0,\beta}(\ti{B},\mathbb{R}^{n})$ for all $\beta\in (0,\beta_{*})$. A standard covering argument then completes the proof.
\subsection{Proof of Theorem \ref{mt1}} Theorem \ref{mt1} is a direct consequence of Theorem \ref{mt2}: in fact, the (formal) choice $\psi\equiv -\infty$ allows restoring the unconstrained setting (see comments immediately above Remark \ref{r21}). Moreover, the logarithmic integrand $\tx{L}(z):=\snr{z}\log(1+\snr{z})$ is covered by \eqref{ass.1}, as well as the singular case $s=0$ is admissible for the $q$-power component of \eqref{ll1}. The assumptions of Theorem \ref{mt2} are then satisfied and the gradient local H\"older continuity of a priori bounded minimizers follows, together with Lipschitz estimate \eqref{lip.0}.
\subsection{Proof of Corollary \ref{cor1}}\label{pc1} Once known the gradient local H\"older continuity from Theorem \ref{mt2}, the improvement in the nonsingular case $s>0$ can be derived by adapting the arguments in \cite[Section 5.11]{dm}. By assumptions, in \eqref{ll1} integrand $\tx{L}$ is locally $C^{2}$-regular and $s>0$, so we can directly work on the original minimizer, that thanks to Theorem \ref{mt2} is already $\beta$-H\"older continuous for some $\beta\in (0,1)$. Moreover, from Theorem \ref{mt2} we know that once fixed open sets $\ti{B}\Subset B\subset 2B\Subset \Omega$, bound 
\eqn{fin.11}
$$
\left\{
\begin{array}{c}
\displaystyle
\nr{Du}_{L^{\infty}(\ti{B})}+[Du]_{0,\beta;\ti{B}}\le \tx{c}_{0}\\[8pt]\displaystyle
\beta,\tx{c}_{0}\equiv \beta,\tx{c}_{0}(\data(2B),\tx{d}(\ti{B},B),\mathcal{H}(u;B))
\end{array}
\right.
$$
holds. Of course, there is no loss of generality in assuming that $0<\beta< \alpha/2$ if $\mu>1$ or $0<\beta<\alpha$ for $\mu=1$, otherwise there would be nothing to prove. We then define
$$
\tx{f}_{B}:=\nr{\diver (\partial \tx{L}(D\psi))}_{L^{\infty}(B)}+4nq\left(\nr{a}_{L^{\infty}(B)}+1\right)\nr{\snr{D\psi}^{q-2}D^{2}\psi}_{L^{\infty}(B)},
$$
set threshold $\rrr_{*}\in (0,1]$ so small that
\eqn{rs}
$$
\tx{c}_{\tx{p}}\Lambda\rrr_{*}\tx{f}_{B}\le \frac{1}{2} \ \Longrightarrow \ \rrr_{*}\equiv \rrr_{*}(\data(B)),
$$
pick any ball $B_{\tau}(\equiv B_{\tau}(x_{0}))\subset B_{2\tau}\Subset \ti{B}$ with radius $\tau\in  (0,\rrr_{*}]$, introduce functions
$$
\tx{H}_{i;\tau}(z;B_{2\tau}):=\tx{L}(z)+\left(\tau^{\alpha}+\inf_{x\in B_{2\tau}}a(x)\right)\ell_{s}(z)^{q},\qquad \quad \tx{f}_{\tau}:=-\diver(\partial\tx{H}_{i;\tau}(D\psi;B_{2\tau}))\in L^{\infty}(B),
$$
and let $v_{1}\in u+W^{1,q}_{0}(B_{2\tau})$ be the solution of Dirichlet problem 
$$
u+W^{1,q}_{0}(B_{2\tau})\ni w\mapsto \min \int_{B_{2\tau}}\tx{H}_{i;\tau}(Dw;B_{2\tau})-\tx{f}_{\tau}w\dx,
$$
 and $v_{2}\in v_{1}+W^{1,q}_{0}(B_{\tau})$ be the solution to Dirichlet problem 
 $$
v_{1}+W^{1,q}_{0}(B_{\tau})\ni w\mapsto \min \int_{B_{\tau}}\tx{H}_{i;\tau}(Dw;B_{2\tau})\dx.
$$
Estimates \eqref{gh.2}-\eqref{gh.4} can then be rearranged as
 \eqn{5.3.0}
$$
\left\{
\begin{array}{c}
\displaystyle
\ \nr{Dv_{2}}_{L^{\infty}(B_{3\tau/4})}+\nr{Dv_{1}}_{L^{\infty}(B_{3\tau/2})}\le c\\[8pt]\displaystyle
\ \nr{Du-Dv_{2}}_{L^{2}(B_{3\tau/4})}^{2}\le c\tau^{n+\alpha},
\end{array}
\right.
$$
for $c\equiv c(\data(2B),\tx{d}(\ti{B},B),\mathcal{H}(u;B))$ - keep in mind \eqref{fin.11}. We next observe that in $B_{3\tau/4}$, matrix $\tx{A}_{i;\tau}(x):=\partial^{2}\tx{H}_{i;\tau}(Dv_{2}(x);B_{2\tau})$ is uniformly elliptic with
\eqn{5.3.1}
$$
\frac{\snr{\xi}^{2}}{c(1+\tx{c}_{0}^{2})^{\mu/2}}\stackrel{\eqref{ass.1}_{2}}{\le} \langle\tx{A}_{i;\tau}(x)\xi,\xi\rangle\qquad \mbox{and}\qquad \snr{\tx{A}_{i;\tau}}\stackrel{\eqref{ass.1}_{2}}{\le}c(\tx{g}(\tx{c}_{0}) +1)+c\left(\nr{a}_{L^{\infty}(B)}+1\right)s^{q-2},
$$
for $c\equiv c(n,\Lambda,q)$. The availability of \eqref{5.3.0}-\eqref{5.3.1} allows us proceeding exactly as in \cite[Section 5.11]{dm}, and eventually conclude, after a standard covering argument, that $Du\in C^{0,\alpha/2}_{\loc}(\Omega,\mathbb{R}^{n})$. 
\subsubsection*{H\"older improvement in the plain logarithmic case} Within the same setting as Section \ref{pc1}, assume $\mu=1$, and jump back to the above comparison scheme. With $v_{1}$, $v_{2}$ as before, we rearrange \eqref{gh.h} (see also \eqref{com.666}) via $\eqref{Vm}_{2}$ and \eqref{fin.11} as 
\begin{eqnarray}\label{rb2}
\nra{V_{1,1}(Du)-V_{1,1}(Dv_{1})}_{L^{2}(B_{2\tau})}&\le&c\left(\left([a]_{0,\alpha;B}+1\right)^{\frac{1}{2}}\tau^{\frac{\alpha}{2}}+\nr{\tx{f}_{\tau}}_{L^{\infty}(B)}^{\frac{1}{2}}\tau^{\frac{1}{2}}\right)\nra{Du-Dv_{1}}_{L^{1}(B_{2\tau})}^{\frac{1}{2}}\nonumber \\
&\le&c\tau^{\frac{\alpha}{2}}\nra{V_{1,1}(Du)-V_{1,1}(Dv_{1})}_{L^{2}(B_{2\tau})}\nonumber \\
&&+c\tau^{\frac{\alpha}{2}}\nra{(V_{1,1}(Du)-V_{1,1}(Dv_{1}))\ell_{1}(Du)^{1/2}}_{L^{1}(B_{2\tau})}^{\frac{1}{2}}\nonumber \\
&\le&\left(c\rrr_{*}^{\frac{\alpha}{2}}+\frac{1}{4}\right)\nra{V_{1,1}(Du)-V_{1,1}(Dv_{1})}_{L^{2}(B_{2\tau})}+c\tau^{\alpha}\nra{\ell_{1}(Du)}_{L^{1}(B_{2\tau})}^{\frac{1}{2}}\nonumber \\
&\le&\left(c\rrr_{*}^{\frac{\alpha}{2}}+\frac{1}{4}\right)\nra{V_{1,1}(Du)-V_{1,1}(Dv_{1})}_{L^{2}(B_{2\tau})}+c\tau^{\alpha},
\end{eqnarray}
for $c\equiv c(\data(2B),\tx{d}(\ti{B},B),\mathcal{H}(u;B))$, while concerning estimate \eqref{gh.h1}, by $\eqref{Vm}_{2}$ and $\eqref{5.3.0}_{1}$ we have
\begin{eqnarray}\label{rb1}
\nra{V_{1,1}(Dv_{1})-V_{1,1}(Dv_{2})}_{L^{2}(B_{\tau})}&\le&c\tau^{\frac{1}{2}}\nr{\tx{f}_{\tau}}_{L^{\infty}(B)}^{\frac{1}{2}}\nra{Dv_{1}-Dv_{2}}_{L^{1}(B_{\tau})}^{\frac{1}{2}}\nonumber \\
&\le&\left(c\rrr_{*}^{\frac{1}{2}}+\frac{1}{4}\right)\nra{V_{1,1}(Dv_{1})-V_{1,1}(Dv_{2})}_{L^{2}(B_{\tau})}+c\tau\nra{\ell_{1}(Dv_{1})}_{L^{1}(B_{\tau})}^{\frac{1}{2}}\nonumber \\
&\le&\left(c\rrr_{*}^{\frac{1}{2}}+\frac{1}{4}\right)\nra{V_{1,1}(Dv_{1})-V_{1,1}(Dv_{2})}_{L^{2}(B_{\tau})}+c\tau,
\end{eqnarray}
with $c\equiv c(\data(2B),\tx{d}(\ti{B},B),\mathcal{H}(u;B))$. We then restrict further threshold $\rrr_{*}$ with respect to \eqref{rs}:
$$
\max\left\{\tx{c}_{\tx{p}}\Lambda\tx{r}_{*}\tx{f}_{B},c\rrr_{*}^{\frac{\alpha}{2}},c\rrr_{*}^{\frac{1}{2}}\right\}\le \frac{1}{4} \ \Longrightarrow \ \rrr_{*}\equiv \rrr_{*}(\data(2B),\tx{d}(\ti{B},B),\mathcal{H}(u;B)),
$$
so that we can reabsorbe terms in \eqref{rb2}-\eqref{rb1} and derive
\begin{eqnarray}\label{5.3.2}
\nra{Du-Dv_{2}}_{L^{2}(B_{3\tau/4})}&\stackrel{\eqref{Vm}_{1}}{\le}&c\nra{V_{1,1}(Du)-V_{1,1}(Dv_{2})}_{L^{2}(B_{\tau})}\nonumber \\
&\le&c\nra{V_{1,1}(Du)-V_{1,1}(Dv_{1})}_{L^{2}(B_{2\tau})}\nonumber \\
&&+c\nra{V_{1,1}(Dv_{1})-V_{1,1}(Dv_{2})}_{L^{2}(B_{\tau})}\le c\tau^{\alpha},
\end{eqnarray}
for $c\equiv c(\data(2B),\tx{d}(\ti{B},B),\mathcal{H}(u;B))$. The bound in \eqref{5.3.2} now updates $\eqref{5.3.0}_{2}$, and, together with \eqref{5.3.1} allows repeating verbatim the arguments in \cite[Section 10.1]{piovra}, and conclude after covering that $Du\in C^{0,\alpha}_{\loc}(\Omega,\mathbb{R}^{n})$ if $\alpha\in (0,1)$ and $Du\in C^{0,\gamma}_{\loc}(\Omega,\mathbb{R}^{n})$ for all $\gamma\in (0,1)$ if $\alpha=1$. The proof is complete.

\section{Sharpness and Theorems \ref{count}-\ref{count.3}, after Balci \& Diening \& Surnachev}\label{sec6}
\noindent In this section we revisit the fractal constructions developed in \cite{balci,balci2}. Our ultimate goal is the proof of Theorems \ref{count}-\ref{count.3}, establishing the sharpness of Theorems \ref{mt1}-\ref{mt2} and of Lemma \ref{alav}. Here, we shall denote $Q:=(-1,1)^{n}$, $2Q:=(-2,2)^{n}$, $Q':=(-98/100,98/100)^{n}$, $Q'':=(-97/100,97/100)^{n}$, often use the product notation $x\equiv (\bar{x},x_{n})\in \mathbb{R}^{n-1}\times \mathbb{R}$, and permanently work under assumptions \eqref{66.0} and \eqref{ass.0}. Before proving Theorems \ref{count}-\ref{count.3}, we need to outline some common background.
\subsubsection*{Generalized Cantor sets} A key ingredient in these arguments is the possibility of building multidimensional fractals of Cantor type with prescribed Hausdorff dimension ranging in $(0,n-1)$. To this aim, fix any $0<\varepsilon<\min\{q-1-\alpha,n-1\}$, which is possible via \eqref{66.0}. Following \cite[Section 2.2]{balci}, \cite[Section 2.3]{balci2}, set $0<\lambda:=(1/2)^{(n-1)/(n-1-\varepsilon)}<1/2$, and for $i\in \mathbb{N}$ define sequence $\{\textit{l}_{i}\}_{i\in \mathbb{N}}:=\{\lambda^{i}\}_{i\in \mathbb{N}}$. It is readily verified that $\{\textit{l}_{i}\}_{i\in \mathbb{N}}$ is decreasing with $\textit{l}_{i}>2\textit{l}_{i+1}$ and $\textit{l}_{i-1}-2\textit{l}_{i}>\textit{l}_{i}-2\textit{l}_{i+1}$ for all $i\in \mathbb{N}$. We can then apply the standard construction of multidimensional Cantor sets to design set $\mathcal{C}^{n-1}_{\varepsilon}$ that is the Cartesian product of $n-1$ generalised Cantor sets defined upon sequence $\{\textit{l}_{i}\}_{i\in \mathbb{N}}$ starting from interval $[-1/2,1/2]$, and the corresponding Cantor measure $\gamma^{n-1}_{\varepsilon}$. By construction and \cite[Definition 5]{balci2}, it is $\gamma^{n-1}_{\varepsilon}(\mathcal{C}_{\varepsilon}^{n-1})=1$, $\supp(\gamma_{\varepsilon}^{n-1})=\mathcal{C}_{\varepsilon}^{n-1}$, and $\tx{d}:=\textnormal{dim}_{\mathcal{H}}(\mathcal{C}^{n-1}_{\varepsilon})=n-1-\varepsilon$. Finally, for simplicity set $\mathcal{C}:=\mathcal{C}_{\varepsilon}^{n-1}\times \{0\}$ and $\gamma:=\gamma_{\varepsilon}^{n-1}\times \delta_{0}$. 
\subsubsection*{Technical setup} Recall from \cite[Lemma 5]{balci} that there exist maps $\chi_{*},\chi_{a}\in C^{\infty}(\mathbb{R}^{n}\setminus \mathcal{C})$ satisfying
\eqn{6.0}
$$
\begin{array}{c}
\displaystyle
\mathds{1}_{\{\dist(\bar{x},\mathcal{C}^{n-1}_{\varepsilon})\le 2\snr{x_{n}}\}}\le \chi_{*}\le \mathds{1}_{\{\dist(\bar{x},\mathcal{C}^{n-1}_{\varepsilon})\le 4\snr{x_{n}}\}},\qquad \quad \snr{D\chi_{*}}\lesssim_{n}\snr{x_{n}}^{-1}\mathds{1}_{\{2\snr{x_{n}}\le \dist(\bar{x},\mathcal{C}^{n-1}_{\varepsilon})\le 4\snr{x_{n}}\}},\\[8pt]\displaystyle
\mathds{1}_{\{\dist(\bar{x},\mathcal{C}^{n-1}_{\varepsilon})\le \snr{x_{n}}/2\}}\le \chi_{a}\le \mathds{1}_{\{\dist(\bar{x},\mathcal{C}^{n-1}_{\varepsilon})\le 2\snr{x_{n}}\}},\qquad \quad \snr{D\chi_{a}}\lesssim_{n}\snr{x_{n}}^{-1}\mathds{1}_{\{\snr{x_{n}}/2\le \dist(\bar{x},\mathcal{C}^{n-1}_{\varepsilon})\le 2\snr{x_{n}}\}}.
\end{array}
$$
Moreover pick $\theta\in C^{\infty}(0,\infty)$ with $\mathds{1}_{(1/2,\infty)}\le \theta\le \mathds{1}_{(1/4,\infty)}$ and $\snr{\theta'}\le 6$ and introduce 
\begin{flalign*}
\begin{array}{c}
\displaystyle
\tx{Z}_{n}(x):=\frac{\snr{\bar{x}}^{1-n}}{\sigma_{n-1}}\theta\left(\frac{\snr{\bar{x}}}{\snr{x_{n}}}\right)\begin{bmatrix} 0 & -\bar{x}^{t} \\ \bar{x} & 0 \end{bmatrix}\\[12pt]\displaystyle
\tx{Z}:=\left(\gamma^{n-1}_{\varepsilon}\times \delta_{0}\right)*\tx{Z}_{n},\qquad \quad \tx{z}:=\diver(\tx{Z}),
\end{array}
\end{flalign*}
where $\delta_{0}$ is the delta measure centered in zero, $\sigma_{n-1}$ is the surface area of the $(n-1)$-dimensional sphere, and symbol "$t$" denotes transposition. By \cite[Propositions 2 and 14]{balci} and \cite[Chapter 10.3]{lf} it is $\tx{Z}_{n}\in W^{1,1}_{\loc}(\mathbb{R}^{n},\otimes_{2}\mathbb{R}^{n})\cap C^{\infty}(\mathbb{R}^{n}\setminus \{0\},\otimes_{2}\mathbb{R}^{n})$, $\tx{Z}\in W^{1,1}(2Q,\otimes_{2}\mathbb{R}^{n})\cap C^{\infty}(\overline{2Q}\setminus \mathcal{C},\otimes_{2}\mathbb{R}^{n})$, and $\tx{z}\in L^{1}(2Q,\mathbb{R}^{n})\cap C^{\infty}(\overline{2Q}\setminus \mathcal{C},\mathbb{R}^{n})$. With $\chi_{*},\chi_{a}$ as in \eqref{6.0} above and $\phi\in C^{\infty}_{c}(Q)$ being such that $\mathds{1}_{(-3/4,3/4)^{n}}\le \phi\le \mathds{1}_{(-5/6,5/6)^{n}}$ and $\snr{D\phi}\lesssim_{n}1$, we let
\begin{flalign}\label{6.3}
\begin{array}{c}
\displaystyle
u_{*}(x):=\frac{1}{2}\textnormal{sgn}(x_{n})\chi_{*}(x),\qquad \ti{u}(x):=(1-\phi(x))u_{*}(x),\qquad a(x):=\snr{x_{n}}^{\alpha}\chi_{a}(x)\\ [12pt]\displaystyle 
\tx{b}(x):=\mathds{1}_{\{\dist(\bar{x},\mathcal{C}^{n-1}_{\varepsilon})\le \snr{x_{n}}/2\}}\snr{x_{n}}^{-\varepsilon},
\end{array}
\end{flalign}
cf. \cite[Definition 10 and Section 5]{balci2} and \cite[Definitions 1 and 9 (b)]{balci}. Notice that $u_{*}\in L^{\infty}(\overline{2Q})\cap W^{1,1}(2Q)\cap C^{\infty}(\overline{2Q}\setminus \mathcal{C})$, $\ti{u}\in C^{\infty}(\overline{2Q})$,\footnote{Keep in mind the properties of $u_{*}$ and $\phi$, and that $\mathcal{C}\Subset (-4/6,4/6)^{n}$.} $0\le a(\cdot)\in C^{\alpha}(\overline{2Q})$, cf. \cite[Chapter 10.3]{lf} and \cite[Lemma 15]{bal23}, and thanks to \cite[Lemma 6 and Proposition 15]{balci} and $\eqref{6.0}$ it is 
\eqn{6.1}
$$
\left\{
\begin{array}{c}
\displaystyle
\snr{Du_{*}}\lesssim_{n,\varepsilon}\snr{x_{n}}^{-1}\mathds{1}_{\{2\snr{x_{n}}\le \dist(\bar{x},\mathcal{C}^{n-1}_{\varepsilon})\le 4\snr{x_{n}}\}},\qquad Du_{*}\in L^{p}(2Q,\mathbb{R}^{n})\quad \mbox{for all} \ \ p\in [1,1+\varepsilon)\\[8pt]\displaystyle
\tx{b}\in L^{d}(2Q)\quad \mbox{for all} \ \ d\in [1,(1+\varepsilon)/\varepsilon),\qquad \snr{\tx{z}}\lesssim_{n,\varepsilon}\tx{b}\\[8pt]\displaystyle
 \{x\in 2Q\colon \snr{\tx{z}}>0\}\subset\{x\in 2Q\colon \tx{b}\not =0\}\subset \{x\in 2Q\colon a(x)=\snr{x_{n}}^{\alpha}\}\\[8pt]\displaystyle
 \{x\in 2Q\colon \snr{Du_{*}}\not =0\}\subset \{x\in 2Q\colon a(x)=0\}.
\end{array}
\right.
$$
Let $\mathcal{H}$, $\mathcal{G}$ be the integrals governed by integrands $\tx{H}$ in \eqref{hl.00}$_{2}$, $\tx{G}$ in $\eqref{ggg}$ respectively, with coefficient $a$ defined in \eqref{6.3} and exponents $(q,\alpha)$ satisfying \eqref{66.0}, and name $\tx{H}^{*}$, $\tx{G}^{*}$ their convex conjugates. Thanks to $\eqref{6.1}_{1,4}$ it is $\mathcal{H}(u_{*};2Q),\mathcal{G}(u_{*};2Q)<\infty$. Moreover, from \eqref{ass.1}$_{1}$ and $\eqref{6.1}_{3}$, whenever $\snr{\tx{z}},\tx{b}\not =0$ it is $a(x)=\snr{x_{n}}^{\alpha}$, thus 
\eqn{hg}
$$
\tx{H}(x,\tx{z})\ge \snr{x_{n}}^{\alpha}\snr{\tx{z}}^{q}-\Lambda\qquad \mbox{and}\qquad \tx{G}(x,\tx{b})\ge \snr{x_{n}}^{\alpha}\tx{b}^{q}. 
$$
We then bound using $\eqref{6.1}_{2,3}$ and \eqref{hg},
\begin{flalign}\label{6.4.4}
\tx{H}^{*}(x,\tx{z})=\sup_{z\in \mathbb{R}^{n}}\left\{\langle\tx{z},z\rangle-\tx{H}(x,z)\right\}\le\sup_{z\in \mathbb{R}^{n}}\left\{\langle\tx{z},z\rangle-\snr{x_{n}}^{\alpha}\snr{z}^{q}\right\}+\Lambda\lesssim_{n,q,\varepsilon}\snr{x_{n}}^{-\frac{\alpha}{q-1}}\tx{b}^{q'}+\Lambda,
\end{flalign}
and in a totally similar fashion, we also get
\eqn{6.4.2}
$$
\tx{G}^{*}(x,\tx{b})\lesssim_{n,q,\varepsilon}\snr{x_{n}}^{-\frac{\alpha}{q-1}}\tx{b}^{q'}.
$$
Via \eqref{66.0}, $\eqref{6.1}_{2}$, \eqref{6.4.4} and \eqref{6.4.2} we deduce that integrals
$$
\mathcal{H}^{*}(\tx{z};2Q):=\int_{2Q}\tx{H}^{*}(x,\tx{z})\dx,\qquad \quad  \mathcal{G}^{*}(\tx{b};2Q):=\int_{2Q}\tx{G}^{*}(x,\tx{b})\dx
$$
are finite. More precisely, we can apply basic properties of convex conjugation and \eqref{6.4.4}-\eqref{6.4.2} to estimate
\begin{eqnarray}\label{6.4}
 \mathcal{H}^{*}(\tx{z};Q)+\mathcal{G}^{*}(\tx{b};Q)&\le& c\mathcal{H}^{*}(\tx{z};2Q)+c\mathcal{G}^{*}(\tx{b};2Q)+c\le c\int_{2Q}\snr{x_{n}}^{-\frac{\alpha}{q-1}}\tx{b}^{q'}\dx+c\nonumber \\
&=&c\int_{2Q}\mathds{1}_{\{\dist(\bar{x},\mathcal{C}^{n-1}_{\varepsilon})\le \snr{x_{n}}/2\}}\snr{x_{n}}^{-\frac{\alpha+\varepsilon q}{q-1}}\dx+c\nonumber \\
&\le&c\int_{0}^{2}t^{-\frac{\alpha+\varepsilon q}{q-1}}\mathcal{H}^{n-1}\left(\dist(\cdot,\mathcal{C}^{n-1}_{\varepsilon})\le t/2\right)\dt+c\nonumber\\
&\le&c\int_{0}^{2}t^{-\frac{\alpha+\varepsilon q}{q-1}+\varepsilon}\dt+c\le \tx{c}_{*}(n,\Lambda,\tx{g},q,\alpha,\varepsilon)<\infty,
\end{eqnarray}
by \cite[Lemmas 6 (b)]{balci}, that applies thanks to \eqref{66.0}. Let us finally record the elementary estimate valid whenever $w_{1},w_{2}\in W^{1,1}(Q)$ satisfy $\mathcal{H}(w_{1};Q)\le \mathcal{H}(w_{2};Q)<\infty$. By $\eqref{ass.1}_{1}$ we find
\begin{eqnarray}\label{elel}
 \mathcal{G}(w_{1};Q)&\le& \Lambda\mathcal{H}(w_{1};Q)+\Lambda^{2}\snr{Q}\le \Lambda\mathcal{H}(w_{2};Q)+\Lambda^{2}\snr{Q}\nonumber \\
 &\le& 2^{q+\alpha}\Lambda^{2}\mathcal{G}(w_{2};Q)+2^{q+\alpha+2}\Lambda^{2}\snr{Q}\le\ti{\tx{c}}\left(\mathcal{G}(w_{2};Q)+\snr{Q}\right),
\end{eqnarray}
with $\ti{\tx{c}}:=2^{100(nq+\alpha)}\Lambda^{2}n^{2q}(1+c_{n;\varepsilon})$, where $c_{n;\varepsilon}\equiv c_{n;\varepsilon}(n,\varepsilon)>0$ is the product of the two constants from \cite[Lemma 20 and estimate (23)]{balci2},\footnote{At this stage we do not need $\ti{\tx{c}}$ so large, in particular, for \eqref{elel} constant $c_{n;\varepsilon}$ is unnecessary as \cite[Lemma 20 and estimate (23)]{balci2} are not applied. Nonetheless we keep such a definition of $\ti{\tx{c}}$ to account for multiple quantities appearing in some of the forthcoming estimates.} so that $\ti{\tx{c}}\equiv \ti{\tx{c}}(n,\Lambda,q,\alpha,\varepsilon)\ge 1$.
\subsubsection*{Energy estimates} This is the core (abstract) part of the proof. Observe first that by $\eqref{hl.00}_{2}$, \eqref{ggg}, $\eqref{6.1}_{1,4}$, $\eqref{ass.1}_{1}$, \eqref{ass.2} and \cite[Lemma 6, (b)]{balci} (recall that here $\tx{g}$ either satisfies \eqref{ass.2} or $\tx{g}= 1$), whenever $m,\sigma\ge 1$ are constants we have
\begin{eqnarray}\label{6.4.1}
 \mathcal{H}(mu_{*};Q)+\mathcal{G}(mu_{*};Q)&\le& \mathcal{H}(mu_{*};2Q)+\mathcal{G}(mu_{*};2Q)+c\snr{2Q}\nonumber \\
 &\le&c\left(\mathcal{G}(mu_{*};2Q)+\snr{2Q}\right)\le cm^{1+\delta_{*}}\int_{2Q}\snr{Du_{*}}^{1+\delta_{*}}\dx+c\snr{2Q}\nonumber \\
&\le&cm^{1+\delta_{*}}\int_{2Q}\snr{Du_{*}}^{1+\frac{\varepsilon}{2}}\dx+cm^{1+\delta_{*}}\snr{2Q}=:m^{1+\delta_{*}}\tx{c}^{*}(n,\Lambda,\tx{g},\varepsilon,\delta_{*})<\infty,
\end{eqnarray}
where $0<\delta_{*}\le\min\{\varepsilon,\alpha\}/2$ if $\tx{g}$ is as in \eqref{ass.2} or $\delta_{*}= 0$ if $\tx{g}= 1$, and, by \eqref{6.4.4}-\eqref{6.4},
\begin{eqnarray}\label{6.4.3}
   \max\left\{\mathcal{H}^{*}(\sigma\tx{z};Q)+\mathcal{G}^{*}(\sigma\tx{b};Q),\mathcal{H}^{*}(\sigma\tx{z};2Q)+\mathcal{G}^{*}(\sigma\tx{b};2Q)\right\}&\le& c\sigma^{q'}\int_{2Q}\snr{x_{n}}^{-\frac{\alpha}{q-1}}\tx{b}^{q'}\dx+c\nonumber \\
   &\le& \sigma^{q'}\tx{c}_{*}(n,\Lambda,\tx{g},q,\alpha,\varepsilon)<\infty.
\end{eqnarray}
Next, with $m_{*}\ge 1$ being a number, we set 
\eqn{u0u0}
$$
u_{0}:=m_{*}u_{*}\in L^{\infty}(\overline{2Q})\cap C^{\infty}(\overline{2Q}\setminus \mathcal{C})\qquad \mbox{and}\qquad \ti{u}_{0}:=m_{*}\ti{u}\in C^{\infty}(\overline{2Q}).
$$
Now, given any $\kappa>0$, we can always find $\sigma_{*}\ge 1$, and eventually fix the value of $m_{*}\ge 1$ such that $\mathcal{G}(u_{0};Q)+\mathcal{G}^{*}(\sigma_{*}\tx{b};Q)\le \mathcal{G}(u_{0};2Q)+\mathcal{G}^{*}(\sigma_{*}\tx{b};2Q)<\kappa m_{*} \sigma_{*}$, that is \cite[Assumption 14]{balci2}. In fact, thanks to \eqref{6.4.3} and \eqref{6.4.1} with $\delta_{*}=0$ if $\tx{g}=1$ or $\delta_{*}=\min\{\varepsilon,\alpha\}/2$ if $\tx{g}$ verifies \eqref{ass.2}, we have
\begin{eqnarray}\label{6.5}
 \mathcal{G}(u_{0};Q)+\mathcal{G}^{*}(\sigma_{*}\tx{b};Q)&\le&\mathcal{G}(u_{0};2Q)+\mathcal{G}^{*}(\sigma_{*}\tx{b};2Q)\nonumber \\
 &\le&\sigma_{*}^{q'}\tx{c}_{*}+m_{*}^{1+\delta_{*}}\tx{c}^{*}\le\sigma_{*}^{\frac{1+\alpha}{\alpha}}\left(\sigma_{*}^{q'-\frac{1+\alpha}{\alpha}}\tx{c}_{*}\right)+m_{*}^{1+\alpha}\left(m_{*}^{\delta_{*}-\alpha}\tx{c}^{*}\right).
\end{eqnarray}
Set $\sigma_{*}=m_{*}^{\alpha}$, observe that by \eqref{66.0} it is $q'<(1+\alpha)/\alpha$, by definition it is always $\delta_{*}<\alpha$, and eventually pick $m_{*}\equiv m_{*}(n,\Lambda,\tx{g},q,\alpha,\varepsilon,\kappa)$ so large that
\eqn{6.5.1}
$$
m_{*}^{\alpha q'-(1+\alpha)}\tx{c}_{*}<\frac{\kappa}{2^{16n}\ti{\tx{c}}^{2}}\qquad \mbox{and}\qquad 2^{n}m_{*}^{\delta_{*}-\alpha}(\tx{c}^{*}+1)<\frac{\kappa}{2^{16n}\ti{\tx{c}}^{2}},
$$
where $\ti{\tx{c}}\equiv \ti{\tx{c}}(n,\Lambda,q,\alpha,\varepsilon)\ge 1$ is the constant appearing in \eqref{elel}. We can then complete estimate \eqref{6.5} as
\eqn{6.8}
$$
\mathcal{G}(u_{0};Q)+\mathcal{G}^{*}(\sigma_{*}\tx{b};Q)\le \mathcal{G}(u_{0};2Q)+\mathcal{G}^{*}(\sigma_{*}\tx{b};2Q)<\frac{\kappa m_{*} \sigma_{*}}{8\ti{\tx{c}}^{2}},
$$
 and \cite[Assumption 14]{balci2} is verified. In the light of \eqref{6.4.1}-\eqref{6.4.3}, we can repeat literally the same computations made in \eqref{6.5} to prove that there exist $m_{*},\sigma_{*}\ge 1$ so that
\eqn{6.7}
$$
\max\left\{\mathcal{H}(u_{0};Q)+\mathcal{H}^{*}(\sigma_{*}\tx{z};Q),\mathcal{H}(u_{0};2Q)+\mathcal{H}^{*}(\sigma_{*}\tx{z};2Q)\right\}<\frac{m_{*}\sigma_{*}}{2},
$$
just use \eqref{6.4}, recall that $\ti{\tx{c}}\ge 1$, and choose $m_{*}$ and $\sigma_{*}$ corresponding to $\kappa=1/2$ above, therefore also \cite[Assumption 29]{balci} is satisfied. Now we are ready to prove Theorems \ref{count}-\ref{count.3}.

\subsection{Proof of Theorem \ref{count}}
The proof of Theorem \ref{count} is split into two steps, where we offer examples of functionals as in \eqref{ll1} satisfying \eqref{ass.0}-\eqref{ass.2} for which Lavrentiev phenomenon occurs or density of smooth functions with respect to modular convergence fails. In principle, this is a direct application of the arguments in \cite{balci}. However, in \cite{balci} the validity of $\Delta_{2}$ and $\nabla_{2}$ conditions is assumed, while our functionals \eqref{ll}-\eqref{ll1} verify only $\Delta_{2}$ condition. In this respect, we show that the failure of $\nabla_{2}$ condition is compensated by the double phase structure of the integrand, thus fractal counterexamples can be constructed nonetheless.
\subsubsection*{Step 1:~occurrence of Lavrentiev phenomenon} Notice that, by \eqref{6.3}, \eqref{u0u0} it is
\eqn{trace}
$$\ti{u}_{0}=u_{0} \ \ \mbox{in} \ \ \overline{2Q}\setminus (-5/6,5/6)^{n} \qquad \mbox{and}\qquad \nr{\ti{u}_{0}}_{L^{\infty}(2Q)}=\nr{u_{0}}_{L^{\infty}(2Q)}
$$ 
and set
\eqn{6.11}
$$
\tx{I}_{\infty}:=\inf_{w\in \ti{u}_{0}+C^{\infty}_{c}(Q)}\mathcal{H}(w;Q)\qquad \mbox{and}\qquad \tx{I}_{1}:=\inf_{w\in \ti{u}_{0}+W^{1,1}_{0}(Q)}\mathcal{H}(w;Q).
$$
Since $\ti{u}_{0}\in C^{\infty}(\bar{Q})$ (recall \eqref{u0u0} and the properties of $\phi$), we have $\tx{I}_{1}\le \tx{I}_{\infty}<\infty,$ and for any given $0<\tau<\sigma_{*}m_{*}/8$ we can find $w_{\tau}\in \ti{u}_{0}+C^{\infty}_{c}(Q)$ such that 
\begin{eqnarray}\label{6.9}
\tx{I}_{\infty}&>&\mathcal{H}(w_{\tau};Q)-\tau\nonumber \\
&\ge&\sigma_{*}\int_{Q}\langle Dw_{\tau},\tx{z}\rangle\dx-\mathcal{H}^{*}(\sigma_{*}\tx{z};Q)-\tau\nonumber \\
&\stackrel{\eqref{6.7}}{>}&\sigma_{*}\int_{Q}\langle D(w_{\tau}-\ti{u}_{0}),\tx{z}\rangle\dx+\sigma_{*}\int_{Q}\langle D\ti{u}_{0},\tx{z}\rangle\dx+\mathcal{H}(u_{0};Q)-\frac{\sigma_{*}m_{*}}{2}-\tau\nonumber \\
&=&\sigma_{*}m_{*}\int_{Q}\langle D\ti{u},\tx{z}\rangle\dx+\mathcal{H}(u_{0};Q)-\frac{\sigma_{*}m_{*}}{2}-\tau\nonumber \\
&>&\frac{3\sigma_{*}m_{*}}{8}+\mathcal{H}(u_{0};Q)\stackrel{\eqref{trace}}{\ge} \frac{3}{8}+\tx{I}_{1}>\tx{I}_{1},
\end{eqnarray}
that is \eqref{count.2}, where we used $w_{\tau}-\ti{u}_{0}\in C^{\infty}_{c}(Q)$ and \cite[Proposition 18]{balci} to have $\int_{Q}\langle D(w_{\tau}-\ti{u}_{0}),\tx{z}\rangle\dx=0$, and \cite[Proposition 25 (a)]{balci}\footnote{In \cite[Propositions 18 and 25]{balci} functions $b$ and $u^{\partial}$ are our maps $\tx{z}$ and $\ti{u}$ respectively.} to secure $\int_{Q}\langle D\ti{u},\tx{z}\rangle\dx=1$. This proves the sharpness of \cite[Theorem 2.3]{buli}.
\subsubsection*{Step 2:~failure of density of smooth maps} Notice that under \eqref{66.0}, the strict inequality in \eqref{count.2} applies to modular $\mathcal{G}$ in \eqref{ggg} with $\tx{G}$ defined in \eqref{ggg} with $\tx{g}$ as in \eqref{ass.2} and $\ti{u}_{0}$ defined in \eqref{u0u0}, namely, \eqref{6.9} holds with $\mathcal{G}$ replacing $\mathcal{H}$. By direct methods and maximum principle \cite[Theorem 2.1]{ls05}, $\tx{I}_{1}$ defined in \eqref{6.11} above (with $\mathcal{G}$ instead of $\mathcal{H}$) is attained by some function $v\in (\ti{u}_{0}+W^{1,1}_{0}(Q))\cap L^{\infty}(Q)$. Let us show first that $v$ cannot be approximated in energy by smooth maps. Assume by contradiction that there is a sequence $\{v_{i}\}_{i\in \mathbb{N}}\subset \ti{u}_{0}+C^{\infty}_{c}(Q)$ such that $\mathcal{G}(v_{i}-v;Q)\to 0$. Then, up to (non relabelled) subsequences, $\mathcal{G}(v_{i};Q)\to \mathcal{G}(v;Q)$ and
\eqn{6.10}
$$
\tx{I}_{\infty}\le \mathcal{G}(v_{i};Q)\to \mathcal{G}(v;Q)=\tx{I}_{1} \ \Longrightarrow \ \tx{I}_{\infty}\le \tx{I}_{1},
$$
in contradiction with \eqref{6.9} - of course here we are taking $\tx{I}_{\infty}$ as in \eqref{6.11} with $\mathcal{G}$ replacing $\mathcal{H}$. Set $\bar{u}_{0}:=v-\ti{u}_{0}\in W^{1,1}_{0}(Q)\cap L^{\infty}(Q)$: if there were $\{u_{i}\}_{i\in \mathbb{N}}\subset C^{\infty}_{c}(Q)$ such that $\mathcal{G}(u_{i}-\bar{u}_{0};Q)\to0$, then $\{v_{i}\}_{i\in \mathbb{N}}:=\{u_{i}+\ti{u}_{0}\}_{i\in \N}\subset \ti{u}_{0}+C^{\infty}_{c}(Q)$ would satisfy $\mathcal{G}(v_{i}-v;Q)\to 0$ violating again \eqref{6.9}. This proves the sharpness of Lemma \ref{alav}, Corollary \ref{alav.2} and \cite[Theorem 2.3]{buli}. The proof is complete.
\subsection{Proof of Theorem \ref{count.3}}\label{pt4} For the ease of exposition, we deal separately with the nearly linear growing case $\tx{g}$ as in \eqref{ass.1}, and the linear growing one $\tx{g}=1$.
\subsubsection{Case 1: nearly linear growth}\label{nlg} In this part, $\tx{g}$ satisfies \eqref{ass.2}. With $\ti{u}_{0}\in C^{\infty}(\bar{Q})$ as in \eqref{u0u0}, we observe that Dirichlet class $\ti{u}_{0}+W^{1,1}_{0}(Q)$ is obviously nonempty and problem \eqref{count.1} admits a unique solution $u\in \ti{u}_{0}+W^{1,1}_{0}(Q)$ by superlinearity and direct methods. Denote by $\Sigma_{u}$ the set of non-Lebesgue points of $u$ in $Q$. Let us show that the requirements of \cite{balci2} are satisfied. By \eqref{trace}, function $u_{0}$ defined in \eqref{u0u0} is an admissible competitor for problem \eqref{count.1}, so by minimality it is $\mathcal{H}(u;Q)\le \mathcal{H}(u_{0};Q)$ and via \eqref{elel} we gain
\eqn{el.1}
$$
\mathcal{G}(u;Q)\le \ti{\tx{c}}\left(\mathcal{G}(u_{0};Q)+\snr{Q}\right),
$$
with the same constant $\ti{\tx{c}}\equiv \ti{\tx{c}}(n,\Lambda,q,\alpha,\varepsilon)\ge 1$ in \eqref{elel}. Key estimate \eqref{6.8} is already available, so we can bound via Young inequality (keep in mind that integrand $\tx{G}$ is radial),
\begin{eqnarray}\label{el.11}
\int_{Q}\snr{Du}\tx{b}\dx&\le& \sigma_{*}^{-1}\int_{Q}\tx{G}(x,\snr{Du})+\tx{G}^{*}(x,\sigma_{*}\tx{b})\dx\nonumber \\
&\stackrel{\eqref{el.1}}{\le}&\frac{\ti{\tx{c}}\snr{Q}}{\sigma_{*}}+\frac{\ti{\tx{c}}}{\sigma_{*}}\left(\mathcal{G}(u_{0};Q)+\mathcal{G}^{*}(\sigma_{*}\tx{b};Q)\right)\nonumber \\
&\stackrel{\eqref{6.8}}{\le}&\frac{\ti{\tx{c}}\snr{Q}}{\sigma_{*}}+\frac{\kappa m_{*}}{8\ti{\tx{c}}}\stackrel{\eqref{6.5.1}_{2}}{\le}\frac{3\kappa m_{*}}{16\ti{\tx{c}}}<\kappa m_{*},
\end{eqnarray}
which is precisely the content of \cite[Lemma 21]{balci2}, in turn granting quantitative control on a restricted Riesz type potential of $u$, \cite[Lemma 20]{balci2}. This is the main ingredient to assure that $\Sigma_{u}$ is almost as large as set $\mathcal{C}$ defined at the beginning of Section \ref{sec6}. In fact, a direct application of \cite[Theorem 25]{balci2} yields that we can fix an arbitrary number $N>3$, set $\kappa=(2N^{2})^{-1}$ so to consequently determine $m_{*}\equiv m_{*}(n,\Lambda,\tx{g},q,\alpha,\varepsilon,N)\ge 1$ such that
\eqn{el.111}
$$
\begin{cases}
\displaystyle
\ \int_{\mathcal{C}_{\varepsilon}^{n-1}}\snr{(u)_{+}(\bar{x})-(u)_{-}(\bar{x})}\d\gamma_{\varepsilon}^{n-1}(\bar{x})\ge m_{*}(1-N^{-1})\vspace{1.5mm}\\ \displaystyle
\ \gamma_{\varepsilon}^{n-1}\left(\left\{\bar{x}\in \mathcal{C}_{\varepsilon}^{n-1}\colon\snr{(u)_{+}(\bar{x})-(u)_{-}(\bar{x})}>m_{*}(1-N^{-1})\right\}\right)\ge 1-N^{-1},
\end{cases}
$$
where $(u)_{\pm}$ indicate the upper and lower traces of $u$ on $\mathcal{C}_{\varepsilon}^{n-1}$. The content of the two previous displays, \cite[Lemma 6 (a)]{balci} and Frostman Lemma eventually imply that $\mathcal{H}^{n-1-\varepsilon}(\Sigma_{u})\ge\mathcal{H}^{n-1-\varepsilon}(\mathcal{C}\cap\Sigma_{u})>0$, therefore $\dim_{\mathcal{H}}(\Sigma_{u})\ge n-1-\varepsilon$. This means that $u\not \in W^{1,p}_{\loc}(Q)$ for all $p>1+\varepsilon$. In fact, pick any $p_{0}\in \left(1+\varepsilon,\min\{n,p\}\right)$ - this possible thanks to the bounds on the size of $\varepsilon$. If $u\in W^{1,p}_{\loc}(Q)$, then $u\in W^{1,p_{0}}_{\loc}(Q)$ and $n-1-\varepsilon>n-p_{0}\ge \dim_{\mathcal{H}}(\Sigma_{u})$, a contradiction, and Theorems \ref{mt1}-\ref{mt2} are sharp.

\subsubsection{Case 2: linear growth}\label{lglg} Here, cubes $Q''\Subset Q'\Subset Q\Subset 2Q$ are as introduced at the beginning of Section \ref{sec6}, and coefficient $a\in C^{\alpha}(\overline{2Q})$ and boundary datum $\ti{u}_{0}\in C^{\infty}(\overline{2Q})$ are defined in \eqref{6.3} and \eqref{u0u0}, respectively. However, now integrand $\tx{L}$ in \eqref{ll1} features linear growth \eqref{gcgc}, so problem \eqref{count.1} must be understood in the sense of relaxation. In this respect, before proceeding further, let us quickly outline the main properties of the Lebesgue-Serrin-Marcellini (LSM) extension of integral $\mathcal{H}$ in \eqref{ll1}. \\\\
\emph{The LSM extension with solid boundary values.} Relaxing functional \eqref{ll1} is needed to secure the existence of (a suitably generalized notion of) solutions to problem \eqref{count.1} in $BV(Q)$, the space of $L^{1}$-functions with finite total variation on $Q$, see \cite{BS1,gkpq} for basic definitions and main properties.\footnote{When working with $BV$-maps $w$, it is customary to adopt the decomposition of the gradient measure $Dw=\nabla w \mathcal{L}^{n}+D^{s}w$ into its absolutely continuous and its singular part with respect to the $n$-dimensional Lebesgue measure $\mathcal{L}^{n}$ - this notation will always be in force here. Needless to say, for $W^{1,1}$-functions $w$ it is $Dw=\nabla w$.} Following \cite{gms79a,gkpq,jw}, we introduce the LSM extension of integral $\mathcal{H}$ in \eqref{ll1}, driven by integrand $\tx{H}$ subject to \eqref{ass.0} and \eqref{gcgc} with $0\le a(\cdot)\in C^{\alpha}(\overline{2Q})$ for some $\alpha>0$. With $\ti{u}_{0}\in C^{\infty}(\overline{2Q})$, set
\begin{flalign}\label{a0a0}
\mathcal{A}_{\ti{u}_{0}}:=&\left\{w\in W^{1,1}(2Q) \cap L^{\infty}(2Q)\colon\tx{G}(\cdot,\nabla w)\in L^{1}(2Q),\ \nr{w}_{L^{\infty}(Q)}\le \nr{\ti{u}_{0}}_{L^{\infty}(2Q)},\  w=\ti{u}_{0} \ \mbox{in} \ 2Q\setminus \bar{Q}\right\},
\end{flalign}
notice that $\mathcal{A}_{\ti{u}_{0}}\not =\emptyset$ as $\ti{u}_{0},u_{0}\in \mathcal{A}_{\ti{u}_{0}}$, cf. \eqref{6.4.1} and \eqref{trace}. With $w\in BV(Q)$, denote by $\mathcal{w}\in BV(2Q)$ the gluing
\eqn{w0w0}
$$
\mathcal{w}:=\begin{cases}
\displaystyle
    \ w\quad &\mbox{on} \ \ Q\vspace{1.5mm}\\
    \displaystyle
    \ \ti{u}_{0}\quad &\mbox{on} \ \ 2Q\setminus \bar{Q},
\end{cases}
$$
which extends\footnote{From now on given any function $w\in BV(Q)$, its calligraphic version $\mathcal{w}$ will denote its extension by $\ti{u}_{0}$ in $2Q\setminus \bar{Q}$ as defined in \eqref{w0w0}.} $w$ by $\ti{u}_{0}$ in $2Q\setminus \bar{Q}$. The LSM extension of $\mathcal{H}$ for solid boundary values $\ti{u}_{0}$ is then defined as
\eqn{def1}
$$
\bar{\mathcal{H}}(\mathcal{w};Q,2Q):=\inf\left\{\liminf_{i\to \infty}\mathcal{H}(\ti{w}_{i};2Q)\colon \{\ti{w}_{i}\}_{i\in \mathbb{N}}\subset \mathcal{A}_{\ti{u}_{0}} \ \mbox{and} \ \ti{w}_{i}\to \mathcal{w} \ \mbox{strongly in} \ L^{1}(2Q)\right\},
$$
cf. \cite[Sections 1.2 and 3.2.2]{gkpq}. Basic density results \cite[Lemma 3.2]{gkpq} and truncation arguments \cite[Section 6]{jw} guarantee that the class of competitors in \eqref{def1} is nonempty for bounded functions whose $L^{\infty}$-norm is controlled by $\nr{\ti{u}_{0}}_{L^{\infty}(2Q)}$. To obtain a functional solely defined on $w\in BV(Q)$, we simply set
\eqn{def2}
$$
\bar{\mathcal{H}}^{*}(w;Q):=\bar{\mathcal{H}}(\mathcal{w};Q,2Q)-\mathcal{H}(\ti{u}_{0};2Q\setminus \bar{Q}),
$$
where $\mathcal{w}$ is defined as in \eqref{w0w0}. By relaxed minimizer of problem \eqref{count.1} we mean $u\in BV(Q)$ such that 
\eqn{relmin}
$$
\bar{\mathcal{H}}^{*}(u;Q)<\infty\quad \mbox{and}\quad \bar{\mathcal{H}}^{*}(u;Q)\le \bar{\mathcal{H}}^{*}(w;Q) \ \ \mbox{for all} \ \ w\in BV(Q).
$$
The main properties of extension \eqref{def1}-\eqref{def2} are in the next proposition.
\begin{proposition}\label{p71}
    Within the setting described above, the following holds.
    \begin{itemize}
         \item[(\emph{i.})] $\bar{\mathcal{H}}^{*}(v;Q)=\mathcal{H}(v;Q)$ for all $v\in \mathcal{A}_{\ti{u}_{0}}$;
        \item[(\emph{ii.})] There exists a minimizer $u\in BV(Q)\cap L^{\infty}(Q)$ in the sense of \eqref{relmin} such that $\nr{u}_{L^{\infty}(Q)}\le \nr{\ti{u}_{0}}_{L^{\infty}(2Q)}$;
        \item[(\emph{iii.})] If $v\in BV(Q)$ satisfies $\bar{\mathcal{H}}^{*}(v;Q)<\infty$, then $\nr{v}_{L^{\infty}(Q)}\le \nr{\ti{u}_{0}}_{L^{\infty}(2Q)}$.
        \item[(\emph{iv.})] If $v\in W^{1,1}(Q)$ satisfies $\bar{\mathcal{H}}^{*}(v;Q)<\infty$, then $\tx{H}(\cdot,\nabla v)\in L^{1}(Q)$ and
        \eqn{iv}
       $$
      \  \mathcal{G}(v;Q)\le \Lambda \bar{\mathcal{H}}^{*}(v;Q)+\Lambda \mathcal{H}(\ti{u}_{0};2Q\setminus \bar{Q})+2^{4n}\Lambda^{2}.
       $$
    \end{itemize}
\end{proposition}
\begin{proof}
Claim ($\emph{i}.$) is a direct consequence of definitions \eqref{def1}-\eqref{def2}. The proof of the existence part of ($\emph{ii}.$) follows exactly as in \cite[Lemma 6.7 and Proposition 6.8]{gkpq} - actually in a much more elementary way as integrand $\tx{H}$ is strictly convex and coercive in $W^{1,1}$. We then focus on ($\emph{iii}.$). Let $\{\ti{v}_{i}\}_{i\in \mathbb{N}}\subset \mathcal{A}_{\ti{u}_{0}}$ be any sequence such that $\ti{v}_{i}\to \mathcal{v}$ in $L^{1}(2Q)$, and introduce the truncation operator $\mathbb{R}\ni s\mapsto\tx{T}(s):=\min\{\nr{\ti{u}_{0}}_{L^{\infty}(2Q)},\max\{s,-\nr{\ti{u}_{0}}_{L^{\infty}(2Q)}\}\}.$ By \eqref{a0a0} it is $\tx{T}(\ti{v}_{i})=\ti{v}_{i}$ for all $i\in \mathbb{N}$, thus, by the 1-Lipschitz character of truncations it is $\tx{T}(v)\in BV(Q)$ and $\nr{\tx{T}(\mathcal{v})-\mathcal{v}}_{L^{1}(2Q)}\le \nr{\tx{T}(\mathcal{v})-\tx{T}(\ti{v}_{i})}_{L^{1}(2Q)}+\nr{\ti{v}_{i}-\mathcal{v}}_{L^{1}(2Q)}\le 2\nr{\ti{v}_{i}-\mathcal{v}}_{L^{1}(2Q)}\to 0$. This implies also the $L^{\infty}$-bound in (\emph{ii}.) via $\eqref{relmin}_{1}$. Concerning (\emph{iv}.), let $\{\ti{v}_{i}\}_{i\in \mathbb{N}}\subset \mathcal{A}_{\ti{u}_{0}}$ be any sequence such that $\ti{v}_{i}\to \mathcal{v}$ in $L^{1}(2Q)$. Keeping in mind that $\mathcal{v}=v$ on $Q$, cf. \eqref{w0w0} and that $v\in W^{1,1}(Q)$, we can apply Serrin's semicontinuity result \cite[Section 4.5]{giu} to derive
\begin{eqnarray*}
\liminf_{i\to \infty}\int_{2Q}\tx{H}(x,\nabla \ti{v}_{i})\dx&\ge& \liminf_{i\to \infty}\int_{2Q}\tx{H}(x,\nabla \ti{v}_{i})+\Lambda\dx-2^{2n}\Lambda\nonumber \\
&\stackrel{\eqref{ass.1}_{1}}{\ge}&\liminf_{i\to \infty}\int_{Q}\tx{H}(x,\nabla \ti{v}_{i})+\Lambda\dx-2^{2n}\Lambda\nonumber \\
&\ge&\int_{Q}\tx{H}(x,\nabla v)\dx-2^{2n}\Lambda\stackrel{\eqref{ass.1}_{1}}{\ge} \frac{1}{\Lambda}\int_{Q}\tx{G}(x,\snr{\nabla v})\dx-2^{4n}\Lambda.
\end{eqnarray*}
The arbitrariness of sequence $\{\ti{v}_{i}\}_{i\in \mathbb{N}}$ and $\eqref{ass.1}_{1}$ yield that $\tx{H}(\cdot,\nabla v)\in L^{1}(Q)$ and \eqref{iv} holds after subtracting $\mathcal{H}(\ti{u}_{0};2Q\setminus \bar{Q})$ on both sides of the above inequality.
\end{proof}
\begin{remark}
    \emph{Functional \eqref{def1}-\eqref{def2} coincides with the usual LSM extension with solid boundary conditions \cite[Section 1.2]{gkpq} on all functions $w\in BV(Q)\cap L^{\infty}(Q)$ such that $\nr{w}_{L^{\infty}(Q)}\le \nr{\ti{u}_{0}}_{L^{\infty}(2Q)}$. In fact, after subtracting an affine function in $z$ (a null Lagrangian), one may assume that $\tx{H}(x,z)\ge \tx{H}(x,0)$ for all $(x,z)\in 2Q\times \mathbb{R}^{n}$, so truncations decrease the energy. It follows that every minimizer admits a bounded representative whose $L^{\infty}$-norm is controlled in terms of $\nr{\ti{u}_{0}}_{L^{\infty}(2Q)}$, \cite[Theorem D.2 and Corollary D.3]{BS1}; if the extension is strictly convex, minima coincide up to an additive constant. Adding/subtracting a null Lagrangian merely shifts $\bar{\mathcal{H}}^{*}$ by a constant, hence does not affect minimizers or lower semicontinuity. In this sense, the relaxation \eqref{def1}-\eqref{def2} works as a selection principle for minima bounded in terms of $\nr{\ti{u}_{0}}_{L^{\infty}(2Q)}$.
    }
\end{remark}

\noindent Now we are ready to complete the proof of Theorem \ref{count.3}.\\\\
\emph{Irregular minima of relaxed area-double phase problems.}
Let $\bar{\mathcal{H}}^{*}(\cdot;Q)$ be the LSM extension of integral $\mathcal{H}$ in \eqref{ll1}. By Proposition \ref{p71} ($\emph{ii}$.), there exists a relaxed minimizer $u\in BV(Q)\cap L^{\infty}(Q)$ of problem \eqref{count.1}. In particular, it is
\eqn{71.8}
$$
\nr{u}_{L^{\infty}(Q)}\le \nr{\ti{u}_{0}}_{L^{\infty}(2Q)}\stackrel{\eqref{u0u0}}{\le}\frac{m_{*}}{2}.
$$
If $u\in BV(Q)\setminus W^{1,1}_{\loc}(Q)$ there is nothing to prove, and we can assume that $u\in BV(Q)\cap W^{1,1}_{\loc}(Q)$ so that, by inner regularity, it is $u\in W^{1,1}(Q)$. This last information together with the finiteness $\bar{\mathcal{H}}^{*}(u;Q)<\infty$ granted by minimality allows applying Proposition \ref{p71} (\emph{iv}.) to deduce that $\tx{H}(\cdot,\nabla u)\in L^{1}(Q)$ and estimate \eqref{iv} holds with $v=u$. As the attainment of traces may fail, we need a slightly more refined argument than the one in Section \ref{nlg} to achieve key estimates \eqref{el.11}-\eqref{el.111}. The minimality of $u$ and Proposition \ref{p71} ($\emph{iv}.$) give 
\begin{eqnarray}\label{71.9}
\mathcal{G}(u;Q)&\stackrel{\eqref{iv}}{\le}& \Lambda\bar{\mathcal{H}}^{*}(u;Q)+\Lambda\mathcal{H}(\ti{u}_{0};2Q\setminus \bar{Q})+2^{4n}\Lambda^{2}\le \Lambda\bar{\mathcal{H}}^{*}(u_{0};Q)+\Lambda \mathcal{H}(\ti{u}_{0};2Q\setminus \bar{Q})+2^{4n}\Lambda^{2}\nonumber \\
&=&\Lambda\mathcal{H}(u_{0};Q)+\Lambda \mathcal{H}(\ti{u}_{0};2Q\setminus \bar{Q})+2^{4n}\Lambda^{2}\nonumber \\
&\stackrel{\eqref{trace}}{=}&\Lambda\mathcal{H}(u_{0};2Q)+2^{4n}\Lambda^{2}\stackrel{\eqref{ass.1}_{1}}{\le}2^{8n+q+\alpha}\Lambda^{2}\left(\mathcal{G}(u_{0};2Q)+\snr{2Q}\right).
\end{eqnarray}
Next, fix $j\ge 1$. By a diagonal argument, \eqref{a0a0}, \eqref{def1} and \eqref{71.8} there is $\mathcal{u}_{j}\in \mathcal{A}_{\ti{u}_{0}}$ such that 
\eqn{8.1}
$$
\begin{array}{c}
\displaystyle
\nr{\mathcal{u}_{j}-\mathcal{u}}_{L^{1}(2Q)}\le j^{-1},\qquad\nr{\mathcal{u}_{j}}_{L^{\infty}(2Q)}\le \nr{\ti{u}_{0}}_{L^{\infty}(2Q)}\\ [8pt]\displaystyle
\mathcal{H}(\mathcal{u}_{j};2Q)\le \bar{\mathcal{H}}^{*}(u;Q)+\mathcal{H}(\ti{u}_{0};2Q\setminus \bar{Q})+j^{-1},
\end{array}
$$
keep in mind \eqref{w0w0}.
Next, let $\phi_{0}\in C^{1}_{c}(Q)$ be a cut-off function such that $\mathds{1}_{Q''}\le \phi_{0}\le \mathds{1}_{Q'}$, $\nr{\nabla \phi_{0}}_{L^{\infty}(Q)}\le 2^{10}n$. For reasons that will be clear in a few lines, let $\kappa>0$ be any number, $\sigma_{*}\equiv \sigma_{*}(n,\Lambda,q,\alpha,\varepsilon,\kappa)$ be as in \eqref{6.5.1}-\eqref{6.8}, and via Young inequality bound
\begin{eqnarray}\label{8.2}
\tx{B}_{j}&:=&\int_{Q}\tx{b}\left(\snr{\nabla \phi_{0}}\snr{u-\mathcal{u}_{j}}+\phi_{0}\snr{\nabla u-\nabla \mathcal{u}_{j}}+\snr{\nabla \mathcal{u}_{j}}\right)\dx\nonumber \\
&\le&\frac{2^{20q}n^{q}}{\sigma_{*}}\int_{Q}[\tx{G}(x,\snr{\nabla u})+\tx{G}^{*}(x,\sigma_{*}\tx{b})]\dx+\frac{2^{20q}n^{q}}{\sigma_{*}}\int_{Q}\tx{G}(x,\snr{\mathcal{u}-\mathcal{u}_{j}})+1\dx\nonumber \\
&&+\frac{2^{20q}n^{q}}{\sigma_{*}}\int_{Q}\tx{G}(x,\snr{\nabla \mathcal{u}_{j}})\dx=:\mbox{(I)}+\mbox{(II)}_{j}+\mbox{(III)}_{j}\stackrel{\eqref{a0a0},\eqref{iv}}{<}\infty.
\end{eqnarray}
Concerning $\mbox{(I)}$, we bound
\begin{flalign*}
\mbox{(I)}\stackrel{\eqref{71.9}}{\le}\frac{2^{30nq+\alpha}\Lambda^{2}n^{q}}{\sigma_{*}}\left(\mathcal{G}(u_{0};2Q)+\mathcal{G}^{*}(\sigma_{*}\tx{b};Q)+\snr{2Q}\right).
\end{flalign*}
Moreover, recalling \eqref{71.8} and \eqref{8.1} it is $\tx{G}(\cdot,\snr{\mathcal{u}-\mathcal{u}_{j}})\in L^{\infty}(Q)$, so we obtain 
\begin{eqnarray*}
\mbox{(II)}_{j}&\le&  \frac{2^{20q}n^{q}}{\sigma_{*}} \left(2^{4nq+\alpha}m_{*}^{q-1}\nr{\mathcal{u}-\mathcal{u}_{j}}_{L^{1}(Q)}+\snr{Q}\right)\nonumber \\
&\le& \frac{2^{24nq+\alpha}n^{q}}{\sigma_{*}}\left(\frac{m_{*}^{q-1}}{j}+1\right).
\end{eqnarray*}
Finally, for term $\mbox{(III)}_{j}$ we have
\begin{eqnarray*}
\mbox{(III)}_{j}&\stackrel{\eqref{ass.1}_{1}}{\le}&\frac{2^{20q}\Lambda n^{q}}{\sigma_{*}}\mathcal{H}(\mathcal{u}_{j};Q)+\frac{2^{20q}\Lambda^{2} n^{q}\snr{Q}}{\sigma_{*}}\nonumber \\
&\stackrel{\eqref{8.1}}{\le}&\frac{2^{20q}\Lambda n^{q}}{\sigma_{*}}\left(\bar{\mathcal{H}}^{*}(u;Q)+\mathcal{H}(\ti{u}_{0};2Q\setminus \bar{Q})+\Lambda\snr{2Q}+j^{-1}\right)+\frac{2^{20q}\Lambda^{2} n^{q} }{\sigma_{*}}\nonumber \\
&\stackrel{\eqref{ass.1}_{1},\eqref{trace}}{\le}&\frac{2^{28nq+\alpha}\Lambda^{2} n^{q}}{\sigma_{*}}\left(\mathcal{G}(u_{0};2Q)+\snr{2Q}\right)+\frac{2^{20q}\Lambda^{2} n^{q} (1+j^{-1})}{\sigma_{*}}.
\end{eqnarray*}
Merging the content of the three above displays, we end up with
\eqn{8.3}
$$
\tx{B}_{j}\le  \frac{\ti{\tx{c}}}{\sigma_{*}c_{n;\varepsilon}}\left(\mathcal{G}(u_{0};2Q)+\mathcal{G}^{*}(\sigma_{*}\tx{b};2Q)+\snr{2Q}+ m^{q-1}_{*}j^{-1}\right),
$$
where $\ti{\tx{c}}\equiv \ti{\tx{c}}(n,\Lambda,q,\alpha,\varepsilon)$ is the same constant immediately below \eqref{elel} and $j\ge 1$ is arbitrary. The finiteness of terms $\mbox{(I)}$-$\mbox{(III)}_{j}$ in \eqref{8.2} in turn grants the existence of upper and lower traces of $u$ and $\mathcal{u}_{j}$ on $Q''\cap \{x_{n}=0\}$, cf. \cite[Lemma 20 and Corollary 24]{balci2}, so to conclude, we only need to derive a "localized" counterpart of the oscillation estimates in \cite[Section 4.4]{balci2}. In this respect, from \eqref{6.3} and \eqref{u0u0} we see that in $\{\snr{\bar{x}}_{\infty}<\snr{x_{n}},\  \pm x_{n}>1\}$ it is $\ti{u}_{0}(x)=2^{-1}\textnormal{sgn}\{x_{n}\}m_{*}$, so we extend $\mathcal{u}_{j}=\ti{u}_{0}$ in $\{\snr{\bar{x}}_{\infty}<\snr{x_{n}}, \ \pm x_{n}>1\}$ and, recalling that $\phi_{0}=0$ in $\mathbb{R}^{n}\setminus Q'$, we extend $\phi_{0}(u-\mathcal{u}_{j})=0$ in $\mathbb{R}^{n}\setminus Q'$. For $\gamma$-a.e. $(\bar{x},0)\in \mathcal{C}\Subset Q''$ we then control, using the traces estimates in \cite[Section 4.3]{balci2} in terms of restricted Riesz type potentials $\textnormal{I}^{\pm}_{1}(\ \cdot \ \mathds{1}_{Q})(\bar{x},0)$ defined in \cite[Section 4.1]{balci2},\footnote{Here, every expression containing $\pm$ is summed over both signs.}
\begin{flalign}\label{8.5}
&\snr{(\mathcal{u}_{j})_{\pm}(\bar{x})\mp m_{*}/2}+\snr{(\phi_{0}u)_{\pm}(\bar{x})-(\phi_{0}\mathcal{u}_{j})_{\pm}(\bar{x})}\nonumber \\
&\qquad \qquad \quad \le c(n)\textnormal{I}^{\pm}_{1}(\snr{\nabla \mathcal{u}_{j}}\mathds{1}_{Q})(\bar{x},0)+c(n)\textnormal{I}^{\pm}_{1}(\snr{\nabla(\phi_{0}(\mathcal{u}- \mathcal{u}_{j}))}\mathds{1}_{Q})(\bar{x},0),
\end{flalign}
therefore, by \cite[Lemma 20]{balci2} and \eqref{8.2} we obtain
\begin{flalign}\label{8.4}
   & \int_{\mathcal{C}_{\varepsilon}^{n-1}}\left[\snr{(\mathcal{u}_{j})_{\pm}(\bar{x})\mp m_{*}/2}+\snr{(\phi_{0}u)_{\pm}(\bar{x})-(\phi_{0}\mathcal{u}_{j})_{\pm}(\bar{x})}\right]\d\gamma_{\varepsilon}^{n-1}(\bar{x})\nonumber \\
   &\qquad  \quad \le c\int_{Q}\left[\textnormal{I}^{\pm}_{1}(\snr{\nabla \mathcal{u}_{j}}\mathds{1}_{Q})(\bar{x},0)+\textnormal{I}^{\pm}_{1}(\snr{\nabla(\phi_{0}(\mathcal{u}- \mathcal{u}_{j}))}\mathds{1}_{Q})(\bar{x},0)\right]\d\gamma(x)\le c_{n;\varepsilon}\tx{B}_{j},
\end{flalign}
where $c_{n;\varepsilon}\equiv c_{n;\varepsilon}(n,\varepsilon)$ is the product between constant $c(n)$ from \eqref{8.5} and the one appearing in \cite[Lemma 20]{balci2}, see also below \eqref{elel}. Keeping in mind that $\mathcal{C}\Subset Q''$ and $\phi_{0}\ge \mathds{1}_{Q''}$, by triangular inequality we bound
\begin{eqnarray}\label{8.6}
m_{*}&=&\left|\frac{m_{*}}{2}-\left(-\frac{m_{*}}{2}\right)\right|\nonumber \\
&\le& \snr{(\mathcal{u}_{j})_{\pm}(\bar{x})\mp m_{*}/2}+\snr{(u)_{\pm}(\bar{x})-(\mathcal{u}_{j})_{\pm}(\bar{x})}+\snr{(u)_{+}(\bar{x})-(u)_{-}(\bar{x})}\nonumber \\
&=&\snr{(\mathcal{u}_{j})_{\pm}(\bar{x})\mp m_{*}/2}+\snr{(\phi_{0}u)_{\pm}(\bar{x})-(\phi_{0}\mathcal{u}_{j})_{\pm}(\bar{x})}+\snr{(u)_{+}(\bar{x})-(u)_{-}(\bar{x})}\nonumber \\
&\stackrel{\eqref{8.5}}{\le}&\snr{(u)_{+}(\bar{x})-(u)_{-}(\bar{x})}+c\textnormal{I}^{\pm}_{1}(\snr{\nabla \mathcal{u}_{j}}\mathds{1}_{Q})(\bar{x},0)+c\textnormal{I}^{\pm}_{1}(\snr{\nabla(\phi_{0}(\mathcal{u}- \mathcal{u}_{j}))}\mathds{1}_{Q})(\bar{x},0),
\end{eqnarray}
for $c\equiv c(n)$. Integrating both sides of \eqref{8.6} we obtain 
\begin{eqnarray*}
    m_{*}&\le& \int_{\mathcal{C}_{\varepsilon}^{n-1}}\snr{(u)_{+}(\bar{x})-(u)_{-}(\bar{x})}\d\gamma_{\varepsilon}^{n-1}(\bar{x})\nonumber \\
    &&+c\int_{\mathcal{C}_{\varepsilon}^{n-1}}\textnormal{I}^{\pm}_{1}(\snr{\nabla \mathcal{u}_{j}}\mathds{1}_{Q})(\bar{x},0)+\textnormal{I}^{\pm}_{1}(\snr{\nabla(\phi_{0}(\mathcal{u}- \mathcal{u}_{j}))}\mathds{1}_{Q})(\bar{x},0)\d\gamma_{\varepsilon}^{n-1}(\bar{x})\nonumber \\
    &\le &\int_{\mathcal{C}_{\varepsilon}^{n-1}}\snr{(u)_{+}(\bar{x})-(u)_{-}(\bar{x})}\d\gamma_{\varepsilon}^{n-1}(\bar{x})\nonumber \\
    &&+c\int_{Q}\textnormal{I}^{\pm}_{1}(\snr{\nabla \mathcal{u}_{j}}\mathds{1}_{Q})(\bar{x},0)+\textnormal{I}^{\pm}_{1}(\snr{\nabla(\phi_{0}(\mathcal{u}- \mathcal{u}_{j}))}\mathds{1}_{Q})(\bar{x},0)\d\gamma(x)\nonumber \\
    &\stackrel{\eqref{8.4}}{\le}&\int_{\mathcal{C}_{\varepsilon}^{n-1}}\snr{(u)_{+}(\bar{x})-(u)_{-}(\bar{x})}\d\gamma_{\varepsilon}^{n-1}(\bar{x})+c_{n;\varepsilon}\tx{B}_{j}\nonumber \\
    &\stackrel{\eqref{8.3}}{\le}&\int_{\mathcal{C}_{\varepsilon}^{n-1}}\snr{(u)_{+}(\bar{x})-(u)_{-}(\bar{x})}\d\gamma_{\varepsilon}^{n-1}(\bar{x})\nonumber \\
    &&+\frac{\ti{\tx{c}}}{\sigma_{*}}\left(\mathcal{G}(u_{0};2Q)+\mathcal{G}^{*}(\sigma_{*}\tx{b};2Q)+\snr{2Q}+ m^{q-1}_{*}j^{-1}\right)\nonumber \\
    &\stackrel{\eqref{6.5.1},\eqref{6.8}}{\le}&\int_{\mathcal{C}_{\varepsilon}^{n-1}}\snr{(u)_{+}(\bar{x})-(u)_{-}(\bar{x})}\d\gamma_{\varepsilon}^{n-1}(\bar{x})+\frac{3\kappa m_{*}}{16\ti{\tx{c}}}+\ti{\tx{c}} m_{*}^{q-1-\alpha}j^{-1}.
\end{eqnarray*}
Letting $j\to \infty$, we obtain
$$
m_{*}\le \int_{\mathcal{C}_{\varepsilon}^{n-1}}\snr{(u)_{+}(\bar{x})-(u)_{-}(\bar{x})}\d\gamma_{\varepsilon}^{n-1}(\bar{x})+\frac{3\kappa m_{*}}{16}.
$$
Fix an arbitrary $N>3$ and set $\kk=(2N^{2})^{-1}$ to get
\eqn{8.7}
$$
\int_{\mathcal{C}_{\varepsilon}^{n-1}}\snr{(u)_{+}(\bar{x})-(u)_{-}(\bar{x})}\d\mathcal{\gamma_{\varepsilon}^{n-1}}(\bar{x})\ge m_{*}\left(1-\frac{1}{2N^{2}}\right)> m_{*}(1-N^{-1}),
$$
that is $\eqref{el.111}_{1}$. Moreover, by \eqref{71.8} it is $\snr{(u)_{+}(\bar{x})-(u)_{-}(\bar{x})}\le m_{*}$ so let $\tau_{*}:=1-(2N)^{-1}$ and bound
\begin{eqnarray}\label{8.7.1}
m_{*}\left(1-\frac{1}{2N^{2}}\right)&\stackrel{\eqref{8.7}}{\le}&\int_{\mathcal{C}_{\varepsilon}^{n-1}}\snr{(u)_{+}(\bar{x})-(u)_{-}(\bar{x})}\d\gamma_{\varepsilon}^{n-1}(\bar{x})\nonumber \\
&\le& m_{*}\gamma_{\varepsilon}^{n-1}\left(\left\{\bar{x}\in \mathcal{C}_{\varepsilon}^{n-1}\colon \snr{(u)_{+}(\bar{x})-(u)_{-}(\bar{x})}>m_{*}\tau_{*}\right\}\right)\nonumber \\
&&+m_{*}\tau_{*}\left(1-\gamma_{\varepsilon}^{n-1}\left(\left\{\bar{x}\in \mathcal{C}_{\varepsilon}^{n-1}\colon \snr{(u)_{+}(\bar{x})-(u)_{-}(\bar{x})}> m_{*}\tau_{*}\right\}\right)\right)\nonumber \\
&\le&m_{*}(1-\tau_{*})\gamma_{\varepsilon}^{n-1}\left(\left\{\bar{x}\in \mathcal{C}_{\varepsilon}^{n-1}\colon \snr{(u)_{+}(\bar{x})-(u)_{-}(\bar{x})}> m_{*}\tau_{*}\right\}\right)+m_{*}\tau_{*}.
\end{eqnarray}
As $\tau_{*}>1-N^{-1}$, estimate \eqref{8.7.1} yields
$$
1-N^{-1}\le \gamma_{\varepsilon}^{n-1}\left(\left\{\bar{x}\in \mathcal{C}_{\varepsilon}^{n-1}\colon \snr{(u)_{+}(\bar{x})-(u)_{-}(\bar{x})}>m_{*}(1-N^{-1}) \right\}\right),
$$
for all $N>3$, thus also $\eqref{el.111}$ is secured and we can proceed as in Section \ref{nlg}. To conclude, we notice that the above construction does not rely on the explicit identification of a specific minimizer, but on energy estimates forcing an energy reduction on minima by comparison with the irregular (yet admissible, \eqref{6.4.1}) competitor $u_{0}$ in \eqref{u0u0}. By minimality, $\eqref{relmin}_{1}$ and Proposition \ref{p71} (\emph{iii}.) this procedure is applicable to all minimizers of $\bar{\mathcal{H}}^{*}$, and the proof is complete.
\begin{remark}\label{rem6} \emph{The counterexamples in Theorems \ref{count}-\ref{count.3} cover the range $q>1+\alpha$, while Theorems \ref{mt1}-\ref{mt2} hold whenever $q<1+\alpha$. In the limiting case $q=1+\alpha$ the nonoccurrence of Lavrentiev phenomenon and the density of smooth maps in modular convergence can still be secured, cf. Lemma \ref{alav}, Corollary \ref{alav.2} and \cite[Theorem 2.3]{buli}, however to prove regularity in such a borderline configuration some 0-order information such as H\"older continuity à la De Giorgi-Nash-Moser is needed, see \cite{CM2,BCM,HO3,bb90}, that is not available here due to the strong rate of nonuniformity of integrals \eqref{ll}-\eqref{ll1}.}
\end{remark}

\end{document}